\documentclass[aos,noinfoline,preprint]{imsart}
\usepackage{amsmath,amsthm,amsfonts,graphicx,amscd,amssymb}
\usepackage[psamsfonts]{amssymb}
\usepackage[english]{babel}
\usepackage{sectsty}

\RequirePackage[OT1]{fontenc}
\usepackage[utf8]{inputenc}

\usepackage{mathtools}
\RequirePackage{bbm}
\usepackage{hyperref}
\usepackage{verbatim}
\setattribute{journal}{name}{}

\usepackage{color}
\usepackage{mathrsfs}
\usepackage{hyperref}

\usepackage{mathtools}

\usepackage{caption}

\usepackage{stmaryrd}

\usepackage{multibib}

\usepackage{multicol}

\usepackage{enumerate}

\usepackage{subfig}
\usepackage{float}

\usepackage{color}

\usepackage{relsize}

\usepackage{tikz-cd}

\startlocaldefs

\newcommand{\R}{\mathbb{R}}

\newcommand{\E}{\mathbb{E}}

\newcommand{\X}{\mathbb{X}}
\newcommand{\Y}{\mathbb{Y}}

\def\restriction#1#2{\mathchoice
              {\setbox1\hbox{${\displaystyle #1}_{\scriptstyle #2}$}
              \restrictionaux{#1}{#2}}
              {\setbox1\hbox{${\textstyle #1}_{\scriptstyle #2}$}
              \restrictionaux{#1}{#2}}
              {\setbox1\hbox{${\scriptstyle #1}_{\scriptscriptstyle #2}$}
              \restrictionaux{#1}{#2}}
              {\setbox1\hbox{${\scriptscriptstyle #1}_{\scriptscriptstyle #2}$}
              \restrictionaux{#1}{#2}}}
\def\restrictionaux#1#2{{#1\,\smash{\vrule height .8\ht1 depth .85\dp1}}_{\,#2}}

\newcommand{\norm}[1]{\left\Vert #1\right\Vert}
\newcommand{\tens}[2]{#1^{\otimes #2}}

\DeclareMathOperator{\image}{im}

\newtheorem{thm}{Theorem}
\newtheorem{lem}{Lemma}
\newtheorem{Def}{Definition}
\newtheorem{prop}{Proposition}

\newtheorem{Athm}{Theorem}[section]
\newtheorem{Alem}[Athm]{Lemma}

\newtheorem{Aprop}[Athm]{Proposition}

\endlocaldefs

\begin{document}

\begin{frontmatter}

\title{Non-asymptotic Rates for Manifold, Tangent Space and Curvature Estimation}
\runtitle{Convergence Rates for Manifold Learning}

\begin{aug}

\author{\fnms{Eddie} \snm{Aamari}\thanksref{m1,t1,t2,t3}\ead[label=e1]{eaamari@ucsd.edu}}
and
\author{\fnms{Cl\'ement} \snm{Levrard}\thanksref{m3,t1,t2}\ead[label=e2]{levrard@math.univ-paris-diderot.fr}}


\ead[label=u1,url]{http://www.math.ucsd.edu/~eaamari/}
\ead[label=u2,url]{http://www.normalesup.org/~levrard/}

\thankstext{t1}{Research supported by ANR project TopData ANR-13-BS01-0008}
\thankstext{t2}{Research supported by Advanced Grant of the European Research Council GUDHI}
\thankstext{t3}{Supported by the Conseil r\'egional d'\^{I}le-de-France program RDM-IdF}
\runauthor{E. Aamari and C. Levrard}

\affiliation{
	U.C. San Diego\thanksmark{m1}
	,
	Universit\'e Paris-Diderot\thanksmark{m3}
		}

\address{
Department of Mathematics \\
University of California San Diego \\
9500 Gilman Dr. La Jolla \\
CA 92093  \\
United States \\
\printead{e1} \\
\printead{u1}
}

\address{
Laboratoire de probabilit\'es et modèles al\'eatoires \\
Bâtiment Sophie Germain \\
Universit\'e Paris-Diderot \\
75013 Paris \\
France \\
\printead{e2} \\
\printead{u2}
}
\end{aug}

\begin{abstract} \textit{:}
Given a noisy sample from a submanifold $M \subset \mathbb{R}^D$, we derive optimal rates for the estimation of tangent spaces $T_X M$, the second fundamental form $II_X^M$, and the submanifold $M$.
After motivating their study, we introduce a quantitative class of $\mathcal{C}^k$-submanifolds in analogy with Hölder classes.
The proposed estimators are based on local polynomials and allow to deal simultaneously with the three problems at stake. 
Minimax lower bounds are derived using a conditional version of Assouad's lemma when the base point $X$ is random.
\end{abstract}



\end{frontmatter}

\maketitle


\section{Introduction}
\label{sec:intro}

A wide variety of data can be thought of as being generated on a shape of low dimensionality compared to possibly high ambient dimension.
This point of view led to the development of the so-called topological data analysis, which proved fruitful for instance when dealing with physical parameters subject to constraints, biomolecule conformations, or natural images \cite{Wasserman16}. 
This field intends to associate geometric quantities to data without regard of any specific coordinate system or parametrization.
If the underlying structure is sufficiently smooth, one can model a point cloud $\X_n = \left\{ X_1,\ldots,X_n \right\}$ as being sampled on a $d$-dimensional submanifold $M \subset \R^D$. 
In such a case, geometric and topological intrinsic quantities include (but are not limited to) homology groups \cite{Niyogi08}, persistent homology \cite{Fasy14}, volume \cite{Arias16}, differential quantities \cite{Cazals06} or the submanifold itself \cite{Genovese12,AamariLevrard2015,Maggioni16}.

The present paper focuses on optimal rates for estimation of quantities up to order two: (0) the submanifold itself, (1) tangent spaces, and (2) second fundamental forms. 

     Among these three questions, a special attention has been paid to the estimation of the submanifold. In particular, it is a central problem in manifold learning. Indeed, there exists a wide bunch of algorithms intended to reconstruct submanifolds from point clouds (Isomap \cite{Tenenbaum00}, LLE \cite{Roweis00}, and restricted Delaunay Complexes \cite{Boissonnat14,Dey05} for instance), but few come with theoretical guarantees \cite{Genovese12,AamariLevrard2015,Maggioni16}. Up to our knowledge, minimax lower bounds were used to prove optimality in only one case \cite{Genovese12}. 
Some of these reconstruction procedures are based on tangent space estimation \cite{Boissonnat14,AamariLevrard2015,Dey05}. Tangent space estimation itself also yields interesting applications in manifold clustering \cite{Gashler11,Arias13}. 
Estimation of curvature-related quantities naturally arises in shape reconstruction, since curvature can drive the size of a meshing. 
As a consequence, most of the associated results deal with the case $d=2$ and $D=3$, though some of them may be extended to higher dimensions \cite{Merigot11,Gumhold01}. 
Several algorithms have been proposed in that case \cite{Rusinkiewicz04,Cazals06,Merigot11,Gumhold01}, but with no analysis of their performances from a statistical point of view. 

To assess the quality of such a geometric estimator, the class of submanifolds over which the procedure is evaluated has to be specified.
Up to now, the most commonly used model for submanifolds relied on the reach $\tau_{M}$, a generalized convexity parameter. 
Assuming $\tau_{M} \geq \tau_{min} > 0$ involves both local regularity --- a bound on curvature --- and global regularity --- no arbitrarily pinched area ---. 
This $\mathcal{C}^2$-like assumption has been extensively used in the computational geometry and geometric inference fields \cite{AamariLevrard2015,Niyogi08,Fasy14,Arias16,Genovese12}.
One attempt of a specific investigation for higher orders of regularity $k\geq 3$ has been proposed in \cite{Cazals06}.

Many works suggest that the regularity of the submanifold has an important impact on convergence rates.
This is pretty clear for tangent space estimation, where convergence rates of PCA-based estimators range from $(1/n)^{1/d}$ in the $\mathcal{C}^2$ case \cite{AamariLevrard2015} to $(1/n)^{\alpha}$ with $1/d<\alpha<2/d$ in more regular settings \cite{Singer12,Tyagi13}.
In addition, it seems that PCA-based estimators are outperformed by estimators taking into account higher orders of smoothness \cite{Cheng16,Cazals06}, for regularities at least $\mathcal{C}^3$. 
For instance fitting quadratic terms leads to a convergence rate of order $(1/n)^{2/d}$ in \cite{Cheng16}. These remarks naturally led us to investigate the properties of local polynomial approximation for regular submanifolds, where ``regular" has to be properly defined.
Local polynomial fitting for geometric inference was studied in several frameworks such as \cite{Cazals06}. In some sense, a part of our work extends these results, by  investigating the dependency of convergence rates on the sample size $n$, but also on the order of regularity $k$ and the ambient and intrinsic dimensions $d$ and $D$.

\subsection{Overview of the Main Results}
\label{subsec:contribution}
In this paper, we build a collection of models for $\mathcal{C}^k$-submanifolds ($k\geq 3$) that naturally generalize the commonly used one for $k=2$ (Section \ref{sec:model}). 
Roughly speaking, these models are defined by their local differential regularity $k$ in the usual sense, and by their minimum reach $\tau_{min}>0$ that may be thought of as a global regularity parameter (see Section \ref{subsec:model_k=2}). On these models, we study the non-asymptotic rates of estimation for tangent space, curvature, and manifold estimation~(Section~\ref{sec:main_results}). 
Roughly speaking, if $M$ is a $\mathcal{C}_{\tau_{min}}^k$ submanifold and if $Y_1, \hdots, Y_n$ is an $n$-sample drawn on $M$ uniformly enough, then we can derive the following minimax bounds:

\noindent
$
\displaystyle
 \text{\textit{(Theorems \ref{thm:upper_bound_tangent} and \ref{thm:lower_bound_tangent})}}
\hfill
\inf_{\hat{T}} \sup_{\substack{M \in \mathcal{C}^k \\ \tau_M \geq \tau_{min}}} \mathbb{E} \max_{1\leq j \leq n} \angle \bigl( T_{Y_j}M, \hat{T}_j \bigr) \asymp 
\left(\frac{1}{n}\right)^{\frac{k-1}{d}},
$

\medskip
\noindent
where $T_y M$ denotes the tangent space of $M$ at $y$;

\noindent
$
\displaystyle
\text{\textit{(Theorems \ref{thm:upper_bound_curvature} and \ref{thm:lower_bound_curvature})}}
\hfill
\inf_{\widehat{II}} \sup_{\substack{M \in \mathcal{C}^k \\ \tau_M \geq \tau_{min}}} \mathbb{E} \max_{1\leq j \leq n} \bigl\| II_{Y_j}^M - \widehat{II}_j \bigr\| \asymp 
\left(\frac{1}{n}\right)^{\frac{k-2}{d}},
$

\medskip
\noindent
where $II_y^M$ denotes the second fundamental form of $M$ at $y$;

\noindent
$
\displaystyle
\text{
\textit{(Theorems \ref{thm:upper_bound_hausdorff} and \ref{thm:lower_bound_hausdorff})}
}
\hfill\hfill\hfill 
\inf_{\hat{M}} \sup_{\substack{M \in \mathcal{C}^k \\ \tau_M \geq \tau_{min}}} \mathbb{E}~d_H \bigl( M, \hat{M} \bigr) \asymp 
\left(\frac{1}{n}\right)^{\frac{k}{d}},
\hfill 
$

\medskip
\noindent
where $d_H$ denotes the Hausdorff distance.

These results shed light on the influence of $k$, $d$, and $n$ on these estimation problems, showing for instance that the ambient dimension $D$ plays no role. The estimators proposed for the upper bounds all rely on the analysis of local polynomials, and allow to deal with the three estimation problems in a unified way 
(Section \ref{subsec:proofs_l_polynomials}). Some of the lower bounds are derived using a new version of Assouad's Lemma (Section \ref{subsubsec:proofs_lower_bounds_tangent_curvature}).

We also emphasize the influence of the reach $\tau_M$ of the manifold $M$ in Theorem \ref{thm:minimax_non_consistency_tangent_and_curvature}. 
Indeed, we show that whatever the local regularity $k$ of $M$, if we only require $\tau_M \geq 0$, then for any fixed point $y \in M$,
            \[
            \inf_{\hat{T}} \sup_{\substack{M \in \mathcal{C}^k \\ \tau_M \geq 0}} \mathbb{E} \angle \bigl ( T_{y} M, \hat{T} \bigr ) \geq 1/2,
            \qquad
                        \inf_{\widehat{II}} \sup_{\substack{M \in \mathcal{C}^k \\ \tau_M \geq 0}} \mathbb{E} \bigl\Vert{II_y^M - \widehat{II}}\bigr\Vert \geq c > 0,
            \]
assessing that the global regularity parameter $\tau_{min}> 0$ is crucial for estimation purpose.

        It is worth mentioning that our bounds also allow for perpendicular noise of amplitude $\sigma>0$. When $\sigma \lesssim (1/n)^{\alpha/d}$ for $1 \leq \alpha$, then our estimators behave as if the corrupted sample $X_1, \hdots, X_n$ were exactly drawn on a manifold with regularity $\alpha$. Hence our estimators turn out to be optimal whenever $\alpha \geq k$. If $\alpha <k$, the lower bounds suggest that better rates could be obtained with different estimators, by pre-processing data as in \cite{Genovese11} for instance.

For the sake of completeness, geometric background and proofs of technical lemmas are given in the Appendix.


\section{$\mathcal{C}^k$ Models for Submanifolds}
\label{sec:model}
\subsection{Notation}
\label{subsec:general_notation}
Throughout the paper, we consider $d$-dimensional compact submanifolds $M \subset \R^D$ without boundary. 
The submanifolds will always be assumed to be at least $\mathcal{C}^2$.
For all $p \in M$, $T_p M$ stands for the tangent space of $M$ at $p$ \cite[Chapter 0]{DoCarmo92}. 
We let $II_p^M: T_p M \times T_p M \rightarrow T_p M^\perp$ denote the second fundamental form of $M$ at $p$ \cite[p. 125]{DoCarmo92}. $II_p^M$ characterizes the curvature of $M$ at $p$.
The standard inner product in $\R^D$ is denoted by $\langle \cdot,\cdot \rangle$ and the Euclidean distance by $\norm{\cdot}$.
Given a linear subspace $T \subset \R^D$, write $T^\perp$ for its orthogonal space. 
We write  $\mathcal{B}(p,r)$ for the closed Euclidean ball of radius $r>0$ centered at $p \in \R^D$, and for short $\mathcal{B}_T(p,r) = \mathcal{B}(p,r) \cap T$. 
For a smooth function $\Phi: \R^D \rightarrow \R^D$ and $i \geq 1$, we let $d_x^i \Phi$ denote the $i$th order differential of $\Phi$ at $x\in \R^D$. 
For a linear map $A$ defined on $T \subset \R^D$, $\norm{A}_\mathrm{op}= {\sup_{v \in T}} \frac{\norm{A v}}{\norm{v}}$ stands for the operator norm. 
We adopt the same notation $\norm{\cdot}_{op}$ for tensors, i.e. multilinear maps. 
Similarly, if $\left\{A_x\right\}_{x \in T'}$ is a family of linear maps, its $L^\infty$ operator norm is denoted by $\norm{A}_{op} = \sup_{x \in T'} \norm{A_x}_{op}$.
When it is well defined, we will write $\pi_B(z)$ for the projection of $z \in \R^D$ onto the closed subset $B \subset \R^D$, that is the nearest neighbor of $z$ in $B$.
The distance between two linear subspaces $U,V \subset \R^D$ of the same dimension is measured by the principal angle {\color{black}{$\angle(U,V) 
= 
\norm{\pi_U - \pi_V}_\mathrm{op}.$}}
The Hausdorff distance \cite{Genovese12} in $\R^D$ is denoted by ${d}_{H}$.
For a probability distribution $P$, $\E_P$ stands for the expectation with respect to $P$. 
We write $P^{\otimes n}$ for the $n$-times tensor product of $P$.

Throughout this paper, $C_\alpha$ will denote a generic constant depending on the parameter $\alpha$. For clarity's sake, $C'_\alpha$, $c_\alpha$, or $c'_\alpha$ may also be used when several constants are involved.

\subsection{Reach and Regularity of Submanifolds}
\label{subsec:model_k=2}
As introduced in \cite{federer1959}, the reach $\tau_M$ of a subset $M \subset \R^D$ is the maximal neighborhood radius for which the projection $\pi_M$ onto $M$ is well defined. More precisely, denoting by $d(\cdot,M)$ the distance to $M$, the medial axis of $M$ is defined to be the set of points which have at least two nearest neighbors on $M$, that is
\begin{align*}
Med(M) 
= 
\left\{
z \in \R^D | \exists p \neq q \in M, \norm{z-p} = \norm{z-q} = d(z,M)
\right\}.
\end{align*}
The reach is then defined by
\begin{align*}
\tau_M = 
\underset{p \in M}{\inf} d\left(p,Med(M)\right) 
= 
\underset{z \in Med(M)}{\inf} d\left(z,M\right)
. 
\end{align*}
It gives a minimal scale of geometric and topological features of $M$.
\begin{figure}\label{fig:reach}
\begin{center}
\includegraphics[scale=0.6]{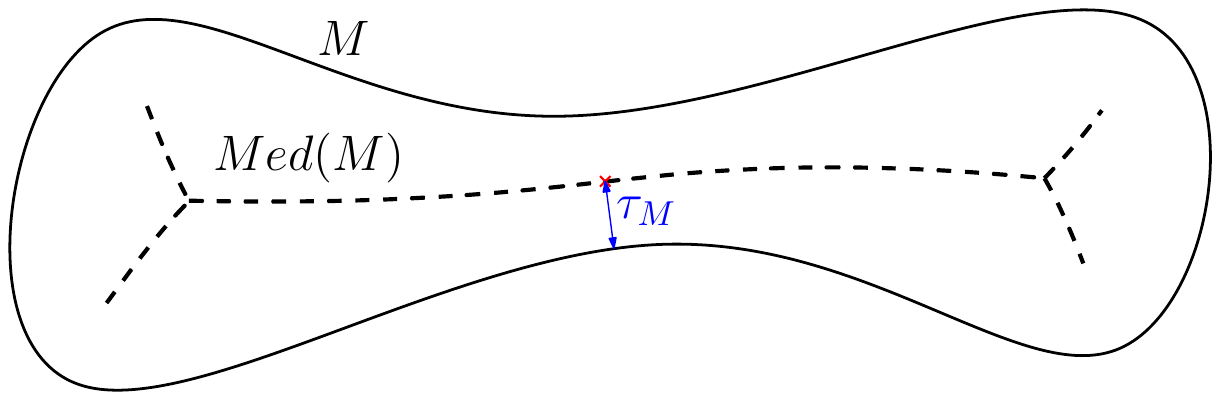}
\end{center}
\caption{Medial axis and reach of a closed curve in the plane.}
\end{figure}
As a generalized convexity parameter, $\tau_M$ is a key parameter in reconstruction \cite{AamariLevrard2015,Genovese12} and in topological inference \cite{Niyogi08}.
Having $\tau_M \geq \tau_{min} > 0$ prevents $M$ from almost auto-intersecting, and bounds its curvature in the sense that $\norm{II^M_p}_{op} \leq \tau_M^{-1 } \leq \tau_{min}^{-1}$ for all $p\in M$ \cite[Proposition 6.1]{Niyogi08}.

{\color{black}
For $\tau_{min}>0$, we let $\mathcal{C}^{2}_{\tau_{min}}$ denote the set of $d$-dimensional compact connected submanifolds $M$ of $\R^D$ such that $\tau_M \geq \tau_{min}>0$.
A key property of submanifolds $M \in \mathcal{C}^2_{\tau_{min}}$ is the existence of a parametrization closely related to the projection onto tangent spaces.
We let $\exp_p : T_p M \rightarrow M$ denote the exponential map of $M$  \cite[Chapter 3]{DoCarmo92}, that is defined by $\exp_p(v) = \gamma_{p,v}(1)$, where $\gamma_{p,v}$ is the unique constant speed geodesic path of $M$ with initial value $p$ and velocity $v$.
\begin{lem}\label{model_properties_k=2}
If $M \in \mathcal{C}^2_{\tau_{min}}$,
$\exp_p: \mathcal{B}_{T_pM}\left(0, \tau_{min}/4 \right) \rightarrow M$ is one-to-one. Moreover, it can be written as
\begin{align*}
  \exp_p \colon \mathcal{B}_{T_pM}\left(0, \tau_{min}/4 \right) &\longrightarrow M 
  \\
  v &\longmapsto p + v + \mathbf{N}_p(v)
\end{align*}
with $\mathbf{N}_p$ such that for all $v \in \mathcal{B}_{T_pM}\left(0, \tau_{min} /4 \right)$,
\begin{align*}
\mathbf{N}_p(0)=0, \quad d_0 \mathbf{N}_p = 0 , \quad \norm{d_v \mathbf{N}_p}_{op} \leq L_\perp \norm{v},
\end{align*}
where $L_\perp = 5/(4\tau_{min})$.
Furthermore, for all $p,y \in M $,
\begin{align*}
y-p = \pi_{T_pM}(y-p) + R_2(y-p),
\end{align*}
where $\norm{R_2(y-p)} \leq \frac{\norm{y-p}^2}{2\tau_{min}}$. 
\end{lem}

A proof of Lemma \ref{model_properties_k=2} is given in Section \ref{subsec:A1} of the Appendix. In other words, elements of $\mathcal{C}^2_{\tau_{min}}$ have local parametrizations on top of their tangent spaces that are defined on neighborhoods with a minimal radius, and these parametrizations differ from the identity map by at most a quadratic term. 
The existence of such local parametrizations leads to the following convergence result:
}
if data $Y_1, \hdots, Y_n$ are drawn uniformly enough on $M \in \mathcal{C}^2_{\tau_{min}}$, then it is shown in  \cite[Proposition 14]{AamariLevrard2015} that a tangent space estimator $\hat{T}$ based on local PCA achieves 
       \[
      \mathbb{E} \max_{1\leq j \leq n} \angle \bigl( T_{Y_j}M, \hat{T}_j \bigr) \leq C 
\left(\frac{1}{n}\right)^{\frac{1}{d}}.
\]

When $M$ is smoother, it has been proved in \cite{Cheng16} that a convergence rate in $n^{-2/d}$ might be achieved, based on the existence of a local order $3$ Taylor expansion of the submanifold on top of its tangent spaces.
{\color{black}
 Thus, a natural extension of the $\mathcal{C}^2_{\tau_{min}}$ model to $\mathcal{C}^{k}$-submanifolds should ensure that such an expansion exists at order $k$ and satisfies some regularity constraints. To this aim, we introduce the following class of regularity $\mathcal{C}^{k}_{\tau_{min}, \mathbf{L}}$.
}
\begin{Def}\label{best_model_definition}
For $k\geq 3$, $\tau_{min}>0$, and $\mathbf{L} =  (L_{\perp}, L_3,\ldots,L_k)$, we let $\mathcal{C}^{k}_{\tau_{min},\mathbf{L}}$ denote the set of $d$-dimensional compact connected submanifolds $M$ of $\R^D$ with $\tau_M \geq \tau_{min}$ and such that, for all $p \in M$, there exists a local one-to-one parametrization $\Psi_p$ of the form:
\begin{align*}
  \Psi_p \colon \mathcal{B}_{T_pM}\left(0,r \right) &\longrightarrow M 
  \\
  v &\longmapsto p + v + \mathbf{N}_p(v)
\end{align*}
for some $r \geq \frac{1}{4 L_\perp}$,
with $\mathbf{N}_p \in \mathcal{C}^{k}\left( \mathcal{B}_{T_pM}\left(0, r \right) , \R^D\right)$  such that
\begin{align*}
\mathbf{N}_p(0)=0 , \quad d_0 \mathbf{N}_p = 0 , \quad \norm{d_v^2 \mathbf{N}_p}_{op} \leq L_\perp,
\end{align*}
for all $\norm{v} \leq \frac{1}{4 L_\perp}$. 
Furthermore, we require that
\begin{align*}
\norm{d^i_v \mathbf{N}_p}_{op} \leq L_i \text{ for all }3 \leq i \leq k. 
\end{align*}
\end{Def}
{\color{black}{
It is important to note that such a family of $\Psi_p$'s exists for any compact $\mathcal{C}^k$-submanifold, if one allows $\tau_{min}^{-1}$, $L_{\perp}$, $L_3$,$\ldots$,$L_k$ to be large enough.
Note that the radius $1/(4L_\perp)$ has been chosen for convenience. 
Other smaller scales would do and we could even parametrize this constant, but without substantial benefits in the results.
}}

The $\Psi_p$'s can be seen as unit parametrizations of $M$. The conditions on $\mathbf{N}_p(0)$, $d_0 \mathbf{N}_p$, and $d^2_v \mathbf{N}_p$ ensure that $\Psi_p^{-1}$ is close to the projection $\pi_{T_p M}$.
The bounds on $d_v^i\mathbf{N}_p$ ($3\leq i \leq k$) allow to control the coefficients of the polynomial expansion we seek. Indeed, whenever $M \in \mathcal{C}^k_{\tau_{min},\mathbf{L}}$, Lemma \ref{lem:fourre_tout_geom_model} shows that for every $p$ in $M$, and $y $ in $\mathcal{B}\bigl(p,\frac{\tau_{min} \wedge L_\perp^{-1}}{4}\bigr)\cap M$, 
       \begin{align}\label{korderdecomposition}
       y-p = \pi^*&(y-p) + \sum_{i=2}^{k-1}{T_i^*( \tens{\pi^*(y-p)}{i})} + R_{k}(y-p),
       \end{align}
where $\pi^*$ denotes the orthogonal projection onto $T_pM$, the $T_i^*$ are $i$-linear maps from $T_pM$ to $\mathbb{R}^D$ with $\|T_i^*\|_{op} \leq L'_i$ and $R_k$ satisfies $\|R_k(y-p)\| \leq C \|y-p\|^k$, where the constants $C$ and the $L'_i$'s depend on the parameters $\tau_{min}$, $d$, $k$, $L_{\perp}, \hdots, L_k$.

 {\color{black} Note that for $k\geq3$ the exponential map can happen to be only $\mathcal{C}^{k-2}$ for a $\mathcal{C}^k$-submanifold \cite{Hartman51}. Hence, it may not be a good choice of $\Psi_p$.} However, for $k=2$, taking $\Psi_p = \exp_p$ is sufficient for our purpose. For ease of notation, we may write $\mathcal{C}^2_{\tau_{min},\mathbf{L}}$ although the specification of $\mathbf{L}$ is useless. In this case, we implicitly set by default $\Psi_p = \exp_p$ and $L_\perp = 5/(4\tau_{min})$.
As will be shown in Theorem \ref{thm:minimax_non_consistency_tangent_and_curvature}, the global assumption $\tau_M \geq \tau_{min}>0$ cannot be dropped, even when higher order regularity bounds $L_i$'s are fixed.




Let us now describe the statistical model. Every $d$-dimensional submanifold $M \subset \R^D$ inherits a natural uniform volume measure by restriction of the ambient $d$-dimensional Hausdorff measure $\mathcal{H}^d$. 
In what follows, we will consider probability distributions that are almost uniform on some $M$ in $\mathcal{C}^k_{\tau_{min},\mathbf{L}}$, with some bounded noise, as stated below. 
\begin{Def}[Noise-Free and Tubular Noise Models]\label{statistical_best_model}
\ \\
\noindent
- \textit{(Noise-Free Model)} For $k\geq 2$, $\tau_{min}>0$, $\mathbf{L} = (L_{\perp}, L_3,\ldots,L_k)$ and $f_{min} \leq f_{max}$, we let $\mathcal{P}^k_{\tau_{min}, \mathbf{L},f_{min},f_{max}}$ denote the set of distributions $P_0$ with support $M \in \mathcal{C}^k_{\tau_{min}, \mathbf{L}}$ that have a density $f$ with respect to the volume measure on $M$, and such that for all $y \in M$,
\begin{align*}
0 < f_{min} \leq f(y) \leq f_{max} < \infty.
\end{align*}

\noindent
- \textit{(Tubular Noise Model)} For $0 \leq \sigma < \tau_{min}$,
we denote by $\mathcal{P}^k_{\tau_{min}, \mathbf{L},f_{min},f_{max}}\left(\sigma \right)$ the set of distributions of random variables $X = Y + Z$, where $Y$ has distribution $P_0 \in \mathcal{P}^k_{\tau_{min}, \mathbf{L},f_{min},f_{max}}$, and $Z \in T_{Y}M^{\perp}$ with $\|Z\| \leq \sigma$ and $\mathbb{E}(Z|Y)=0$.
      
\end{Def}
For short, we write $\mathcal{P}^k$ and $\mathcal{P}^k(\sigma)$ when there is no ambiguity.
We denote by $\X_n$  an i.i.d. $n$-sample $\left\{ X_1,\ldots,X_n \right\}$, that is, a sample with distribution $P^{\otimes n}$ for some $P \in \mathcal{P}^k(\sigma)$, so that $X_i = Y_i + Z_i$, where $Y$ has distribution $P_0 \in \mathcal{P}^k$, $Z\in \mathcal{B}_{T_Y M^\perp}(0,\sigma)$ with $\mathbb{E}(Z|Y)=0$. It is immediate that for $\sigma < \tau_{min}$, we have $Y = \pi_M(X)$. Note that the tubular noise model $\mathcal{P}^k(\sigma)$ is a slight generalization of that in \cite{Genovese11}.

In what follows, though $M$ is unknown, all the parameters of the model will be assumed to be known, including the intrinsic dimension $d$ and the order of regularity $k$.
We will also denote by $\mathcal{P}^k_{(x)}$ the subset of elements in $\mathcal{P}^k$ whose support contains a prescribed $x \in \R^D$.

In view of our minimax study on $\mathcal{P}^k$, it is important to ensure by now that $\mathcal{P}^k$ is stable with respect to deformations and dilations.
\begin{prop}\label{statistical_model_stability}
Let $\Phi : \R^D \rightarrow \R^D$ be a global $\mathcal{C}^k$-diffeomorphism. If
$\norm{d \Phi - I_D}_{op}$
,
$\norm{d^2 \Phi}_{op}$
,
\ldots
,
$\norm{d^k \Phi}_{op}
$ 
are small enough,
then for all $P$ in $\mathcal{P}^{k}_{\tau_{min},\mathbf{L}, f_{min}, f_{max}}$, the pushforward distribution $P' = \Phi_\ast P$ belongs to $\mathcal{P}^{k}_{\tau_{min}/2,2\mathbf{L}, f_{min}/2, 2 f_{max}}$.

Moreover, if $\Phi = \lambda I_D$ ($\lambda>0$) is an homogeneous dilation, then $P' \in \mathcal{P}^{k}_{\lambda \tau_{min},\mathbf{L}_{(\lambda)},f_{min}/\lambda^d,f_{max}/\lambda^d}$, where $\mathbf{L}_{(\lambda)} = (L_\perp/\lambda,L_3/\lambda^2,\ldots,L_k/\lambda^{k-1})$.
\end{prop}
Proposition \ref{statistical_model_stability} follows from a geometric reparametrization argument (Proposition A.5 in Appendix A) and a change of variable result for the Hausdorff measure (Lemma A.6 in Appendix A).


\subsection{Necessity of a Global Assumption}
\label{subsec:reach_is_necessary}

In the previous Section \ref{subsec:model_k=2}, we generalized $\mathcal{C}^2$-like models --- stated in terms of reach --- to $\mathcal{C}^k$, for $k\geq 3$, by imposing  higher order differentiability bounds on parametrizations $\Psi_p$'s. The following Theorem \ref{thm:minimax_non_consistency_tangent_and_curvature} shows that the global assumption  $\tau_M \geq \tau_{min} > 0$ is necessary for estimation purpose.
\begin{thm}\label{thm:minimax_non_consistency_tangent_and_curvature}
Assume that $\tau_{min} = 0$. 
If $D \geq d+3$, then for all $k \geq 3$ and $L_\perp>0$, provided that $L_3/L_\perp^2,\ldots,{L_k}/L_\perp^{k-1},{L_\perp^d}/{f_{min}}$ and ${f_{max}}/{L_\perp^d}$ are large enough (depending only on $d$ and $k$), for all $n\geq 1$, 
\begin{align*}
\inf_{\hat{T}} 
\sup_{P \in \mathcal{P}_{(x)}^k}
\E_{P^{\otimes n}}
\angle
\bigl(
T_{x} M
,
\hat{T}
\bigr)
\geq \frac{1}{2} > 0,
\end{align*}
where the infimum is taken over all the estimators $\hat{T} = \hat{T}\bigl(X_1,\ldots,X_n\bigr)$. 

Moreover, for any $D \geq d+1$, provided that $L_3/L_\perp^2,\ldots,{L_k}/L_\perp^{k-1},{L_\perp^d}/{f_{min}}$ and ${f_{max}}/{L_\perp^d}$ are large enough (depending only on $d$ and $k$), for all $n\geq 1$, 
\begin{align*}
\inf_{\widehat{II}} 
\sup_{P \in \mathcal{P}_{(x)}^k}
\E_{P^{\otimes n}}
\norm{
II_{x}^M
\circ
\pi_{T_{x} M}
-
\widehat{II}
}_{op}
\geq \frac{L_\perp}{4} > 0,
\end{align*}
where the infimum is taken over all the estimators $\widehat{II} = \widehat{II}\bigl(X_1,\ldots,X_n\bigr)$.
\end{thm}
The proof of Theorem \ref{thm:minimax_non_consistency_tangent_and_curvature} can be found in Section \ref{subsec:C5}. In other words, if the  class of submanifolds is allowed to have arbitrarily small reach, no estimator can perform uniformly well to estimate neither $T_x M$ nor $II_x^M$. And this, even though each of the underlying submanifolds have arbitrarily smooth parametrizations.
Indeed, if two parts of $M$ can nearly intersect around $x$ at an arbitrarily small scale $\Lambda \rightarrow 0$, no estimator can decide whether the direction (resp. curvature) of $M$ at $x$ is that of the first part or the second part (see Figures \ref{fig:non_consistency_two_hypotheses_tangent} and \ref{fig:non_consistency_two_hypotheses_curvature}).
\begin{figure}[h]
\centering
\includegraphics[width= 0.7\textwidth]{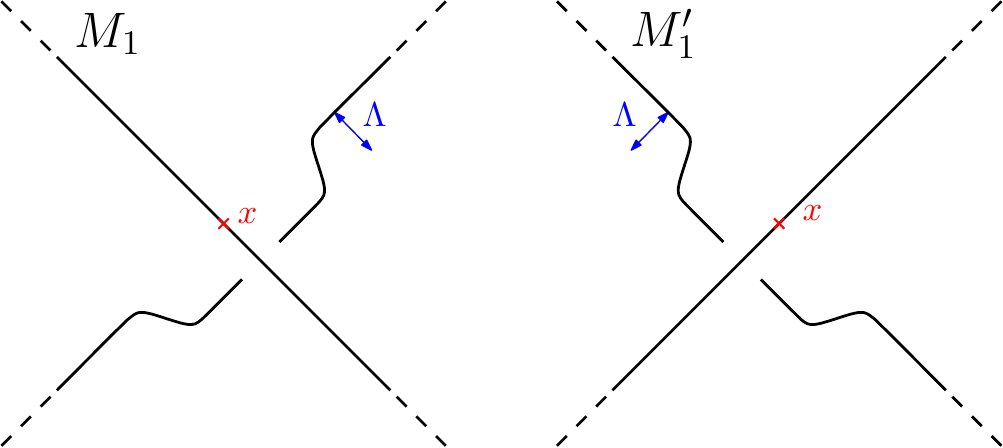}
\caption{Inconsistency of tangent space estimation for $\tau_{min} = 0$. }
\label{fig:non_consistency_two_hypotheses_tangent}
\end{figure}

\begin{figure}[h]
\centering
\includegraphics[width= 0.7\textwidth]{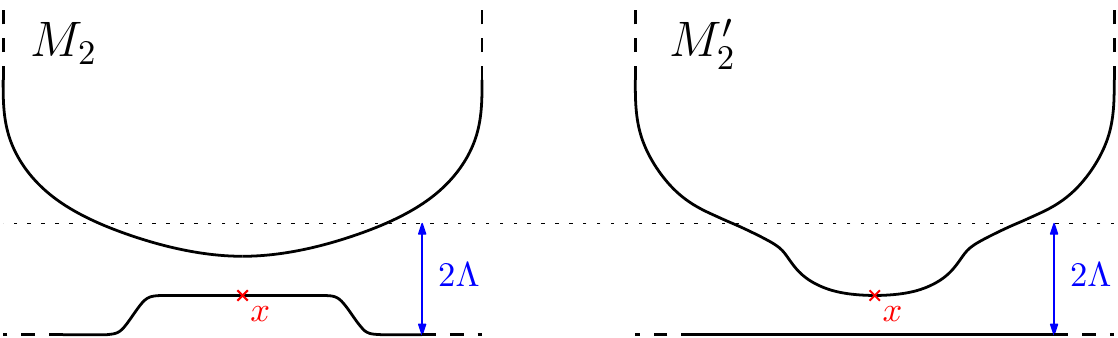}
\caption{Inconsistency of curvature estimation for $\tau_{min} = 0$. }
\label{fig:non_consistency_two_hypotheses_curvature}
\end{figure}


\section{Main Results}
\label{sec:main_results}
Let us now move to the statement of the main results.
Given an i.i.d. $n$-sample $\X_n=\left\{X_1,\ldots,X_n\right\}$ with unknown common distribution $P \in \mathcal{P}^k(\sigma)$, we detail non-asymptotic rates for the estimation of tangent spaces  $T_{Y_j} M$, second fundamental forms $II_{Y_j}^M$, and $M$ itself.

For this, we need one more piece of notation. For $1 \leq j \leq n$, $P_{n-1}^{(j)}$ denotes integration with respect to $1/(n-1) \sum_{i \neq j} \delta_{(X_i - X_j)}$, and $\tens{z}{i}$ denotes the $D \times i$-dimensional vector $(z, \ldots, z)$.
%
For a constant $t > 0$ and a bandwidth $h>0$ to be chosen later, we define the local polynomial estimator $(\hat{\pi}_{j},\hat{T}_{2,j},\ldots,\hat{T}_{k-1,j})$ at $X_j$ to be any element of
\begin{align}
\label{estimator_full_tensors_definition0}
\underset{\pi, \sup_{2 \leq i \leq k} \|T_i\|_{op} \leq t }{\arg\min} 
P_{n-1}^{(j)} 
\left [ 
\norm{
x
-
\pi(x)
- 
\sum_{i=2}^{k-1}
T_i(\tens{\pi (x)}{i})
}^2 \mathbbm{1}_{\mathcal{B}(0,h)}(x) 
\right]
,  
\end{align}
where $\pi$ ranges among all the orthogonal projectors on $d$-dimensional subspaces, and $T_i : \left(\R^D\right)^i \rightarrow \R^D$ among the symmetric tensors of order $i$  such that $\norm{T_i}_{op} \leq t$.
For $k=2$, the sum over the tensors $T_i$ is empty, and the integrated term reduces to $\norm{x - \pi(x)}^2\mathbbm{1}_{\mathcal{B}(0,h)}(x)$.
By compactness of the domain of minimization, such a minimizer exists almost surely. 
In what follows, we will work with a maximum scale $h \leq h_0$, with
\[
h_0 = \frac{\tau_{min} \wedge L_\perp^{-1}}{8}.
\]

The set of $d$-dimensional orthogonal projectors is not convex, which leads to a more involved optimization problem  than usual least squares. In practice, this problem may be solved using tools from optimization on Grassman manifolds \cite{Usevich14}, or adopting a two-stage procedure such as in \cite{Cazals06}: from local PCA, a first $d$-dimensional space is estimated at each sample point, along with an orthonormal basis of it. Then, the optimization problem \eqref{estimator_full_tensors_definition0} is expressed as a minimization problem in terms of the coefficients of $(\pi_j, T_{2,j}, \hdots, T_{k,j})$ in this basis under orthogonality constraints. 
It is worth mentioning that a similar problem is explicitly solved in \cite{Cheng16}, leading to an optimal tangent space estimation procedure in the case $k=3$.

The constraint $\|T_i\|_{op}\leq t$ involves a parameter $t$ to be calibrated. As will be shown in the following section, it is enough to choose $t$ roughly smaller than $1/h$, but still larger than the unknown norm of the optimal tensors $\|T_i^*\|_{op}$. Hence, for $h \rightarrow 0$, the choice $t=h^{-1}$ works to guarantee optimal convergence rates. 
Such a constraint on the higher order tensors might have been stated under the form of a $\|.\|_{op}$-penalized   least squares minimization --- as in ridge regression --- leading to the same results.


\subsection{Tangent Spaces}
\label{subsec:main_results_tangent_space}

By definition, the tangent space $T_{Y_j} M$ is the best linear approximation of $M$ nearby $Y_j$. Thus, it is natural to take the range of the first order term minimizing $\eqref{estimator_full_tensors_definition0}$ and write $\hat{T}_j = \image \hat{\pi}_j$. The $\hat{T}_j$'s approximate simultaneously the $T_{Y_j}M$'s with high probability, as stated below.

\begin{thm}\label{thm:upper_bound_tangent}
Assume that $t \geq C_{k,d,\tau_{min},\mathbf{L}} \geq  \sup_{2 \leq i \leq k} \|T^*_i\|_{op} $.
Set $h = \left( C_{d,k} \frac{f_{max}^2\log n }{f_{min}^3(n-1)} \right )^{1/d}$, for $C_{d,k}$ large enough, and assume that $\sigma \leq h/4$. If $n$ is large enough so that $h \leq h_0$, 
then with probability at least $1 - \left(\frac{1}{n}\right)^{k/d}$,
\begin{align*}
\max_{1 \leq j \leq n} 
\angle 
\bigl(T_{Y_j} M, \hat{T}_j\bigr)
\leq 
C_{d,k,\tau_{min},\mathbf{L}} \sqrt{\frac{f_{max}}{f_{min}}} (h^{k-1}\vee \sigma h^{-1}) (1 + th)
.
\end{align*}
As a consequence, taking $t = h^{-1}$, for $n$ large enough,

\begin{align*}
\sup_{P \in \mathcal{P}^k(\sigma)}
\E_{P^{\otimes n}}
\max_{1 \leq j \leq n} 
\angle 
\bigl(T_{Y_j} M, \hat{T}_j\bigr)
\leq 
C
\left(
\frac{\log n}{n-1}
\right)^{\frac{k-1}{d}} \left\{ 1 \vee \sigma \left (\frac{\log n}{n-1}
\right)^{-\frac{k}{d}} \right\},
\end{align*}
where $C = C_{d,k,\tau_{min},\mathbf{L},f_{min},f_{max}}$.
\end{thm}
The proof of Theorem \ref{thm:upper_bound_tangent} is given in Section \ref{sec_proofofthmupper_bound_tangent}. The same bound holds for the estimation of $T_y M$ at a prescribed $y \in M$ in the model $\mathcal{P}^k_{(y)}(\sigma)$. For that, simply take $P_n^{(y)} = 1/n \sum_{i} \delta_{(X_i - y)}$ as integration in \eqref{estimator_full_tensors_definition0}.

 In the noise-free setting, or when $\sigma \leq h^k$, this result is in line with those of \cite{Cazals06} in terms of the sample size dependency $(1/n)^{(k-1)/d}$. Besides, it shows that the convergence rate of our estimator does not depend on the ambient dimension $D$, even in codimension greater than $2$. 
When $k=2$, we recover the same rate as \cite{AamariLevrard2015}, where we used local PCA, which is a reformulation of \eqref{estimator_full_tensors_definition0}. 
When $k \geq 3$, the procedure \eqref{estimator_full_tensors_definition0} outperforms PCA-based estimators of \cite{Singer12} and \cite{Tyagi13}, where convergence rates of the form $(1/n)^{\beta}$ with $1/d<\beta<2/d$ are obtained. 
This bound also recovers the result of \cite{Cheng16} in the case $k=3$, where a similar procedure is used. When the noise level $\sigma$ is of order $h^{\alpha}$, with $1 \leq \alpha \leq k$, Theorem \ref{thm:upper_bound_tangent} yields a convergence rate in $h^{\alpha-1}$. Since a polynomial decomposition up to order $k_{\alpha} = \lceil \alpha \rceil$ in \eqref{estimator_full_tensors_definition0} results in the same bound, the noise level $\sigma = h^{\alpha}$ may be thought of as an $\alpha$-regularity threshold. At last, it may be worth mentioning that the results of Theorem \ref{thm:upper_bound_tangent} also hold when the assumption $\mathbb{E}(Z|Y) = 0$ is relaxed.
Theorem \ref{thm:upper_bound_tangent} nearly matches the following lower bound.

\begin{thm}\label{thm:lower_bound_tangent}
{If $\tau_{min}{L_\perp},\ldots,\tau_{min}^{k-1}{L_k},({\tau_{min}^d}f_{min})^{-1}$ and ${\tau_{min}^d}{f_{max}}$ are large enough (depending only on $d$ and $k$),}
 then
\begin{multline*}
\inf_{\hat{T}} 
\sup_{P \in \mathcal{P}^k(\sigma)}
\E_{P^{\otimes n}}
\angle
\bigr(
T_{\pi_M(X_1)} M
,
\hat{T}
\bigr) \\
\geq 
c_{d,k,\tau_{min}} 
\left \{
\left(\frac{1}{n-1}\right)^{\frac{k-1}{d}} \right . 
\left .
\vee
\left(\frac{\sigma}{n-1}\right)^{\frac{k-1}{d+k}}
\right \}
,
\end{multline*}
where the infimum is taken over all the estimators $\hat{T} = \hat{T}(X_1,\ldots,X_n)$.
\end{thm}
A proof of Theorem \ref{thm:lower_bound_tangent} can be found in Section \ref{subsubsec:proofs_lower_bounds_tangent_curvature}. When $\sigma \lesssim (1/n)^{k/d}$, the lower bound matches Theorem \ref{thm:upper_bound_tangent} in the noise-free case, up to a $\log n$ factor. Thus, the rate $(1/n)^{(k-1)/d}$ is optimal for tangent space estimation on the model $\mathcal{P}^k$. The rate $(\log n /n)^{1/d}$ obtained in \cite{AamariLevrard2015} for $k=2$ is therefore optimal, as well as the rate $(\log n /n)^{2/d}$ given in \cite{Cheng16} for $k=3$.
The rate $(1/n)^{(k-1)/d}$ naturally appears on the the model $\mathcal{P}^k$, as the estimation rate of differential objects of order $1$ from $k$-smooth submanifolds.

When $\sigma \asymp (1/n)^{\alpha/d}$ with $\alpha <k$, the lower bound provided by Theorem \ref{thm:lower_bound_tangent} is of order $(1/n)^{(k-1)(\alpha + d)/[d(d+k)]}$, hence smaller than the $(1/n)^{\alpha/d}$ rate of Theorem \ref{thm:upper_bound_tangent}. 
This suggests that the local polynomial estimator \eqref{estimator_full_tensors_definition0} is suboptimal whenever $\sigma \gg (1/n)^{k/d}$ on the model $\mathcal{P}^k(\sigma)$.

Here again, the same lower bound holds for the estimation of $T_y M$ at a fixed point $y$ in the model $\mathcal{P}^k_{(y)}(\sigma)$.

\subsection{Curvature}
\label{subsec:main_results_curvature}
The second fundamental form $II_{Y_j}^M : T_{Y_j} M \times T_{Y_j} M \rightarrow T_{Y_j} M^\perp \subset \R^D$ is a symmetric bilinear map that encodes completely the curvature of $M$ at $Y_j$ \cite[Chap. 6, Proposition 3.1]{DoCarmo92}.
Estimating it only from a point cloud $\X_n$ does not trivially make sense, since $II_{Y_j}^M$ has domain $T_{Y_j} M$ which is unknown. To bypass this issue we extend $II_{Y_j}^M$ to $\R^D$. That is, we consider the estimation of $II_{Y_j}^M \circ \pi_{T_{Y_j} M}$ which has full domain $\R^D$. 
Following the same ideas as in the previous Section \ref{subsec:main_results_tangent_space}, we use the second order tensor 
$ \hat{T}_{2,j} \circ \hat{\pi}_j$ obtained in \eqref{estimator_full_tensors_definition0} to estimate $II_{Y_j}^M \circ \pi_{T_{Y_j} M}$.

\begin{thm}\label{thm:upper_bound_curvature}
Let $k \geq 3$. Take $h$ as in Theorem \ref{thm:upper_bound_tangent}, $\sigma \leq h/4$, and $t=1/h$. If $n$ is large enough so that $h \leq h_0$ and $h^{-1} \geq C^{-1}_{k,d,\tau_{min},\mathbf{L}} \geq  (\sup_{2 \leq i \leq k} \|T^*_i\|_{op})^{-1}$, then with probability at least $1 - \left(\frac{1}{n}\right)^{k/d}$,
\begin{align*}	
\max_{1 \leq j \leq n}
\norm{
II_{Y_j}^M \circ \pi_{T_{Y_j} M}- \hat{T}_{2,j} \circ \hat{\pi}_j
}_{op} 
&\leq 
C_{d,k,\tau_{min},\mathbf{L}} \sqrt{\frac{f_{max}}{f_{min}}}(h^{k-2} \vee \sigma h^{-2}) 
.
\end{align*}
In particular, for $n$ large enough,
\begin{multline*}
\sup_{P \in \mathcal{P}^k(\sigma)}
\E_{P^{\otimes n}}
\max_{1 \leq j \leq n} 
\norm{
II_{Y_j}^M \circ \pi_{T_{Y_j} M}- \hat{T}_{2,j} \circ \hat{\pi}_j
}_{op} \\
\leq 
C_{d,k,\tau_{min},\mathbf{L},f_{min},f_{max}}
\left(
\frac{\log n}{n-1}
\right)^{\frac{k-2}{d}}\left\{ 1 \vee \sigma 
\left(
\frac{\log n}{n-1}
\right)^{-\frac{k}{d}}\right\}.
\end{multline*}
\end{thm}
     The proof of Theorem \ref{thm:upper_bound_curvature} is given in Section \ref{subsubsec:proofofuppercurvature}. As in Theorem \ref{thm:upper_bound_tangent}, the case $\sigma \leq h^k$ may be thought of as a noise-free setting, and provides an upper bound of the form $h^{k-2}$. Interestingly, Theorems \ref{thm:upper_bound_tangent} and \ref{thm:upper_bound_curvature} are enough to provide estimators of various notions of curvature.
For instance, consider the scalar curvature \cite[Section 4.4]{DoCarmo92} at a point $Y_j$, defined by 
\begin{align*}
Sc_{Y_j}^M = \frac{1}{d(d-1)}\sum_{r \neq s}\left [ \left \langle II^M_{Y_j}(e_r,e_r),II^M_{Y_j}(e_s,e_s) \right\rangle - \|II^M_{Y_j}(e_r,e_s)\|^2 \right ]
,
\end{align*}
where $(e_r)_{1 \leq r \leq d}$ is an orthonormal basis of $T_{Y_j}M$. A plugin estimator of $Sc_{Y_j}^M$ is
     \[
     \widehat{Sc}_j = \frac{1}{d(d-1)} \sum_{r \neq s}\left [ \left \langle \hat{T}_{2,j}(\hat{e}_r,\hat{e}_r),\hat{T}_{2,j}(\hat{e}_s,\hat{e}_s) \right\rangle - \|\hat{T}_{2,j}(\hat{e}_r,\hat{e}_s)\|^2 \right ],
     \]
     where $(\hat{e}_r)_{1 \leq r \leq d}$ is an orthonormal basis of $\hat{T}_{2,j}$.  Theorems \ref{thm:upper_bound_tangent} and \ref{thm:upper_bound_curvature} yield 
     \[
     \E_{P^{\otimes n}}
\max_{1 \leq j \leq n} 
\left | \widehat{Sc}_j - Sc_{Y_j}^M \right |
\leq 
C
\left (
\frac{\log n}{n-1}
\right)^{\frac{k-2}{d}}\left \{ 1 \vee \sigma \left (
\frac{\log n}{n-1}
\right)^{-\frac{k}{d}}\right\},
\]
where $C = C_{d,k,\tau_{min},\mathbf{L},f_{min},f_{max}}$.    

The (near-)optimality of the bound stated in Theorem \ref{thm:upper_bound_curvature} is assessed by the following lower bound.

\begin{thm}\label{thm:lower_bound_curvature}
If $\tau_{min}{L_\perp},\ldots,\tau_{min}^{k-1}{L_k},({\tau_{min}^d}f_{min})^{-1}$ and ${\tau_{min}^d}{f_{max}}$ are large enough (depending only on $d$ and $k$), then
\begin{multline*}
\inf_{\widehat{II}} 
\sup_{P \in \mathcal{P}^k(\sigma)}
\E_{P^{\otimes n}}
\norm{
II_{\pi_M(X_1)}^M
\circ
\pi_{T_{\pi_M(X_1)} M}
-
\widehat{II}
}_{op} \\
\geq 
c_{d,k,\tau_{min}} 
\left\{
\left(\frac{1}{n-1}\right)^{\frac{k-2}{d}}  \right .
\vee
\left .
\left(\frac{\sigma}{n-1}\right)^{\frac{k-2}{d+k}}
\right\},
\end{multline*}
where the infimum is taken over all the estimators $\widehat{II} = \widehat{II}(X_1,\ldots,X_n)$.
\end{thm}
The proof of Theorem \ref{thm:lower_bound_curvature} is given in Section \ref{subsubsec:proofs_lower_bounds_tangent_curvature}. The same remarks as in Section \ref{subsec:main_results_tangent_space} hold. 
If the estimation problem consists in approximating $II_{y}^M$ at a fixed point $y$ known to belong to $M$ beforehand, we obtain the same rates. 
The ambient dimension $D$ still plays no role. 
The shift $k-2$ in the rate of convergence on a $\mathcal{C}^k$-model can be interpreted as the order of derivation of the object of interest, that is $2$ for curvature.

Notice that the lower bound (Theorem \ref{thm:lower_bound_curvature}) does not require $k\geq 3$. Hence, we get that for $k=2$, curvature cannot be estimated uniformly consistently on the $\mathcal{C}^2$-model $\mathcal{P}^2$. 
This seems natural, since the estimation of a second order quantity should require an additional degree of smoothness.

\subsection{Support Estimation}
\label{subsec:main_results_hausdorff}
For each $1 \leq j \leq n$, the minimization \eqref{estimator_full_tensors_definition0} outputs a series of tensors $(\hat{\pi}_{j},\hat{T}_{2,j},\ldots,\hat{T}_{k-1,j})$. This collection of multidimensional monomials can be further exploited as follows. By construction, they fit $M$ at scale $h$ around $Y_j$, so that
\begin{align*}
\widehat{\Psi}_j(v)
&=
X_j + v + \sum_{i=2}^{k-1} \hat{T}_{i,j}\left(\tens{v}{i}\right)
\end{align*}
is a good candidate for an approximate parametrization in a neighborhood of $Y_j$. We do not know the domain $T_{Y_j} M$ of the initial parametrization, though we have at hand an approximation $\hat{T}_j = \image \hat{\pi}_j$ which was proved to be consistent in Section \ref{subsec:main_results_tangent_space}. As a consequence, we let the support estimator based on local polynomials $\hat{M}$ be
\begin{align*}
\hat{M} 
&= 
\bigcup_{j=1}^{n} \widehat{\Psi}_j \left(\mathcal{B}_{\hat{T}_j}(0,{7h/8}) \right).
\end{align*}
\begin{figure}
\centering
\includegraphics[width=0.5\textwidth]{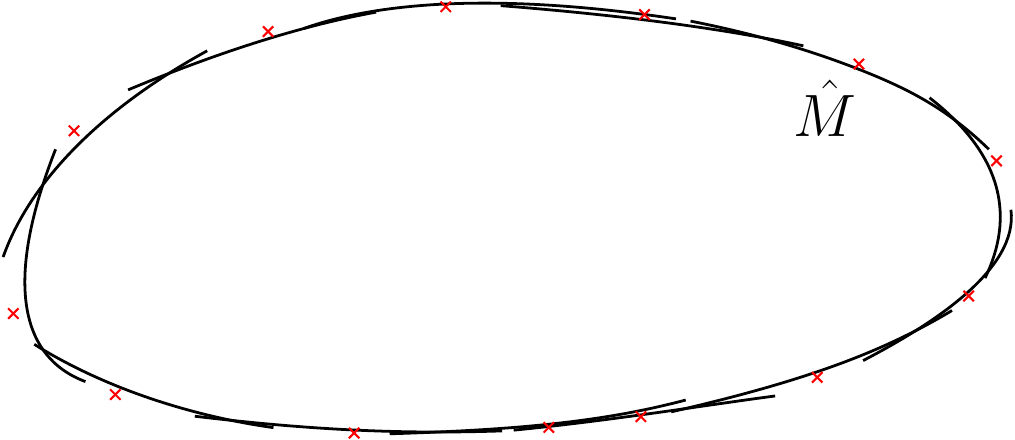}
\caption{ $\hat{M}$ is a union of polynomial patches at sample points.}
\label{fig:patches}
\end{figure}
The set $\hat{M}$ has no reason to be globally smooth, since it consists of a mere union of polynomial patches (Figure \ref{fig:patches}). However, $\hat{M}$ is provably close to $M$ for the Hausdorff distance.
\begin{thm}\label{thm:upper_bound_hausdorff}
With the same assumptions as Theorem \ref{thm:upper_bound_curvature}, with probability at least $1-2\left(\frac{1}{n}\right)^{\frac{k}{d}}$, we have
\begin{align*}
d_{H}\bigl(M,\hat{M}\bigr) \leq C_{d,k,\tau_{min},\mathbf{L},f_{min},f_{max}}(h^{k}\vee \sigma).
\end{align*} 
In particular, for $n$ large enough,
\begin{align*}
\sup_{P \in \mathcal{P}^k(\sigma)}
\E_{P^{\otimes n}}
d_{H}\bigl(M,\hat{M}\bigr)
\leq 
C_{d,k,\tau_{min},\mathbf{L},f_{min},f_{max}} \left\{
\left(
\frac{\log n }{n-1}
\right)^{\frac{k}{d}} \vee \sigma \right\}.
\end{align*}
\end{thm}

A proof of Theorem \ref{thm:upper_bound_hausdorff} is given in Section \ref{subsubsec:hausdorffupper}. As in Theorem \ref{thm:upper_bound_tangent}, for a noise level of order $h^{\alpha}$, $\alpha \geq 1$, Theorem \ref{thm:upper_bound_hausdorff} yields a convergence rate of order $h^{(k \wedge \alpha)/d}$. Thus the noise level $\sigma$ may also be thought of as a regularity threshold. 
Contrary to  \cite[Theorem 2]{Genovese11}, the case $h/4 <\sigma < \tau_{min}$ is not in the scope of Theorem \ref{thm:upper_bound_hausdorff}. Moreover, for $1 \leq \alpha  < 2d/(d+2)$, \cite[Theorem  2]{Genovese11} provides a better convergence rate of $h^{2/(d+2)}$. Note however that Theorem \ref{thm:upper_bound_hausdorff} is also valid whenever the assumption $\mathbb{E}(Z|Y)=0$ is relaxed. In this non-centered noise framework, Theorem \ref{thm:upper_bound_hausdorff} outperforms \cite[Theorem 7]{Maggioni16} in the case $d \geq 3$, $k=2$, and $\sigma \leq h^2$.

In the noise-free case or when $\sigma \leq h^k$, for $k=2$, we recover the rate $(\log n/n)^{2/d}$ obtained in \cite{AamariLevrard2015,Genovese12,Kim2015} and improve the rate $(\log n/n)^{2/(d+2)}$ in \cite{Genovese11,Maggioni16}. However, our estimator $\hat{M}$ is an unstructured union of $d$-dimensional balls in $\R^D$. Consequently, $\hat{M}$ does not recover the topology of $M$ as the estimator of \cite{AamariLevrard2015} does.

When $k\geq 3$, $\hat{M}$ outperforms reconstruction procedures based on a somewhat piecewise linear interpolation \cite{AamariLevrard2015,Genovese12,Maggioni16}, and achieves the faster rate $(\log n / n)^{k/d}$ for the Hausdorff loss. This seems quite natural, since our procedure fits higher order terms. This is done at the price of a probably worse dependency on the dimension $d$ than in \cite{AamariLevrard2015,Genovese12}.
Theorem \ref{thm:upper_bound_hausdorff} is now proved to be (almost) minimax optimal.
\begin{thm}\label{thm:lower_bound_hausdorff}
If $\tau_{min}{L_\perp},\ldots,\tau_{min}^{k-1}{L_k},({\tau_{min}^d}f_{min})^{-1}$ and ${\tau_{min}^d}{f_{max}}$ are large enough (depending only on $d$ and $k$), then for $n$ large enough,
\begin{align*}
\inf_{\hat{M}} 
\sup_{P \in \mathcal{P}^k(\sigma)}
\E_{P^{\otimes n}}
d_H
\bigr(
M,
\hat{M}
\bigr)
\geq 
c_{d,k,\tau_{min}} 
\left\{
\left(\frac{1}{n}\right)^{\frac{k}{d}}
\vee
\left(\frac{\sigma}{n}\right)^{\frac{k}{d+k}}
\right\}
,
\end{align*}
where the infimum is taken over all the estimators $\hat{M} = \hat{M}(X_1,\ldots,X_n)$.
\end{thm}

Theorem \ref{thm:lower_bound_hausdorff}, whose proof is given in Section \ref{subsubsec:proofs_lower_bounds_hausdorff}, is obtained from Le Cam's Lemma (Theorem \ref{lecam}). 
Let us note that it is likely for the extra $\log n$ term appearing in Theorem \ref{thm:upper_bound_hausdorff} to actually be present in the minimax rate. Roughly, it is due to the fact that the Hausdorff distance $d_H$ is similar to a $L^\infty$ loss.
The $\log n$ term may be obtained in Theorem \ref{thm:lower_bound_hausdorff} with the same combinatorial analysis as in \cite{Kim2015} for $k=2$.

As for the estimation of tangent spaces and curvature, Theorem \ref{thm:lower_bound_hausdorff} matches the upper bound in Theorem \ref{thm:upper_bound_hausdorff} in the noise-free case $\sigma \lesssim (1/n)^{k/d}$. Moreover, for $\sigma < \tau_{min}$, it also generalizes Theorem 1 in \cite{Genovese11} to higher orders of regularity ($k \geq 3$). Again, for $\sigma \gg (1/n)^{-k/d}$, the upper bound in Theorem \ref{thm:upper_bound_hausdorff} is larger than the lower bound stated in Theorem \ref{thm:lower_bound_hausdorff}. However our estimator $\hat{M}$ achieves the same convergence rate if the assumption $\mathbb{E}(Z|Y)$ is dropped. 

\section{Conclusion, Prospects}
\label{sec:conclusion}
In this article, we derived non-asymptotic bounds for inference of geometric objects associated with smooth submanifolds $M \subset \R^D$. We focused on tangent spaces, second fundamental forms, and the submanifold itself.
We introduced new regularity classes $\mathcal{C}^k_{\tau_{min},\mathbf{L}}$ for submanifolds that extend the case $k=2$. 
For each object of interest, the proposed estimator relies on local polynomials that can be computed through a least square minimization.
Minimax lower bounds were presented, matching the upper bounds up to $\log n$ factors in the regime of small noise.

The implementation of \eqref{estimator_full_tensors_definition0} needs to be investigated. The non-convexity of the criterion comes from that we minimize over the space of orthogonal projectors, which is non-convex. However, that space is pretty well understood, and it seems possible to implement gradient descents on it \cite{Usevich14}.
Another way to improve our procedure could be to fit orthogonal polynomials instead of monomials. 
Such a modification may also lead to improved dependency on the dimension $d$ and the regularity $k$ in the bounds for both tangent space and support estimation. 

Though the stated lower bounds are valid for quite general tubular noise levels $\sigma$, it seems that our estimators based on local polynomials are suboptimal whenever $\sigma$ is larger than the expected precision for $\mathcal{C}^k$ models in a $d$-dimensional space (roughly $(1/n)^{k/d}$). In such a setting, it is likely that a preliminary centering procedure is needed, as the one exposed in \cite{Genovese11}. Other pre-processings of the data might adapt our estimators to other types of noise. For instance, whenevever outliers are allowed in the model $\mathcal{C}^2$, \cite{AamariLevrard2015} proposes an iterative denoising procedure based on tangent space estimation. 
It exploits the fact that tangent space estimation allows to remove a part of outliers, and removing outliers enhances tangent space estimation.
An interesting question would be to study how this method can apply with local polynomials.

Another open question is that of exact topology recovery with fast rates for $k \geq 3$. Indeed, $\hat{M}$ converges at rate $(\log n / n)^{k/d}$ but is unstructured.
It would be nice to glue the patches of $\hat{M}$ together, for example using interpolation techniques, following the ideas of \cite{Fefferman15}.


\section{Proofs}
\label{sec:proofs}

\subsection{Upper bounds}
\label{subsec:proofs_l_polynomials}
\subsubsection{Preliminary results on polynomial expansions}

To prove Theorem \ref{thm:upper_bound_tangent}, \ref{thm:upper_bound_curvature} and \ref{thm:upper_bound_hausdorff}, the following lemmas are needed.
First, we relate the existence of parametrizations $\Psi_p$'s mentioned in Definition \ref{best_model_definition} to a local polynomial decomposition.
      
\begin{lem}\label{lem:fourre_tout_geom_model}
For any $M \in \mathcal{C}^{k}_{\tau_{min},\mathbf{L}}$ and $y\in M$, the following holds.

\begin{enumerate}[(i)]

\item For all $v_1,v_2\in \mathcal{B}_{T_yM}\left(0,\frac{1}{4L_\perp}\right)$,
\begin{align*}
\frac{3}{4}\norm{v_2-v_1} \leq \norm{\Psi_y(v_2)-\Psi_y(v_1)} \leq \frac{5}{4} \norm{v_2-v_1}
.
\end{align*}

\item For all $h \leq \frac{1}{4L_\perp} \wedge \frac{2 \tau_{min}}{5}$, 
\begin{align*}
M \cap \mathcal{B}\left(y,\frac{3h}{5}\right)
\subset
\Psi_y \left( \mathcal{B}_{T_y M} \left(y,h\right) \right)
\subset
M \cap \mathcal{B}\left(y,\frac{5h}{4}\right)
.
\end{align*}

\item For all $h \leq \frac{\tau_{min}}{2}$, 
\begin{align*}
\mathcal{B}_{T_y M}\left(0,\frac{7h}{8}\right) \subset \pi_{T_y M}\left(\mathcal{B}(y,h) \cap M\right).
\end{align*}

\item
Denoting by $\pi^\ast = \pi_{T_y M}$ the orthogonal projection onto $T_y M$, for all $y \in M$, there exist multilinear maps $T_2^\ast,\ldots,T_{k-1}^\ast$ from $T_y M$ to $\R^D$, and $R_k$ such that for all $y' \in  \mathcal{B}\left(y,\frac{\tau_{min} \wedge L_\perp^{-1}}{4}\right)\cap M$,
\begin{align*}
y'-y = \pi^*&(y'-y) + T_2^*( \tens{\pi^*(y'-y)}{2}) + \hdots + T_{k-1}^*(\tens{\pi^*(y'-y)}{k-1}) 
		\\
		&+ R_{k}(y'-y),
\end{align*}
with
\begin{align*}
\norm{R_k (y'-y)} \leq C \norm{y'-y}^{k}
\text{ \quad  and \quad }
\norm{T_i^\ast}_{op} \leq L_i', \text{ for } 2 \leq i \leq k-1,
\end{align*}

where $L'_i$ depends on $d, k, \tau_{min}, L_\perp, \ldots, L_i$, and $C$ on $d$, $k$, $\tau_{min}$, $L_\perp$,$\ldots$,  $L_k$. Moreover, for $k\geq 3$, $T_2^\ast = II^M_y$.

\item For all $y \in M$, $\norm{II^M_y}_{op} \leq 1/\tau_{min}$. In particular, the sectional curvatures of $M$ satisfy 
\begin{align*}
\frac{-2}{\tau_{min}^2}
\leq
\kappa
\leq 
\frac{1}{\tau_{min}^2}
.
\end{align*}

\end{enumerate}

\end{lem}

          The proof of Lemma \ref{lem:fourre_tout_geom_model} can be found in Section \ref{subsec:A2}. A direct consequence of Lemma \ref{lem:fourre_tout_geom_model} is the following Lemma \ref{lem:pol_expression}.
          
          \begin{lem}\label{lem:pol_expression}
          Set $h_0= (\tau_{min} \wedge L_{\perp}^{-1})/8$ and $h \leq h_0$. Let $M \in \mathcal{C}^{k}_{\tau_{min},\mathbf{L}}$, $x_0 = y_0 + z_0$, with $y_0 \in M$ and $\|z_0\| \leq \sigma \leq h/4$. Denote by $\pi^*$ the orthogonal projection onto $T_{y_0}M$, and by $T_2^*, \hdots, T_{k-1}^*$ the multilinear maps given by Lemma \ref{lem:fourre_tout_geom_model}, $iv)$.
          
           Then, for any $x = y+z$ such that    $y \in M$, $\|z\| \leq \sigma \leq h/4$ and $x \in \mathcal{B}(x_0,h)$, for any orthogonal projection $\pi$ and multilinear maps $T_2, \hdots, T_{k-1}$, we have 
           \begin{multline*}
        x-x_0 - \pi(x-x_0) - \sum_{j=2}^{k-1}{ T_j(\tens{\pi(x-x_0)}{j})} = \sum_{j=1}^{k}T'_j(\tens{\pi^*(y-y_0)}{j}) \\ + R_k{(x-x_0)},    
\end{multline*}     
where $T'_j$ are $j$-linear maps, and $\|R_k(x-x_0)\| \leq C \left( \sigma \vee h^k \right )(1 + th)$, with $t = \max_{j=2\hdots, k}{\|T\|_{op}}$ and $C$ depending on $d$, $k$, $\tau_{min}$, $L_\perp$,$\ldots$,  $L_k$. Moreover, we have              
\[
    \begin{array}{@{}ccc}
    T'_1  &=& (\pi^*-{\pi}),\\
    T'_2  &=& (\pi^*-{\pi})\circ T_2^* + (T_2^* \circ \pi^* - {T}_2 \circ {\pi}),
    \end{array}
    \]
    and, if $\pi = \pi^*$ and $T_i = T_i^*$, for $i=2, \hdots, k-1$, then $T'_j=0$, for $j=1, \hdots, k$.
          \end{lem}
          Lemma \ref{lem:pol_expression} roughly states that, if $\pi$, $T_j$, $j \geq 2$ are designed to locally approximate $x=y+z$ around $x_0=y_0 + z_0$, then the approximation error may be expressed as a polynomial expansion in $\pi^*(y-y_0)$. 
          
          \begin{proof}[Proof of Lemma \ref{lem:pol_expression}]
          For short assume that $y_0 = 0$. In what follows $C$ will denote a constant depending on $d$, $k$, $\tau_{min}$, $L_\perp$,$\ldots$,  $L_k$. We may write
    \begin{multline*}
    x-x_0 - \pi(x-x_0) - \sum_{j=2}^{k-1}{ T_j(\tens{\pi(x-x_0)}{j})} = y - \pi(y) - \sum_{j=2}^{k-1}{ T_j(\tens{\pi(y)}{j})}\\ + R'_k(x-x_0),
\end{multline*}              
          with $\|R'_k(x-x_0)\| \leq C \sigma(1+ t h)$. Since $\sigma \leq h/4$, $y \in \mathcal{B}(0,3h/2)$, with $h \leq h_0$. Hence Lemma \ref{lem:fourre_tout_geom_model} entails  
          \begin{align*}
y = \pi^*&(y) + T_2^*( \tens{\pi^*(y)}{2}) + \hdots + T_{k-1}^*(\tens{\pi^*(y)}{k-1}) 
		\\
		&+ R''_{k}(y),
\end{align*}
           with $\| R''_{k}(y) \| \leq C h^k$. We deduce that
          \begin{multline*}
          y - \pi(y) - \sum_{j=2}^{k-1}{ T_j(\tens{\pi(y)}{j})} = (\pi^*-\pi \circ \pi^*)(y) + 
          T_2^*(\tens{\pi^*(y)}{2}) - \pi(T_2^*(\tens{\pi^*(y)}{2})) \\ - T_2(\tens{\pi \circ\pi^*(y)}{2}) + \sum_{j=3}^{k}{ T'_k(\tens{\pi^*(y)}{j})} - \pi(R''_k(y)) - R'''_k(y),
\end{multline*}           
with $\|R'''_k(y)\| \leq C t h^{k+1}$, since only tensors of order greater than $2$ are involved in $R'''_k$. Since $T_2^* = II^M_0$, $\pi^* \circ T_2^* =0$, hence the result.
          \end{proof}
          
          At last, we need a result relating deviation in terms of polynomial norm and $L^2(P_{0,n-1}^{(j)})$ norm, where $P_0 \in \mathcal{P}^k$, for polynomials taking arguments in $\pi^{*,(j)}(y)$.
 For clarity's sake, the bounds are given for $j=1$, and we denote $ P_{0,n-1}^{(1)}$ by $P_{0,n-1}$.
Without loss of generality, we can assume that $Y_1=0$.

Let $\mathbb{R}^k[y_{1:d}]$ denote the set of real-valued polynomial functions in $d$ variables with degree less than $k$. For $S \in \mathbb{R}^k[y_{1:d}]$, we denote by $\|S\|_2$ the Euclidean norm of its coefficients, and by $S_h$ the polynomial defined by $S_h(y_{1:d}) = S(hy_{1:d})$. With a slight abuse of notation, $S(\pi^*(y))$ will denote $S(e_1^*(\pi^*(y)), \hdots, e_d^*(\pi^*(y)))$, where $e_1^*, \hdots, e_d^*$ form an orthonormal coordinate system of $T_0M$. 

\begin{prop}\label{polminoration}
      Set $h= \left ( K \frac{\log n }{n-1} \right )^{\frac{1}{d}}$. There exist constants $\kappa_{k,d}$, $c_{k,d}$ and $C_d$ such that, if $K \geq (\kappa_{k,d} f_{max}^2 /f_{min}^3)$ and $n$ is large enough so that $h \leq h_0 \leq \tau_{min}/8$, then with probability at least $1-\left ( \frac{1}{n} \right )^{\frac{k}{d}+1}$, we have
      \[
      \begin{array}{@{}ccc}
      P_{0,n-1} [S^2(\pi^*(y)) \mathbbm{1}_{\mathcal{B}(h/2)}(y)] &\geq& c_{k,d} h^d f_{min} \|S_h\|_2^2, \\
      N(3h/2) & \leq & C_d f_{max} (n-1) h^d,
      \end{array}
      \]
      for every $S \in \R^{k}[y_{1:d}]$, where $N(3h/2) = \sum_{j = 2}^n \mathbbm{1}_{\mathcal{B}(0,3h/2)}(Y_j)$. 
\end{prop}

The proof of Proposition \ref{polminoration} is deferred to Section \ref{subsec:B2}. 
\subsubsection{Upper Bound for Tangent Space Estimation}\label{sec_proofofthmupper_bound_tangent}
\begin{proof}[Proof of Theorem \ref{thm:upper_bound_tangent}]

We recall that for every $j=1, \hdots, n$, $X_j = Y_j + Z_j$, where $Y_j \in M$ is drawn from $P_0$ and $\|Z_j\|\leq \sigma \leq h/4$, where $h \leq h_0$ as defined in Lemma \ref{lem:pol_expression}. Without loss of generality we consider the case $j=1$, $Y_1=0$. From now on we assume that the probability event defined in Proposition \ref{polminoration} occurs, and denote by $\mathcal{R}_{n-1}(\pi,T_{2},\hdots, T_{k-1})$ the empirical criterion defined by \eqref{estimator_full_tensors_definition0}. Note that $X_j \in \mathcal{B}(X_1,h)$ entails $Y_j \in \mathcal{B}(0,3h/2)$. Moreover, since for $t \geq \max_{i=2, \hdots, k-1}\|T^*_{i}\|_{op}$, $
\mathcal{R}_{n-1}(\hat{\pi},\hat{T}_1, \hdots, \hat{T}_{k-1}) \leq \mathcal{R}_{n-1}(\pi^*,T^*_2, \hdots, T_{k-1}^*)$, 
we deduce that \[
\mathcal{R}_{n-1}(\hat{\pi},\hat{T}_1, \hdots, \hat{T}_{k-1}) \leq \frac{C_{\tau_{min},\mathbf{L}} \left( \sigma^2 \vee h^{2k} \right )(1 + th)^2 N(3h/2)}{n-1} , 
\] according to Lemma \ref{lem:pol_expression}.
On the other hand, note that if $Y_j \in \mathcal{B}(0,h/2)$, then $X_j \in \mathcal{B}(X_1,h)$. Lemma \ref{lem:pol_expression} then yields 
\begin{align*}
\mathcal{R}_{n-1}(\hat{\pi},\hat{T}_2, \hdots, \hat{T}_{k-1})  \geq & P_{0,n-1} \left (\left \| \sum_{j=1}^{k}\hat{T}'_j(\tens{\pi^*(y)}{j}) \right \|^2\mathbbm{1}_{\mathcal{B}(0,h/2)}(y) \right ) \\ & - \frac{C_{\tau_{min},\mathbf{L}} \left( \sigma^2 \vee h^{2k} \right )(1 + th)^2 N(3h/2)}{n-1}.
\end{align*} 
Using Proposition \ref{polminoration}, we can decompose the right-hand side as
\begin{multline*}
\sum_{r=1}^D P_{0,n-1} \left ( \sum_{j=1}^{k}{\hat{T}'{}^{(r)}_j}(\tens{\pi^*(y)}{j})\mathbbm{1}_{\mathcal{B}(0,h/2)}(y) \right )^2 \\ \leq C_{\tau_{min},\mathbf{L}} f_{max} h^d \left( \sigma^2 \vee h^{2k} \right )(1 + th)^2, 
\end{multline*}
where for any tensor $T$, $T^{(r)}$ denotes the $r$-th coordinate of $T$ and is considered as a real valued $r$-order polynomial.
Then, applying Proposition \ref{polminoration} to each coordinate leads to
\begin{align*}
c_{d,k} f_{min} \sum_{r=1}^{D}\sum_{j=1}^{k} \left \| \left ( T'{}^{(r)}_j(\tens{\pi^*(y)}{j}) \right )_h \right \|^2_2   \leq C_{\tau_{min},\mathbf{L}} f_{max} h^d \left( \sigma^2 \vee h^{2k} \right )(1 + th)^2.
\end{align*}
It follows that, for $1\leq j \leq k$, 
\begin{align}
        \| \hat{T}'_j \circ \pi^* \|^2_{op}  &\leq C_{d,k,\mathbf{L},\tau_{min}} \frac{f_{max}}{f_{min}}(h^{2(k-j)}\vee \sigma^2 h^{-2j})(1 + t^2 h^2). \label{normopmajoration}
\end{align}
Noting that, according to \cite[Section 2.6.2]{Golub96}, 
  \[\|\hat{T}'_1 \circ \pi^*\|_{op}
  = 
  \|(\pi^*-\hat{\pi})\pi^*\|_{op}
  =
  \|\pi_{\hat{T}_1^\perp} \circ \pi^\ast\|
  =
  \angle(T_{Y_1} M, \hat{T}_1),
  \]
we deduce that 
\[
\angle 
\bigl(T_{Y_1} M, \hat{T}_1\bigr) \leq C_{d,k,\mathbf{L},\tau_{min}} \sqrt{\frac{f_{max}}{f_{min}}}(h^{(k-1)}\vee \sigma h^{-1})(1 + t h).
\]
Theorem \ref{thm:upper_bound_tangent} then follows from a straightforward union bound.
\end{proof}

\subsubsection{Upper Bound for Curvature Estimation}\label{subsubsec:proofofuppercurvature}
\begin{proof}[{Proof of Theorem \ref{thm:upper_bound_curvature}}]
         Without loss of generality, the derivation is conducted in the same framework as in the previous Section \ref{sec_proofofthmupper_bound_tangent}. In accordance with assumptions of Theorem \ref{thm:upper_bound_curvature}, we assume that $\max_{2\leq i \leq k} \|T_i^*\|_{op} \leq t \leq 1/h$.         
         Since, according to Lemma \ref{lem:pol_expression}, 
         \[
         T'_2 (\tens{\pi^*(y)}{2}) = (\pi^*-\hat{\pi})(T_2^*(\tens{\pi^*(y)}{2})) + (T_2^* \circ \pi^* - \hat{T}_2 \circ \hat{\pi})(\tens{\pi^*(y)}{2}),
        \] 
        we deduce that 
        \begin{multline*}
        \| T_2^* \circ \pi^* - \hat{T}_2 \circ \hat{\pi}\|_{op} \leq \|T'_2 \circ \pi^*\|_{op} + \|\hat{\pi} - \pi^*\|_{op} 
        + \| \hat{T}_2 \circ \hat{\pi} \circ \pi^* - \hat{T}_2 \circ \hat{\pi} \circ \hat{\pi} \|_{op}.
        \end{multline*}
        Using \eqref{normopmajoration} with $j=1,2$ and $th \leq 1$ leads to
        \begin{align*}
        \| T_2^* \circ \pi^* - \hat{T}_2 \circ \hat{\pi}\|_{op} \leq C_{d,k,\mathbf{L},\tau_{min}} \sqrt{\frac{f_{max}}{f_{min}}}(h^{(k-2)}\vee \sigma h^{-2}).
        \end{align*}
       Finally, Lemma \ref{lem:fourre_tout_geom_model} states that $II_{Y_1}^M = T_2^*$. Theorem \ref{thm:upper_bound_curvature} follows from a union bound. 
\end{proof}        

\subsubsection{Upper Bound for Manifold Estimation}\label{subsubsec:hausdorffupper}
\begin{proof}[{Proof of Theorem \ref{thm:upper_bound_hausdorff}} 
]
Recall that we take $X_i = Y_i + Z_i$, where $Y_i$ has distribution $P_0$ and $\|Z_j\| \leq \sigma \leq h/4$. We also assume that the probability events of Proposition \ref{polminoration} occur simultaneously at each $Y_i$, so that \eqref{normopmajoration} holds for all $i$, with probability larger than $1-(1/n)^{k/d}$. Without loss of generality set $Y_1=0$.  Let $v \in \mathcal{B}_{\hat{T}_1 M}(0,7h/8)$ be fixed. Notice that $\pi^* (v) \in \mathcal{B}_{T_0 M}(0,7h/8)$. Hence, according to  Lemma \ref{lem:fourre_tout_geom_model}, there exists $y \in \mathcal{B}(0,h) \cap M$ such that $\pi^*(v) = \pi^* (y)$. According to \eqref{normopmajoration}, we may write
        \begin{align*}
    \widehat{\Psi}(v) = Z_1 + v + \sum_{j=2}^{k-1} \hat{T}_j(\tens{v}{j}) & = \pi^*(v) + \sum_{j=2}^{k-1} \hat{T}_j(\tens{\pi^*(v)}{j}) + R_{k}(v),    
        \end{align*}
        where, since $\|\hat{T}_j\|_{op} \leq 1/h$, $\|R_k(v)\| \leq C_{k,d,\tau_{min},\mathbf{L}} \sqrt{f_{max}/f_{min}} (h^{k}\vee \sigma)$. Using \eqref{normopmajoration} again leads to 
        \begin{align*}
        \pi^*(v) + \sum_{j=2}^{k-1} \hat{T}_j(\tens{\pi^* (v)}{j}) & = \pi^* (v) + \sum_{j=2}^{k-1} T^*_j(\tens{\pi^*( v)}{j}) + R'(\pi^*(v)) \\
        &= \pi^*(y) + \sum_{j=2}^{k-1} T^*_j(\tens{\pi^* (y)}{j}) + R'(\pi^*(y)),
        \end{align*}
        where $\|R'(\pi^*(y))\| \leq C_{k,d,\tau_{min},\mathbf{L}} \sqrt{f_{max}/f_{min}}(h^{k}\vee \sigma)$. According to Lemma \ref{lem:fourre_tout_geom_model}, we deduce that $\|\widehat{\Psi}(v) - y \| \leq C_{k,d,\tau_{min},\mathbf{L}} \sqrt{f_{max}/f_{min}} (h^{k}\vee \sigma)$, hence
        \begin{align}\label{hausdorffsens1}
        \sup_{u \in \hat{M}} d(u,M) \leq  C_{k,d,\tau_{min},\mathbf{L}} \sqrt{\frac{f_{max}}{f_{min}}} (h^{k}\vee \sigma).
        \end{align}        
Now we focus on $\sup_{y \in M} d(y,\hat{M})$. For this, we need a lemma ensuring that $\Y_n = \{Y_1, \hdots, Y_n\}$ covers $M$ with high probability. 
       \begin{lem}\label{lem:volume+hausdorff_density}
%
Let $h = \left( \frac{C'_d k}{f_{min}}  \frac{\log n}{n}\right)^{1/d}$ with $C'_d$ large enough. Then for $n$ large enough so that $h \leq \tau_{min}/4$, with probability at least $1 - \left( \frac{1}{n}\right)^{k/d}$,
\begin{align*}
d_H\left(M,\Y_n\right)
\leq
h/2.
\end{align*}
\end{lem}
The proof of Lemma \ref{lem:volume+hausdorff_density} is given in Section \ref{subsec:B1}.
Now we choose $h$ satisfying the conditions of Proposition \ref{polminoration} and Lemma \ref{lem:volume+hausdorff_density}. Let $y$ be in $M$ and assume that $\norm{y - Y_{j_0}} \leq h/2$. Then $y \in \mathcal{B}(X_{j_0},3h/4)$. According to Lemma \ref{lem:pol_expression} and \eqref{normopmajoration}, we deduce that $\| \widehat{\Psi}_{j_0}(\hat{\pi}_{j_0}(y-X_{j_0})) - y \| \leq C_{k,d,\tau_{min},\mathbf{L}} \sqrt{f_{max}/f_{min}} (h^{k}\vee \sigma)$. Hence, from Lemma \ref{lem:volume+hausdorff_density},
        \begin{align}\label{hausdorff_sens2}
        \sup_{y \in M}d(y,\hat{M}) \leq C_{k,d,\tau_M,\mathbf{L}}\sqrt{\frac{f_{max}}{f_{min}}} (h^{k}\vee \sigma)
        \end{align} 
with probability at least $1 - 2\left(\frac{1}{n}\right)^{k/d}$.
Combining \eqref{hausdorffsens1} and \eqref{hausdorff_sens2} gives Theorem \ref{thm:upper_bound_hausdorff}.
\end{proof}
\subsection{Minimax Lower Bounds}
\label{subsec:proofs_lower_bounds}

This section is devoted to describe the main ideas of the proofs of the minimax lower bounds. 
We prove Theorem \ref{thm:lower_bound_hausdorff} on one side, and
Theorem \ref{thm:lower_bound_tangent} and Theorem \ref{thm:lower_bound_curvature} in a unified way on the other side.
The methods used rely on hypothesis comparison \cite{Yu97}.

\subsubsection{Lower Bound for Manifold Estimation}
\label{subsubsec:proofs_lower_bounds_hausdorff}

We recall that for two distributions $Q$ and $Q'$ defined on the same space, 
the $L^1$ test affinity $\norm{Q\wedge Q'}_1$ is given by
\begin{align*}
\norm{Q\wedge Q'}_1 = \int dQ \wedge dQ',
\end{align*}
where $dQ$ and $dQ'$ denote densities of $Q$ and $Q'$ with respect to any dominating measure.


The first technique we use, involving only two hypotheses, is usually referred to as Le Cam's Lemma \cite{Yu97}.
Let $\mathcal{P}$ be a model and $\theta(P)$ be the parameter of interest. Assume that $\theta(P)$ belongs to a pseudo-metric space $(\mathcal{D},d)$, that is $d(\cdot,\cdot)$ is symmetric and satisfies the triangle inequality.
Le Cam's Lemma can be adapted to our framework as follows.
\begin{thm}[Le Cam's Lemma \cite{Yu97}]\label{lecam}
For all pairs $P,P'$ in $\mathcal{P}$,
\begin{align*}
\inf_{ \hat{\theta}}
\sup_{P \in \mathcal{P}}
\E_{P^{\otimes n}}
d( \theta(P) , \hat{\theta} )
\geq
\frac{1}{2} 
d\left( \theta(P) , \theta(P') \right) 
\norm{P \wedge P'}_1^n
,
\end{align*}
where the infimum is taken over all the estimators $\hat{\theta} = \hat{\theta}(X_1,\ldots,X_n)$.

\end{thm}
In this section, we will get interested in $\mathcal{P} = \mathcal{P}^k(\sigma)$ and $\theta(P) = M$, with $d  =d_H$.
In order to derive Theorem \ref{thm:lower_bound_hausdorff}, we build two different pairs $(P_0,P_1)$, $(P_0^\sigma,P_1^\sigma)$ of hypotheses in the model $\mathcal{P}^k(\sigma)$. Each pair will exploit a different property of the model $\mathcal{P}^k(\sigma)$.

The first pair $(P_0,P_1)$ of hypotheses (Lemma \ref{lem:hypotheses_hausdorff_nonoise}) is built in the model $\mathcal{P}^k \subset \mathcal{P}^k(\sigma)$, and exploits the geometric difficulty of manifold reconstruction, even if no noise is present. These hypotheses, depicted in Figure \ref{fig:bump_hausdorff}, consist of bumped versions of one another.
\begin{lem}
\label{lem:hypotheses_hausdorff_nonoise}
Under the assumptions of Theorem \ref{thm:lower_bound_hausdorff}, there exist $P_0,P_1 \in \mathcal{P}^k$ with associated submanifolds $M_0,M_1$ such that
\begin{align*}
d_H(M_0,M_1) 
\geq 
c_{k,d,\tau_{min}}
\left(\frac{1}{n}\right)^{\frac{k}{d}}
,
\text{ and}
\quad
\norm{P_0 \wedge P_1}_1^n \geq c_0.
\end{align*}
\end{lem}
The proof of Lemma \ref{lem:hypotheses_hausdorff_nonoise} is to be found in Section \ref{subsec:C4}.1.

\begin{figure}
\includegraphics[width = 0.5\textwidth]{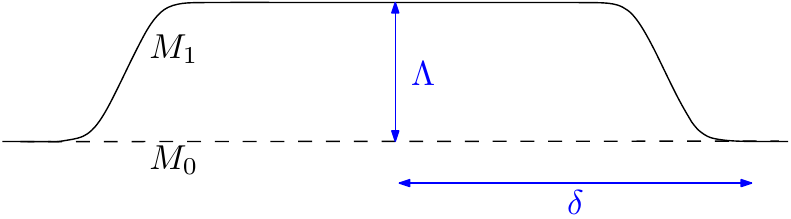}
\caption{Manifolds $M_0$ and $M_1$ of Lemma \ref{lem:hypotheses_hausdorff_nonoise} and Lemma \ref{lem:hypotheses_hausdorff_noise}. 
The width $\delta$ of the bump is chosen to have $\norm{P_0^\sigma \wedge P_1^\sigma}_1^n$ constant.
The distance $\Lambda = d_H(M_0,M_1)$ is of order $\delta^k$ to ensure that $M_1 \in \mathcal{C}^k$.}
\label{fig:bump_hausdorff}
\end{figure}

The second pair $(P_0^\sigma,P_1^\sigma)$ of hypotheses (Lemma \ref{lem:hypotheses_hausdorff_noise}) has a similar construction than $(P_0,P_1)$. 
Roughly speaking, they are the uniform distributions on the offsets of radii $\sigma/2$ of $M_0$ and $M_1$ of Figure \ref{fig:bump_hausdorff}.
Here, the hypotheses are built in $\mathcal{P}^k(\sigma)$, and fully exploit the statistical difficulty of manifold reconstruction induced by noise.
\begin{lem}
\label{lem:hypotheses_hausdorff_noise}
Under the assumptions of Theorem \ref{thm:lower_bound_hausdorff}, there exist $P_0^\sigma,P_1^\sigma \in \mathcal{P}^k(\sigma)$ with associated submanifolds $M_0^\sigma,M_1^\sigma$ such that
\begin{align*}
d_H(M_0^\sigma,M_1^\sigma) 
\geq 
c_{k,d,\tau_{min}}
\left(\frac{\sigma}{n}\right)^{\frac{k}{d+k}}
,
\text{ and}
\quad
\norm{P_0^\sigma \wedge P_1^\sigma}_1^n \geq c_0.
\end{align*}
\end{lem}
The proof of Lemma \ref{lem:hypotheses_hausdorff_noise} is to be found in Section \ref{subsec:C4}.2.
We are now in position to prove Theorem \ref{thm:lower_bound_hausdorff}.
\begin{proof}[Proof of Theorem \ref{thm:lower_bound_hausdorff}]

Let us apply Theorem \ref{lecam} with $\mathcal{P} = \mathcal{P}^k(\sigma)$, $\theta(P) = M$ and $d = d_H$.
Taking $P = P_0$ and $P' = P_1$ of Lemma \ref{lem:hypotheses_hausdorff_nonoise}, these distributions both belong to $\mathcal{P}^k \subset \mathcal{P}^k(\sigma)$, so that Theorem \ref{lecam} yields
\begin{align*}
\inf_{\hat{M}} 
\sup_{P \in \mathcal{P}^k(\sigma)}
\E_{P^{\otimes n}}
d_H
\bigr(
M,
\hat{M}
\bigr)
&\geq
d_H(M_0,M_1) \norm{P_0 \wedge P_1}_1^n
\\
&\geq
c_{k,d,\tau_{min}}
\left(\frac{1}{n}\right)^{\frac{k}{d}} \times c_0.
\end{align*}
Similarly, setting hypotheses $P = P_0^\sigma$ and $P' = P_1^\sigma$ of Lemma \ref{lem:hypotheses_hausdorff_noise} yields
\begin{align*}
\inf_{\hat{M}} 
\sup_{P \in \mathcal{P}^k(\sigma)}
\E_{P^{\otimes n}}
d_H
\bigr(
M,
\hat{M}
\bigr)
&\geq
d_H(M_0^\sigma,M_1^\sigma) \norm{P_0^\sigma \wedge P_1^\sigma}_1^n
\\
&\geq
c_{k,d,\tau_{min}}
\left(\frac{\sigma}{n}\right)^{\frac{k}{k+d}} \times c_0,
\end{align*}
which concludes the proof.
\end{proof}

\subsubsection{Lower Bounds for Tangent Space and Curvature Estimation}
\label{subsubsec:proofs_lower_bounds_tangent_curvature}

Let us now move to the proof of Theorem \ref{thm:lower_bound_tangent} and \ref{thm:lower_bound_curvature}, that consist of lower bounds for the estimation of $T_{X_1} M$ and $II_{X_1}^M$ with random base point $X_1$. In both cases, the loss can be cast as
\begin{align*}
E_{P^{\otimes n}} ~d( \theta_{X_1}(P) , \hat{\theta} )
&=
\E_{P^{\otimes n-1}} 
\left[ 
E_{P} ~d( \theta_{X_1}(P) , \hat{\theta} )
\right]
\\
&=
\E_{P^{\otimes n-1}} \left[ \norm{d \bigl( \theta_{\cdot}(P) , \hat{\theta}\bigr)}_{L^1(P)} \right]
,
\end{align*}
where $\hat{\theta} = \hat{\theta}(X,X')$, with $X = X_1$ driving the parameter of interest, and $X' = (X_2,\ldots,X_n) = X_{2:n}$.
Since $\| . \|_{L^1(P)}$ obviously depends on $P$, the technique exposed in the previous section does not apply anymore.
However, a slight adaptation of Assouad's Lemma \cite{Yu97} with an extra conditioning on $X = X_1$ carries out for our purpose.
Let us now detail a general framework where the method applies. 

We let $\mathcal{X},\mathcal{X}'$ denote measured spaces.
For a probability distribution $Q$ on $\mathcal{X} \times \mathcal{X}'$, we let $(X,X')$ be a random variable with distribution $Q$. The marginals of $Q$ on $\mathcal{X}$ and $\mathcal{X}'$ are denoted by $\mu$ and $\nu$ respectively.
Let $(\mathcal{D},d)$ be a pseudo-metric space.
For $Q \in \mathcal{Q}$, we let $\theta_\cdot(Q): \mathcal{X} \rightarrow \mathcal{D}$ be defined $\mu$-almost surely, where $\mu$ is the marginal distribution of $Q$ on $\mathcal{X}$.
The parameter of interest is $\theta_X(Q)$, and the associated minimax risk over $\mathcal{Q}$ is
\begin{align}
\label{eqn:minimax_risk_random}
\inf_{\hat{\theta}} 
\sup_{Q \in \mathcal{Q}}
\E_Q
\left[
d
	\bigl(
	\theta_X(Q)
	,
	\hat{\theta}(X,X')
	\bigr)
\right]
,
\end{align}
where the infimum is taken over all the estimators $\hat{\theta} : \mathcal{X} \times \mathcal{X}' \rightarrow \mathcal{D}$.

Given a set of probability distributions $\mathcal{Q}$ on $\mathcal{X}\times\mathcal{X}'$, write $\overline{Conv}(\mathcal{Q})$ for the set of mixture probability distributions with components in $\mathcal{Q}$.
For all $\tau = (\tau_1,\ldots,\tau_m) \in \left\{0,1 \right\}^m$, $\tau^k$ denotes the $m$-tuple that differs from $\tau$ only at the $k$th position.
We are now in position to state the conditional version of Assouad's Lemma that allows to lower bound the minimax risk \eqref{eqn:minimax_risk_random}.

\begin{lem}[Conditional Assouad]
\label{conditional_assouad}
Let $m \geq 1$ be an integer and let 
$\left\{ \mathcal{Q}_\tau\right\}_{\tau\in \{0,1\}^m}$ be a family of $2^m$ submodels $\mathcal{Q}_\tau \subset \mathcal{Q}$.
Let $\left\{U_k \times U'_k\right\}_{1 \leq k \leq m}$ be a family of pairwise disjoint subsets of $\mathcal{X} \times \mathcal{X}'$, and $\mathcal{D}_{\tau,k}$ be subsets of $\mathcal{D}$. Assume that
for all $\tau \in \left\{ 0,1 \right\}^m$ and $1 \leq k \leq m$,
\begin{itemize}
\item for all $Q_\tau \in \mathcal{Q}_\tau$, $\theta_X(Q_\tau) \in \mathcal{D}_{\tau,k}$ on the event $\left\{ X \in U_k \right\}$;
\item for all $\theta \in \mathcal{D}_{\tau,k}$ and $\theta' \in \mathcal{D}_{\tau^k,k}$, $d(\theta,\theta')\geq \Delta$.
\end{itemize}
For all $\tau \in \left\{0,1 \right\}^m$, let $\overline{Q}_\tau \in \overline{Conv} (\mathcal{Q}_\tau)$, and write $\bar{\mu}_\tau$ and $\bar{\nu}_\tau$ for the marginal distributions of $\overline{Q}_\tau$ on $\mathcal{X}$ and $\mathcal{X}'$ respectively. Assume that if $(X,X')$ has distribution $\overline{Q}_\tau$, $X$ and $X'$ are independent conditionally on the event $\left\{(X,X')\in U_k \times U'_k \right\}$, and that
\begin{align*}
\min_{\mathclap{\substack{\tau \in \left\{ 0,1 \right\}^m \\ 1 \leq k \leq m}}}
~~
\left\{
\left( 
\int_{U_k} d \bar{\mu}_\tau \wedge d \bar{\mu}_{\tau^k}
\right)
\left( 
\int_{U'_k} d \bar{\nu}_\tau \wedge d \bar{\nu}_{\tau^k}
\right)
\right\}
\geq
1- \alpha.
\end{align*}
Then,
\begin{align*}
\inf_{\hat{\theta}} 
\sup_{Q \in \mathcal{Q}}
\E_Q
\left[
d
	\bigl( 
	\theta_X(Q)
	,
	\hat{\theta}(X,X')
	\bigr)
\right]
&\geq
m
\frac{\Delta}{2} 
(1- \alpha)
,
\end{align*}
where the infimum is taken over all the estimators $\hat{\theta} : \mathcal{X} \times \mathcal{X}' \rightarrow \mathcal{D}$.
\end{lem}

Note that for a model of the form $\mathcal{Q}=\left\{ \delta_{x_0} \otimes P, P \in \mathcal{P} \right\}$ with fixed $x_0 \in \mathcal{X}$, one recovers the classical Assouad's Lemma \cite{Yu97} taking $U_k = \mathcal{X}$ and $U'_k=\mathcal{X}'$. Indeed, when $X=x$ is deterministic, the parameter of interest $\theta_X(Q) = \theta(Q)$ can be seen as non-random.

In this section, we will get interested in $\mathcal{Q} = \mathcal{P}^k(\sigma)^{\otimes n}$, and $\theta_{X}(Q) = \theta_{X_1}(Q)$ being alternatively $T_{X_1} M$ and $II_{X_1}^M$.
Similarly to Section \ref{subsubsec:proofs_lower_bounds_hausdorff}, we build two different families of submodels, each of them will exploit a different kind of difficulty for tangent space and curvature estimation.

{\color{black}
The first family, described in Lemma \ref{lem:hypotheses_glissant_nonoise}, highlights the geometric difficulty of the estimation problems, even when the noise level $\sigma$ is small, or even zero.
Let us emphasize that the estimation error is integrated with respect to the distribution of $X_1$.
Hence, considering mixture hypotheses is natural, since building manifolds with different tangent spaces (or curvature) necessarily leads to distributions that are locally singular. Here, as in Section \ref{subsubsec:proofs_lower_bounds_hausdorff}, the considered hypotheses are composed of bumped manifolds (see Figure \ref{fig:bumps_all}). We defer the proof of Lemma \ref{lem:hypotheses_glissant_nonoise} to Section \ref{subsec:C3}.1.
}

\begin{lem}
\label{lem:hypotheses_glissant_nonoise}

Assume that the conditions of Theorem \ref{thm:lower_bound_tangent} or \ref{thm:lower_bound_curvature} hold. 
Given $i \in \{1,2\}$, there exists a family of $2^m$ submodels 
$\bigl\{ \mathcal{P}^{(i)}_{\tau}\bigr\}_{\tau \in \{0,1\}^m} \subset \mathcal{P}^k$,
together with pairwise disjoint subsets $\{U_k \times U_k'\}_{1 \leq k \leq m}$ of $\R^D \times \bigl(\R^D\bigr)^{n-1}$ such that the following holds for all $\tau \in \{0,1\}^m$ and $1 \leq k \leq m$.

For any distribution $P^{(i)}_{\tau} \in \mathcal{P}^{(i)}_{\tau}$ with support $M^{(i)}_{\tau} = Supp\bigl( P^{(i)}_{\tau} \bigr)$, if $\left(X_1,\ldots,X_n\right)$ has distribution $\bigl(P^{(i)}_{\tau}\bigr)^{\otimes n}$, then on the event $\left\{X_1 \in U_k \right\}$, we have:
\begin{itemize}
\item  if $\tau_k = 0$,
\begin{align*}
T_{X_1} {M^{(i)}_{\tau}}  = \R^d \times \left\{0\right\}^{D-d}
\text{\quad , \quad }
\norm{
II_{X_1}^{M^{(i)}_{\tau}}
\circ
\pi_{T_{X_1} {M^{(i)}_{\tau}}}
}_{op} = 0,
\end{align*}
\item if $\tau_k = 1$,
\begin{itemize}
\item for $i = 1$:
$
\displaystyle
\angle \left(T_{X_1} M^{(1)}_{\tau} , \R^d \times \left\{0\right\}^{D-d} \right)
\geq
c_{k,d,\tau_{min}} \left( \frac{1}{n-1}\right)^{\frac{k-1}{d}}
$,
\item for $i = 2$:
$
\displaystyle
\norm{II_{X_1}^{M^{(2)}_{\tau}}
\circ
\pi_{T_{X_1} {M^{(2)}_{\tau}}}
 }_{op} 
\geq 
c_{k,d,\tau_{min}} \left( \frac{1}{n-1}\right)^{\frac{k-2}{d}}
.
$
\end{itemize}
\end{itemize}
Furthermore, there exists $\bar{Q}^{(i)}_{\tau,n} \in \overline{Conv} \bigl(\bigl( \mathcal{P}^{(i)}_{\tau} \bigr)^{\otimes n} \bigr)$
such that if $\left(Z_1,\ldots,Z_n\right) = \left(Z_1,Z_{2:n}\right)$ has distribution $\bar{Q}^{(i)}_{\tau,n}$, $Z_1$ and $Z_{2:n}$ are independent conditionally on the event $\left\{\left(Z_1,Z_{2:n}\right) \in U_k \times U'_k\right\}$.
The marginal distributions of $\bar{Q}^{(i)}_{\tau,n}$ on $\R^D \times \bigl(\R^D\bigr)^{n-1}$ are $\bar{Q}^{(i)}_{\tau,1}$ and $\bar{Q}^{(i)}_{\tau,n-1}$, and we have
\begin{align*}
\int_{U'_k}
d \bar{Q}^{(i)}_{\tau,n-1} \wedge d \bar{Q}^{(i)}_{\tau^k,n-1} 
\geq
c_0
,
\text{ and}
\quad
m \cdot 
\int_{U_k}
d \bar{Q}^{(i)}_{\tau,1} \wedge d \bar{Q}^{(i)}_{\tau^k,1} 
\geq c_d.
\end{align*}

\end{lem}

\begin{figure}[h]
	\begin{center}
		\subfloat[Flat bump: $\tau_k = 0$.]{
			\includegraphics[width = 0.5\textwidth]{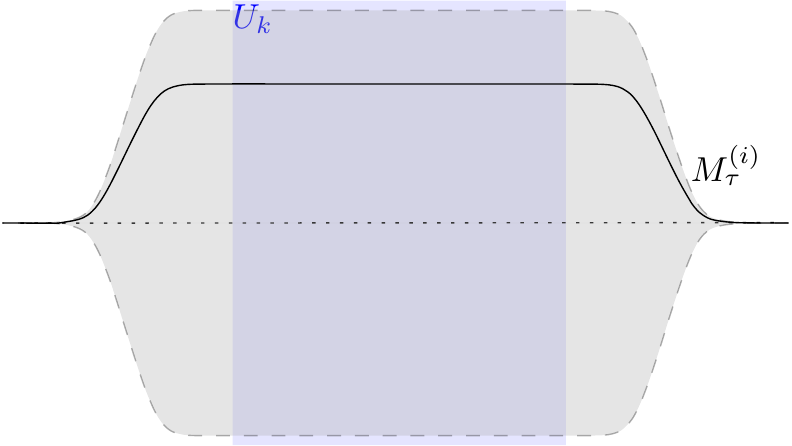}
			\label{subfig:bump_flat}
		}

		\subfloat[Linear bump: $\tau_k = 1$, $i = 1$.]{		
			\includegraphics[width = 0.5\textwidth]{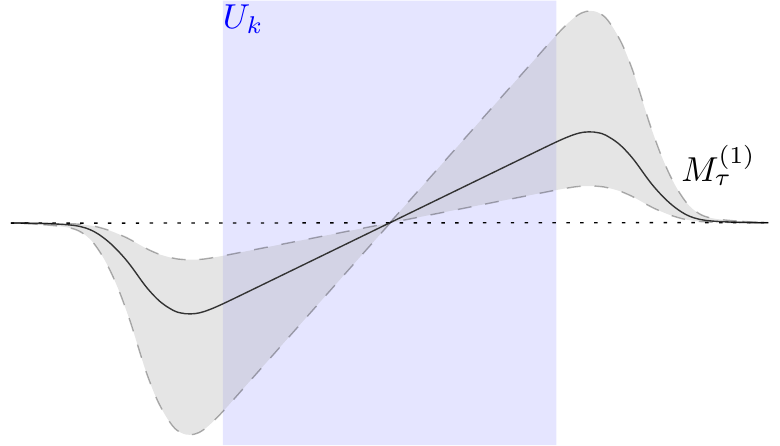}
			\label{subfig:bump_linear}
		}
		\subfloat[Quadratic bump: $\tau_k = 1$, $i = 2$.]{
			\includegraphics[width = 0.5\textwidth]{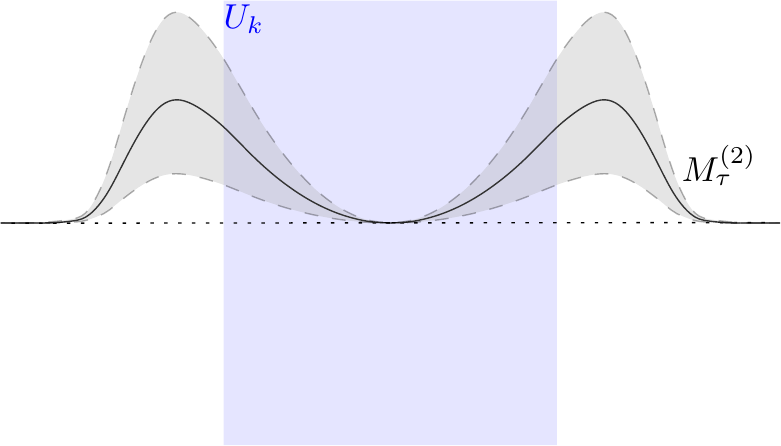}
			\label{subfig:bump_quadratic}
		}

	\end{center}	
	\caption{
		Distributions of Lemma \ref{lem:hypotheses_glissant_nonoise} in the neighborhood of $U_k$ ($1 \leq k \leq m$). Black curves correspond to the support $M_\tau^{(i)}$ of a distribution of $\mathcal{P}_\tau^{(i)} \subset \mathcal{P}^k$. The area shaded in grey depicts the mixture distribution $\bar{Q}_{\tau,1}^{(i)} \in \overline{Conv}\bigl( \mathcal{P}_\tau^{(i)} \bigr)$.
	}
	\label{fig:bumps_all}
\end{figure}

The second family, described in Lemma \ref{lem:hypotheses_glissant_noise}, testifies of the statistical difficulty of the estimation problem when the noise level $\sigma$ is large enough. The construction is very similar to Lemma \ref{lem:hypotheses_glissant_nonoise} (see Figure \ref{fig:bumps_all}).
Though, in this case, the magnitude of the noise drives the statistical difficulty, as opposed to the sampling scale in Lemma \ref{lem:hypotheses_glissant_nonoise}.
Note that in this case, considering mixture distributions is not necessary since the ample-enough noise make bumps that are absolutely continuous with respect to each other.
The proof of Lemma \ref{lem:hypotheses_glissant_noise} can be found in Section \ref{subsec:C3}.2.
\begin{lem}
\label{lem:hypotheses_glissant_noise}
Assume that the conditions of Theorem \ref{thm:lower_bound_tangent} or \ref{thm:lower_bound_curvature} hold, and that $\sigma \geq C_{k,d,\tau_{min}} \left({1}/({n-1)} \right)^{k/d}$ for $C_{k,d,\tau_{min}} > 0$ large enough.
Given $i \in \{1,2\}$, there exists a collection of $2^m$ distributions $\bigl\{\mathbf{P}_\tau^{(i),\sigma}\bigr\}_{\tau \in \{0,1\}^m} \subset \mathcal{P}^k(\sigma)$ with associated submanifolds $\bigl\{M_\tau^{(i),\sigma}\bigr\}_{\tau \in \{0,1\}^m}$, together with pairwise disjoint subsets $\{U^\sigma_k\}_{1 \leq k \leq m}$ of $\R^D$ such that the following holds for all $\tau \in \{0,1\}^m$ and $1 \leq k \leq m$.

If $x \in U_k^\sigma$ and $y = \pi_{M_\tau^{(i),\sigma}} (x)$, we have
\begin{itemize}
\item  if $\tau_k = 0$,
\begin{align*}
T_{y} M_\tau^{(i),\sigma}  = \R^d \times \left\{0\right\}^{D-d}
\text{\quad , \quad }
\norm{
II_{y}^{ M_\tau^{(i),\sigma} }
\circ
\pi_{T_{y} M_\tau^{(i),\sigma}}
}_{op} = 0,
\end{align*}
 
\item if $\tau_k = 1$,
\begin{itemize}
\item for $i = 1$:
$
\displaystyle
\angle 
\left(
T_{y} M_\tau^{(1),\sigma} 
, 
\R^d \times \left\{0\right\}^{D-d} \right) 
\geq c_{k,d,\tau_{min}} \left( \frac{\sigma}{n-1}\right)^{\frac{k-1}{k+d}}$,
\item for $i = 2$:
$
\displaystyle
\norm{
II_{y}^{M_\tau^{(2),\sigma} }
\circ
\pi_{
T_{y} M_\tau^{(2),\sigma} 
}
}_{op} 
\geq c'_{k,d,\tau_{min}} \left( \frac{\sigma}{n-1}\right)^{\frac{k-2}{k+d}}
$. 
\end{itemize}
\end{itemize}
Furthermore,
\begin{align*}
\int_{(\R^D)^{n-1}}
\bigl(\mathbf{P}_{\tau}^{(i),\sigma}\bigr)^{\otimes n-1} 
\wedge 
\bigl(\mathbf{P}_{\tau^k}^{(i),\sigma}\bigr)^{\otimes n-1} 
\geq
c_0,
\text{ and}
\quad
m \cdot \int_{U_k^\sigma}
\mathbf{P}_{\tau}^{(i),\sigma} \wedge \mathbf{P}_{\tau^k}^{(i),\sigma}
&
\geq
c_d.
\end{align*}
\end{lem}

\begin{proof}[Proof of Theorem \ref{thm:lower_bound_tangent}]

Let us apply Lemma \ref{conditional_assouad} with 
$\mathcal{X} = \R^D$, $
\mathcal{X}' = \bigl(\R^D\bigr)^{n-1}$,
$\mathcal{Q} = \bigl(\mathcal{P}^k(\sigma)\bigr)^{\otimes n}$,
$X = X_1$,
$X' = (X_2,\ldots,X_n) = X_{2:n}$,
$\theta_X(Q) = T_X M$,
and the angle between linear subspaces as the distance $d$.

If $\sigma < C_{k,d,\tau_{min}} \left({1}/({n-1)} \right)^{k/d}$, for $C_{k,d,\tau_{min}} > 0$ defined in Lemma \ref{lem:hypotheses_glissant_noise}, then, applying Lemma \ref{conditional_assouad} to the family $\bigl\{\bar{Q}^{(1)}_{\tau,n}\bigr\}_{\tau }$ together with the disjoint sets $U_k \times U_k'$ of Lemma \ref{lem:hypotheses_glissant_nonoise}, we get
\begin{align*}
\inf_{\hat{T}} 
\sup_{P \in \mathcal{P}^k(\sigma)}
\E_{P^{\otimes n}}
\angle
\bigr(
T_{\pi_M(X_1)} M
&,
\hat{T}
\bigr)
\geq 
m \cdot
c_{k,d,\tau_{min}} \left( \frac{1}{n-1}\right)^{\frac{k-1}{d}}
\cdot
c_0 \cdot c_d
\\
&=
c'_{d,k,\tau_{min}} 
\left\{
\left(\frac{1}{n-1}\right)^{\frac{k-1}{d}}  \right .
\vee
\left .
\left(\frac{\sigma}{n-1}\right)^{\frac{k-1}{d+k}}
\right\},
\end{align*}
where the second line uses that $\sigma < C_{k,d,\tau_{min}} \left({1}/({n-1)} \right)^{k/d}$.

If $\sigma \geq C_{k,d,\tau_{min}} \left({1}/({n-1)} \right)^{k/d}$, then Lemma \ref{lem:hypotheses_glissant_noise} holds, and considering the family $\bigl\{\bigl(\mathbf{P}_\tau^{(1),\sigma}\bigr)^{\otimes n} \bigr\}_{\tau}$, together with the disjoint sets $U_k^\sigma \times \bigl(\R^D\bigr)^{n-1}$, Lemma \ref{conditional_assouad} gives 
\begin{align*}
\inf_{\hat{T}} 
\sup_{P \in \mathcal{P}^k(\sigma)}
\E_{P^{\otimes n}}
\angle
\bigr(
T_{\pi_M(X_1)} M
&,
\hat{T}
\bigr)\geq 
m \cdot
c_{k,d,\tau_{min}} \left( \frac{\sigma}{n-1}\right)^{\frac{k-1}{k+d}}
\cdot
c_0 \cdot c_d
\\
&=
c''_{d,k,\tau_{min}} 
\left\{
\left(\frac{1}{n-1}\right)^{\frac{k-1}{d}}  \right .
\vee
\left .
\left(\frac{\sigma}{n-1}\right)^{\frac{k-1}{d+k}}
\right\},
\end{align*}
hence the result.

\end{proof}

\begin{proof}[Proof of Theorem \ref{thm:lower_bound_curvature}]
The proof follows the exact same lines as that of Theorem \ref{thm:lower_bound_tangent} just above. Namely, consider the same setting with $\theta_{X}(Q) = II_{\pi_M(X)}^M$. If $\sigma \geq C_{k,d,\tau_{min}} \left({1}/({n-1)} \right)^{k/d}$, apply Lemma \ref{conditional_assouad} with the family $\bigl\{\bar{Q}^{(2)}_{\tau,n}\bigr\}_{\tau}$ of Lemma \ref{lem:hypotheses_glissant_nonoise}. If $\sigma > C_{k,d,\tau_{min}} \left({1}/({n-1)} \right)^{k/d}$, Lemma \ref{conditional_assouad} can be applied to $\bigl\{\bigl(\mathbf{P}_\tau^{(2),\sigma}\bigr)^{\otimes n} \bigr\}_{\tau}$ in Lemma \ref{lem:hypotheses_glissant_noise}. This yields the announced rate.
\end{proof}

\section*{Acknowledgements}

We would like to thank Fr\'ed\'eric Chazal and Pascal Massart for their constant encouragements, suggestions and stimulating discussions.
We also thank the anonymous reviewers for valuable comments and suggestions.

\bibliographystyle{imsart-number}
\bibliography{biblio}

\begin{thebibliography}{36}

\bibitem{AamariLevrard2015}
\begin{barticle}[author]
\bauthor{\bsnm{{Aamari}},~\bfnm{E.}\binits{E.}} \AND
  \bauthor{\bsnm{{Levrard}},~\bfnm{C.}\binits{C.}}
(\byear{2015}).
\btitle{{Stability and Minimax Optimality of Tangential Delaunay Complexes for
  Manifold Reconstruction}}.
\bjournal{ArXiv e-prints}.
\end{barticle}
\endbibitem

\bibitem{AamariL17}
\begin{barticle}[author]
\bauthor{\bsnm{Aamari},~\bfnm{Eddie}\binits{E.}} \AND
  \bauthor{\bsnm{Levrard},~\bfnm{Clément}\binits{C.}}
(\byear{2017}).
\btitle{Non-asymptotic rates for manifold, tangent space and curvature
  estimation}.
\end{barticle}
\endbibitem

\bibitem{Alexander06}
\begin{barticle}[author]
\bauthor{\bsnm{Alexander},~\bfnm{Stephanie~B.}\binits{S.~B.}} \AND
  \bauthor{\bsnm{Bishop},~\bfnm{Richard~L.}\binits{R.~L.}}
(\byear{2006}).
\btitle{Gauss equation and injectivity radii for subspaces in spaces of
  curvature bounded above}.
\bjournal{Geom. Dedicata}
\bvolume{117}
\bpages{65--84}.
\bdoi{10.1007/s10711-005-9011-6}
\bmrnumber{2231159 (2007c:53110)}
\end{barticle}
\endbibitem

\bibitem{Arias13}
\begin{barticle}[author]
\bauthor{\bsnm{{Arias-Castro}},~\bfnm{E.}\binits{E.}},
  \bauthor{\bsnm{{Lerman}},~\bfnm{G.}\binits{G.}} \AND
  \bauthor{\bsnm{{Zhang}},~\bfnm{T.}\binits{T.}}
(\byear{2013}).
\btitle{{Spectral Clustering Based on Local PCA}}.
\bjournal{ArXiv e-prints}.
\end{barticle}
\endbibitem

\bibitem{Arias16}
\begin{barticle}[author]
\bauthor{\bsnm{{Arias-Castro}},~\bfnm{E.}\binits{E.}},
  \bauthor{\bsnm{{Pateiro-L{\'o}pez}},~\bfnm{B.}\binits{B.}} \AND
  \bauthor{\bsnm{{Rodr{\'{\i}}guez-Casal}},~\bfnm{A.}\binits{A.}}
(\byear{2016}).
\btitle{{Minimax Estimation of the Volume of a Set with Smooth Boundary}}.
\bjournal{ArXiv e-prints}.
\end{barticle}
\endbibitem

\bibitem{Boissonnat14}
\begin{barticle}[author]
\bauthor{\bsnm{Boissonnat},~\bfnm{Jean-Daniel}\binits{J.-D.}} \AND
  \bauthor{\bsnm{Ghosh},~\bfnm{Arijit}\binits{A.}}
(\byear{2014}).
\btitle{Manifold reconstruction using tangential {D}elaunay complexes}.
\bjournal{Discrete Comput. Geom.}
\bvolume{51}
\bpages{221--267}.
\bdoi{10.1007/s00454-013-9557-2}
\bmrnumber{3148657}
\end{barticle}
\endbibitem

\bibitem{Massart13}
\begin{bbook}[author]
\bauthor{\bsnm{Boucheron},~\bfnm{St\'ephane}\binits{S.}},
  \bauthor{\bsnm{Lugosi},~\bfnm{G\'abor}\binits{G.}} \AND
  \bauthor{\bsnm{Massart},~\bfnm{Pascal}\binits{P.}}
(\byear{2013}).
\btitle{Concentration inequalities}.
\bpublisher{Oxford University Press, Oxford}
\bnote{A nonasymptotic theory of independence, With a foreword by Michel
  Ledoux}.
\bdoi{10.1093/acprof:oso/9780199535255.001.0001}
\bmrnumber{3185193}
\end{bbook}
\endbibitem

\bibitem{Bousquet02}
\begin{barticle}[author]
\bauthor{\bsnm{Bousquet},~\bfnm{Olivier}\binits{O.}}
(\byear{2002}).
\btitle{A {B}ennett concentration inequality and its application to suprema of
  empirical processes}.
\bjournal{C. R. Math. Acad. Sci. Paris}
\bvolume{334}
\bpages{495--500}.
\bdoi{10.1016/S1631-073X(02)02292-6}
\bmrnumber{1890640 (2003f:60039)}
\end{barticle}
\endbibitem

\bibitem{Cazals06}
\begin{barticle}[author]
\bauthor{\bsnm{Cazals},~\bfnm{F.}\binits{F.}} \AND
  \bauthor{\bsnm{Pouget},~\bfnm{M.}\binits{M.}}
(\byear{2005}).
\btitle{Estimating differential quantities using polynomial fitting of
  osculating jets}.
\bjournal{Comput. Aided Geom. Design}
\bvolume{22}
\bpages{121--146}.
\bdoi{10.1016/j.cagd.2004.09.004}
\bmrnumber{2116098}
\end{barticle}
\endbibitem

\bibitem{Chazal2013}
\begin{barticle}[author]
\bauthor{\bsnm{{Chazal}},~\bfnm{F.}\binits{F.}},
  \bauthor{\bsnm{{Glisse}},~\bfnm{M.}\binits{M.}},
  \bauthor{\bsnm{{Labru{\`e}re}},~\bfnm{C.}\binits{C.}} \AND
  \bauthor{\bsnm{{Michel}},~\bfnm{B.}\binits{B.}}
(\byear{2013}).
\btitle{{Optimal rates of convergence for persistence diagrams in Topological
  Data Analysis}}.
\bjournal{ArXiv e-prints}.
\end{barticle}
\endbibitem

\bibitem{Cheng16}
\begin{barticle}[author]
\bauthor{\bsnm{Cheng},~\bfnm{Siu-Wing}\binits{S.-W.}} \AND
  \bauthor{\bsnm{Chiu},~\bfnm{Man-Kwun}\binits{M.-K.}}
(\byear{2016}).
\btitle{Tangent estimation from point samples}.
\bjournal{Discrete Comput. Geom.}
\bvolume{56}
\bpages{505--557}.
\bdoi{10.1007/s00454-016-9809-z}
\bmrnumber{3544007}
\end{barticle}
\endbibitem

\bibitem{Dey05}
\begin{binproceedings}[author]
\bauthor{\bsnm{Cheng},~\bfnm{Siu-Wing}\binits{S.-W.}},
  \bauthor{\bsnm{Dey},~\bfnm{Tamal~K.}\binits{T.~K.}} \AND
  \bauthor{\bsnm{Ramos},~\bfnm{Edgar~A.}\binits{E.~A.}}
(\byear{2005}).
\btitle{Manifold reconstruction from point samples}.
In \bbooktitle{Proceedings of the {S}ixteenth {A}nnual {ACM}-{SIAM} {S}ymposium
  on {D}iscrete {A}lgorithms}
\bpages{1018--1027}.
\bpublisher{ACM, New York}.
\bmrnumber{2298361}
\end{binproceedings}
\endbibitem

\bibitem{DoCarmo92}
\begin{bbook}[author]
\bauthor{\bparticle{do}
  \bsnm{Carmo},~\bfnm{Manfredo~Perdig{\~a}o}\binits{M.~P.}}
(\byear{1992}).
\btitle{Riemannian geometry}.
\bseries{Mathematics: Theory \& Applications}.
\bpublisher{Birkh\"auser Boston, Inc., Boston, MA}
\bnote{Translated from the second Portuguese edition by Francis Flaherty}.
\bdoi{10.1007/978-1-4757-2201-7}
\bmrnumber{1138207 (92i:53001)}
\end{bbook}
\endbibitem

\bibitem{Dyer15}
\begin{barticle}[author]
\bauthor{\bsnm{Dyer},~\bfnm{Ramsay}\binits{R.}},
  \bauthor{\bsnm{Vegter},~\bfnm{Gert}\binits{G.}} \AND
  \bauthor{\bsnm{Wintraecken},~\bfnm{Mathijs}\binits{M.}}
(\byear{2015}).
\btitle{Riemannian simplices and triangulations}.
\bjournal{Geom. Dedicata}
\bvolume{179}
\bpages{91--138}.
\bdoi{10.1007/s10711-015-0069-5}
\bmrnumber{3424659}
\end{barticle}
\endbibitem

\bibitem{Fasy14}
\begin{barticle}[author]
\bauthor{\bsnm{Fasy},~\bfnm{Brittany~Terese}\binits{B.~T.}},
  \bauthor{\bsnm{Lecci},~\bfnm{Fabrizio}\binits{F.}},
  \bauthor{\bsnm{Rinaldo},~\bfnm{Alessandro}\binits{A.}},
  \bauthor{\bsnm{Wasserman},~\bfnm{Larry}\binits{L.}},
  \bauthor{\bsnm{Balakrishnan},~\bfnm{Sivaraman}\binits{S.}} \AND
  \bauthor{\bsnm{Singh},~\bfnm{Aarti}\binits{A.}}
(\byear{2014}).
\btitle{Confidence sets for persistence diagrams}.
\bjournal{Ann. Statist.}
\bvolume{42}
\bpages{2301--2339}.
\bdoi{10.1214/14-AOS1252}
\bmrnumber{3269981}
\end{barticle}
\endbibitem

\bibitem{federer1959}
\begin{barticle}[author]
\bauthor{\bsnm{Federer},~\bfnm{Herbert}\binits{H.}}
(\byear{1959}).
\btitle{Curvature measures}.
\bjournal{Trans. Amer. Math. Soc.}
\bvolume{93}
\bpages{418--491}.
\bmrnumber{0110078 (22 \#\#961)}
\end{barticle}
\endbibitem

\bibitem{Federer69}
\begin{bbook}[author]
\bauthor{\bsnm{Federer},~\bfnm{Herbert}\binits{H.}}
(\byear{1969}).
\btitle{Geometric measure theory}.
\bseries{Die Grundlehren der mathematischen Wissenschaften, Band 153}.
\bpublisher{Springer-Verlag New York Inc., New York}.
\bmrnumber{0257325 (41 \#\#1976)}
\end{bbook}
\endbibitem

\bibitem{Fefferman15}
\begin{barticle}[author]
\bauthor{\bsnm{{Fefferman}},~\bfnm{C.}\binits{C.}},
  \bauthor{\bsnm{{Ivanov}},~\bfnm{S.}\binits{S.}},
  \bauthor{\bsnm{{Kurylev}},~\bfnm{Y.}\binits{Y.}},
  \bauthor{\bsnm{{Lassas}},~\bfnm{M.}\binits{M.}} \AND
  \bauthor{\bsnm{{Narayanan}},~\bfnm{H.}\binits{H.}}
(\byear{2015}).
\btitle{{Reconstruction and interpolation of manifolds I: The geometric Whitney
  problem}}.
\bjournal{ArXiv e-prints}.
\end{barticle}
\endbibitem

\bibitem{Gashler11}
\begin{bincollection}[author]
\bauthor{\bsnm{Gashler},~\bfnm{Michael~S.}\binits{M.~S.}} \AND
  \bauthor{\bsnm{Martinez},~\bfnm{Tony}\binits{T.}}
(\byear{2011}).
\btitle{Tangent Space Guided Intelligent Neighbor Finding}.
In \bbooktitle{Proceedings of the {IEEE} International Joint Conference on
  Neural Networks {IJCNN}'11}
\bpages{2617--2624}.
\bpublisher{IEEE Press}.
\end{bincollection}
\endbibitem

\bibitem{Genovese12}
\begin{barticle}[author]
\bauthor{\bsnm{Genovese},~\bfnm{Christopher~R.}\binits{C.~R.}},
  \bauthor{\bsnm{Perone-Pacifico},~\bfnm{Marco}\binits{M.}},
  \bauthor{\bsnm{Verdinelli},~\bfnm{Isabella}\binits{I.}} \AND
  \bauthor{\bsnm{Wasserman},~\bfnm{Larry}\binits{L.}}
(\byear{2012}).
\btitle{Manifold estimation and singular deconvolution under {H}ausdorff loss}.
\bjournal{Ann. Statist.}
\bvolume{40}
\bpages{941--963}.
\bdoi{10.1214/12-AOS994}
\bmrnumber{2985939}
\end{barticle}
\endbibitem

\bibitem{Genovese11}
\begin{barticle}[author]
\bauthor{\bsnm{Genovese},~\bfnm{Christopher~R.}\binits{C.~R.}},
  \bauthor{\bsnm{Perone-Pacifico},~\bfnm{Marco}\binits{M.}},
  \bauthor{\bsnm{Verdinelli},~\bfnm{Isabella}\binits{I.}} \AND
  \bauthor{\bsnm{Wasserman},~\bfnm{Larry}\binits{L.}}
(\byear{2012}).
\btitle{Minimax manifold estimation}.
\bjournal{J. Mach. Learn. Res.}
\bvolume{13}
\bpages{1263--1291}.
\bmrnumber{2930639}
\end{barticle}
\endbibitem

\bibitem{Golub96}
\begin{bbook}[author]
\bauthor{\bsnm{Golub},~\bfnm{Gene~H.}\binits{G.~H.}} \AND
  \bauthor{\bsnm{Van~Loan},~\bfnm{Charles~F.}\binits{C.~F.}}
(\byear{1996}).
\btitle{Matrix computations},
\bedition{third} ed.
\bseries{Johns Hopkins Studies in the Mathematical Sciences}.
\bpublisher{Johns Hopkins University Press, Baltimore, MD}.
\bmrnumber{1417720 (97g:65006)}
\end{bbook}
\endbibitem

\bibitem{Gumhold01}
\begin{binproceedings}[author]
\bauthor{\bsnm{Gumhold},~\bfnm{Stefan}\binits{S.}},
  \bauthor{\bsnm{Wang},~\bfnm{Xinlong}\binits{X.}} \AND
  \bauthor{\bsnm{MacLeod},~\bfnm{Rob}\binits{R.}}
(\byear{2001}).
\btitle{{Feature Extraction from Point Clouds}}.
In \bbooktitle{10th International Meshing Roundtable}
\bpages{293--305}.
\bpublisher{Sandia National Laboratories,}.
\end{binproceedings}
\endbibitem

\bibitem{Hartman51}
\begin{barticle}[author]
\bauthor{\bsnm{Hartman},~\bfnm{Philip}\binits{P.}}
(\byear{1951}).
\btitle{On geodesic coordinates}.
\bjournal{Amer. J. Math.}
\bvolume{73}
\bpages{949--954}.
\bdoi{10.2307/2372125}
\bmrnumber{0046087}
\end{barticle}
\endbibitem

\bibitem{Kim2015}
\begin{barticle}[author]
\bauthor{\bsnm{Kim},~\bfnm{Arlene K.~H.}\binits{A.~K.~H.}} \AND
  \bauthor{\bsnm{Zhou},~\bfnm{Harrison~H.}\binits{H.~H.}}
(\byear{2015}).
\btitle{Tight minimax rates for manifold estimation under {H}ausdorff loss}.
\bjournal{Electron. J. Stat.}
\bvolume{9}
\bpages{1562--1582}.
\bdoi{10.1214/15-EJS1039}
\bmrnumber{3376117}
\end{barticle}
\endbibitem

\bibitem{Maggioni16}
\begin{barticle}[author]
\bauthor{\bsnm{Maggioni},~\bfnm{Mauro}\binits{M.}},
  \bauthor{\bsnm{Minsker},~\bfnm{Stanislav}\binits{S.}} \AND
  \bauthor{\bsnm{Strawn},~\bfnm{Nate}\binits{N.}}
(\byear{2016}).
\btitle{Multiscale dictionary learning: non-asymptotic bounds and robustness}.
\bjournal{J. Mach. Learn. Res.}
\bvolume{17}
\bpages{Paper No. 2, 51}.
\bmrnumber{3482922}
\end{barticle}
\endbibitem

\bibitem{Merigot11}
\begin{barticle}[author]
\bauthor{\bsnm{Merigot},~\bfnm{Q.}\binits{Q.}},
  \bauthor{\bsnm{Ovsjanikov},~\bfnm{M.}\binits{M.}} \AND
  \bauthor{\bsnm{Guibas},~\bfnm{L.~J.}\binits{L.~J.}}
(\byear{2011}).
\btitle{Voronoi-Based Curvature and Feature Estimation from Point Clouds}.
\bjournal{IEEE Transactions on Visualization and Computer Graphics}
\bvolume{17}
\bpages{743-756}.
\bdoi{10.1109/TVCG.2010.261}
\end{barticle}
\endbibitem

\bibitem{Niyogi08}
\begin{barticle}[author]
\bauthor{\bsnm{Niyogi},~\bfnm{Partha}\binits{P.}},
  \bauthor{\bsnm{Smale},~\bfnm{Stephen}\binits{S.}} \AND
  \bauthor{\bsnm{Weinberger},~\bfnm{Shmuel}\binits{S.}}
(\byear{2008}).
\btitle{Finding the homology of submanifolds with high confidence from random
  samples}.
\bjournal{Discrete Comput. Geom.}
\bvolume{39}
\bpages{419--441}.
\bdoi{10.1007/s00454-008-9053-2}
\bmrnumber{2383768 (2009b:60038)}
\end{barticle}
\endbibitem

\bibitem{Roweis00}
\begin{barticle}[author]
\bauthor{\bsnm{Roweis},~\bfnm{Sam~T.}\binits{S.~T.}} \AND
  \bauthor{\bsnm{Saul},~\bfnm{Lawrence~K.}\binits{L.~K.}}
(\byear{2000}).
\btitle{Nonlinear dimensionality reduction by locally linear embedding}.
\bjournal{SCIENCE}
\bvolume{290}
\bpages{2323--2326}.
\end{barticle}
\endbibitem

\bibitem{Rusinkiewicz04}
\begin{binproceedings}[author]
\bauthor{\bsnm{Rusinkiewicz},~\bfnm{Szymon}\binits{S.}}
(\byear{2004}).
\btitle{Estimating Curvatures and Their Derivatives on Triangle Meshes}.
In \bbooktitle{Proceedings of the 3D Data Processing, Visualization, and
  Transmission, 2Nd International Symposium}.
\bseries{3DPVT '04}
\bpages{486--493}.
\bpublisher{IEEE Computer Society}, \baddress{Washington, DC, USA}.
\bdoi{10.1109/3DPVT.2004.54}
\end{binproceedings}
\endbibitem

\bibitem{Singer12}
\begin{barticle}[author]
\bauthor{\bsnm{Singer},~\bfnm{A.}\binits{A.}} \AND
  \bauthor{\bsnm{Wu},~\bfnm{H.~T.}\binits{H.~T.}}
(\byear{2012}).
\btitle{Vector diffusion maps and the connection {L}aplacian}.
\bjournal{Comm. Pure Appl. Math.}
\bvolume{65}
\bpages{1067--1144}.
\bdoi{10.1002/cpa.21395}
\bmrnumber{2928092}
\end{barticle}
\endbibitem

\bibitem{Tenenbaum00}
\begin{barticle}[author]
\bauthor{\bsnm{Tenenbaum},~\bfnm{Joshua~B.}\binits{J.~B.}},
  \bauthor{\bparticle{de} \bsnm{Silva},~\bfnm{Vin}\binits{V.}} \AND
  \bauthor{\bsnm{Langford},~\bfnm{John~C.}\binits{J.~C.}}
(\byear{2000}).
\btitle{A Global Geometric Framework for Nonlinear Dimensionality Reduction}.
\bjournal{Science}
\bvolume{290}
\bpages{2319}.
\end{barticle}
\endbibitem

\bibitem{Tyagi13}
\begin{barticle}[author]
\bauthor{\bsnm{Tyagi},~\bfnm{Hemant}\binits{H.}},
  \bauthor{\bsnm{Vural},~\bfnm{El\i~f}\binits{E.~f.}} \AND
  \bauthor{\bsnm{Frossard},~\bfnm{Pascal}\binits{P.}}
(\byear{2013}).
\btitle{Tangent space estimation for smooth embeddings of {R}iemannian
  manifolds}.
\bjournal{Inf. Inference}
\bvolume{2}
\bpages{69--114}.
\bdoi{10.1093/imaiai/iat003}
\bmrnumber{3311444}
\end{barticle}
\endbibitem

\bibitem{Usevich14}
\begin{barticle}[author]
\bauthor{\bsnm{Usevich},~\bfnm{Konstantin}\binits{K.}} \AND
  \bauthor{\bsnm{Markovsky},~\bfnm{Ivan}\binits{I.}}
(\byear{2014}).
\btitle{Optimization on a Grassmann manifold with application to system
  identification}.
\bjournal{Automatica}
\bvolume{50}
\bpages{1656 - 1662}.
\bdoi{http://dx.doi.org/10.1016/j.automatica.2014.04.010}
\end{barticle}
\endbibitem

\bibitem{Wasserman16}
\begin{barticle}[author]
\bauthor{\bsnm{{Wasserman}},~\bfnm{L.}\binits{L.}}
(\byear{2016}).
\btitle{{Topological Data Analysis}}.
\bjournal{ArXiv e-prints}.
\end{barticle}
\endbibitem

\bibitem{Yu97}
\begin{bincollection}[author]
\bauthor{\bsnm{Yu},~\bfnm{Bin}\binits{B.}}
(\byear{1997}).
\btitle{Assouad, fano, and le cam}.
In \bbooktitle{Festschrift for Lucien Le Cam}
\bpages{423--435}.
\bpublisher{Springer}.
\end{bincollection}
\endbibitem

\end{thebibliography}

\newpage
\setcounter{section}{0}
\setcounter{subsection}{0}
\setcounter{subsubsection}{0}
\renewcommand{\thesection}{\Alph{section}}
\sectionfont{\centering\normalfont\fontsize{19}{22}}
\renewcommand{\thesubsection}{\thesection.\arabic{subsection}}

{\Large
\begin{center}
\textsc{Appendix: Geometric background and proofs of intermediate results}
\end{center}
}
\medskip
\section*{Appendix A: Properties and Stability of the Models}
\label{sec:appendix:model_properties}
\addtocounter{section}{1}
\subsection{Property of the Exponential Map in $\mathcal{C}^2_{\tau_{min}}$}\label{subsec:A1}

Here we show the following Lemma 1, reproduced as Lemma \ref{model_properties_k=2}.

\begin{Alem}\label{model_properties_k=2}
If $M \in \mathcal{C}^2_{\tau_{min}}$,
$\exp_p: \mathcal{B}_{T_pM}\left(0, \tau_{min}/4 \right) \rightarrow M$ is one-to-one. Moreover, it can be written as
\begin{align*}
  \exp_p \colon \mathcal{B}_{T_pM}\left(0, \tau_{min}/4 \right) &\longrightarrow M 
  \\
  v &\longmapsto p + v + \mathbf{N}_p(v)
\end{align*}
with $\mathbf{N}_p$ such that for all $v \in \mathcal{B}_{T_pM}\left(0, \tau_{min} /4 \right)$,
\begin{align*}
\mathbf{N}_p(0)=0, \quad d_0 \mathbf{N}_p = 0 , \quad \norm{d_v \mathbf{N}_p}_{op} \leq L_\perp \norm{v},
\end{align*}
where $L_\perp = 5/(4\tau_{min})$.
Furthermore, for all $p,y \in M $,
\begin{align*}
y-p = \pi_{T_pM}(y-p) + R_2(y-p),
\end{align*}
where $\norm{R_2(y-p)} \leq \frac{\norm{y-p}^2}{2\tau_{min}}$. 
\end{Alem}

\begin{proof}[Proof of Lemma \ref{model_properties_k=2}]
Proposition 6.1 in \cite{Niyogi08} states that for all $x\in M$, $\norm{II^M_x}_{op} \leq 1/\tau_{min}$. In particular, Gauss equation (\cite[Proposition 3.1 (a), p.135]{DoCarmo92}) yields that the sectional curvatures of $M$ satisfy $-2/\tau_{min}^2 \leq \kappa \leq 1/\tau_{min}^2$. Using Corollary 1.4 of \cite{Alexander06}, we get that the injectivity radius of $M$ is at least $\pi \tau_{min} \geq \tau_{min}/4$. Therefore, $\exp_p: \mathcal{B}_{T_p M} (0,\tau_{min}/4) \rightarrow M$ is one-to-one.

Let us write $\mathbf{N}_p(v) = \exp_p(v) - p - v$. We clearly have $\mathbf{N}_p(0) = 0$ and $d_0 \mathbf{N}_p = 0$. 
Let now $v \in \mathcal{B}_{T_p M}(0,\tau_{min}/4)$ be fixed. We have $d_v \mathbf{N}_p =  d_v \exp_p - Id_{T_p M}$.
For $0 \leq t \leq \norm{v}$, we write $\gamma(t) = \exp_p(tv/\norm{v})$ for the arc-length parametrized geodesic from $p$ to $\exp_p(v)$, and $P_{t}$ for the parallel translation along $\gamma$. From Lemma 18 of \cite{Dyer15},
\begin{align*}
\norm{
d_{t \frac{v}{\norm{v}}} \exp_p
-
P_{t}
}_{op}
\leq
\frac{2}{\tau_{min}^2}
\frac{t^2}{2}
\leq
\frac{t}{4\tau_{min}}.
\end{align*}
We now derive an upper bound for $\norm{ P_{t} - Id_{T_p M}}_{op}$. 
For this, fix two unit vectors $u \in \R^D$ and $w \in T_p M$, and write $g(t) =  \langle P_t(w) - w , u \rangle $. Letting $\bar{\nabla}$ denote the ambient derivative in $\R^D$, by definition of parallel translation,
\begin{align*}
\left\lvert
g'(t)
\right\rvert
&=
\left\lvert
\langle \bar{\nabla}_{\gamma'(t)} P_t(w) - w , u \rangle
\right\rvert
\\
&=
\left\lvert
\langle II^M_{\gamma(t)}\bigl(\gamma'(t),  P_t(w) \bigr) , u \rangle
\right\rvert
\\
&\leq
1/\tau_{min}.
\end{align*}
Since $g(0) = 0$, we get $\norm{ P_{t} - Id_{T_p M}}_{op} \leq t/\tau_{min}$.
Finally, the triangle inequality leads to 
\begin{align*}
\norm{d_v \mathbf{N}_p}_{op}
&=
\norm{d_{v} \exp - Id_{T_p M}}_{op} 
\\
&\leq
\norm{d_v \exp -  P_{\norm{v}}}_{op} 
+
\norm{P_{\norm{v}} - Id_{T_p M}}_{op} 
\\
&\leq
\frac{5 \norm{v} }{4 \tau_{min}}.
\end{align*} 
We conclude with the property of the projection $\pi^\ast = \pi_{T_p M}$. Indeed, defining $R_2(y-p) = (y-p) - \pi^\ast(y-p)$, Lemma 4.7 in \cite{federer1959} gives
\begin{align*}
\norm{R_2(y-p)}
&=
d(y-p,T_p M)
\\
&\leq
\frac{\norm{y-p}^2}{2\tau_{min}}.
\end{align*}

\end{proof}

\subsection{Geometric Properties of the Models $\mathcal{C}^k$}\label{subsec:A2}

\begin{Alem}\label{Alem:fourre_tout_geom_model}
For any $M \in \mathcal{C}^{k}_{\tau_{min},\mathbf{L}}$ and $x\in M$, the following holds.

\begin{enumerate}[(i)]

\item For all $v_1,v_2\in \mathcal{B}_{T_xM}\left(0,\frac{1}{4L_\perp}\right)$,
\begin{align*}
\frac{3}{4}\norm{v_2-v_1} \leq \norm{\Psi_x(v_2)-\Psi_x(v_1)} \leq \frac{5}{4} \norm{v_2-v_1}
.
\end{align*}

\item For all $h \leq \frac{1}{4L_\perp} \wedge \frac{2 \tau_{min}}{5}$, 
\begin{align*}
M \cap \mathcal{B}\left(x,\frac{3h}{5}\right)
\subset
\Psi_x \left( \mathcal{B}_{T_x M} \left(x,h\right) \right)
\subset
M \cap \mathcal{B}\left(x,\frac{5h}{4}\right)
.
\end{align*}

\item For all $h \leq \frac{\tau_{min}}{2}$, 
\begin{align*}
\mathcal{B}_{T_x M}\left(0,\frac{7h}{8}\right) \subset \pi_{T_x M}\left(\mathcal{B}(x,h) \cap M\right).
\end{align*}

\item\label{korderdecomposition}
Denoting by $\pi^\ast = \pi_{T_x M}$ the orthogonal projection onto $T_x M$, for all $x \in M$, there exist multilinear maps $T_2^\ast,\ldots,T_{k-1}^\ast$ from $T_x M$ to $\R^D$, and $R_k$ such that for all $y \in  \mathcal{B}\left(x,\frac{\tau_{min} \wedge L_\perp^{-1}}{4}\right)\cap M$,
\begin{align*}
y-x = \pi^*&(y-x) + T_2^*( \tens{\pi^*(y-x)}{2}) + \hdots + T_{k-1}^*(\tens{\pi^*(y-x)}{k-1}) 
		\\
		&+ R_{k}(y-x),
\end{align*}
with
\begin{align*}
\norm{R_k (y-x)} \leq C \norm{y-x}^{k}
\text{ \quad  and \quad }
\norm{T_i^\ast}_{op} \leq L_i', \text{ for } 2 \leq i \leq k-1,
\end{align*}
where $L'_i$ depends on $d,k,\tau_{min},L_\perp,\ldots,L_i$, and $C$ on $d$, $k$, $\tau_{min}$, $L_\perp$, $\ldots$,  $L_k$. Moreover, for $k\geq 3$, $T_2^\ast = II^M_x$. 

\item For all $x \in M$, $\norm{II^M_x}_{op} \leq 1/\tau_{min}$. In particular, the sectional curvatures of $M$ satisfy 
\begin{align*}
\frac{-2}{\tau_{min}^2}
\leq
\kappa
\leq 
\frac{1}{\tau_{min}^2}
.
\end{align*}

\end{enumerate}

\end{Alem}

\begin{proof}[Proof of Lemma \ref{Alem:fourre_tout_geom_model}]
\begin{enumerate}[(i)]

\item Simply notice that from the reverse triangle inequality,
\begin{align*}
\left|
\frac{\norm{\Psi_x(v_2)-\Psi_x(v_1)}}{\norm{v_2-v_1}}
-
1
\right|
&\leq
\frac{\norm{N_x(v_2)-N_x(v_1)}}{\norm{v_2-v_1}}
\leq
L_\perp (\norm{v_1}\vee \norm{v_2})
\leq \frac{1}{4}.
\end{align*}

\item 
The right-hand side inclusion follows straightforwardly from (i). Let us focus on the left-hand side inclusion. For this, consider the map defined by $G =\pi_{T_x M} \circ \Psi_x$ on the domain $\mathcal{B}_{T_x M }\left(0, h\right)$. For all $v \in \mathcal{B}_{T_x M }\left(0, h\right)$, we have
\begin{align*}
\norm{d_v G - Id_{T_x M}}_{op}
&=
\norm{ \pi_{T_x M} \circ d_v \mathbf{N}_x}_{op}
\leq
\norm{ d_v \mathbf{N}_x}_{op}
\leq
L_\perp \norm{v}
\leq \frac{1}{4}<1.
\end{align*}
Hence, $G$ is a diffeomorphism onto its image and it satisfies $\norm{G(v)}\geq {3 \norm{v}}/{4}$. It follows that
\begin{align*}
\mathcal{B}_{T_x M }\left(0, \frac{3h}{4}\right) \subset G\left( \mathcal{B}_{T_x M }\left(0, h\right) \right)
=
\pi_{T_xM} \left( \Psi_x \left( \mathcal{B}_{T_x M }\left(0, h\right) \right) \right)
.
\end{align*}

Now, according to Lemma \ref{model_properties_k=2}, for all $y \in \mathcal{B}\left(x, \frac{3h}{5}\right)\cap M$,
\begin{align*}
\norm{\pi_{T_x M}(y-x)}
& \leq 
\norm{y-x} + \frac{\norm{y-x}^2}{2\tau_{min}}
\leq
\left(1 + \frac{1}{4}\right)\norm{y-x}
\leq
\frac{3h}{4},
\end{align*}
from what we deduce $\pi_{T_x M} \left( \mathcal{B}\left(x,\frac{3h}{5}\right) \cap M \right) \subset \mathcal{B}_{T_x M}\left( 0, \frac{3h}{4}\right)$.
As a consequence, 
\begin{align*}
\pi_{T_x M} \left( \mathcal{B}\left(x,\frac{3h}{5}\right) \cap M \right)
\subset
\pi_{T_xM} \left( \Psi_x \left( \mathcal{B}_{T_x M }\left(0, h\right) \right) \right),
\end{align*}
which yields the announced inclusion since $\pi_{T_x M}$ is one to one on $\mathcal{B}\left(x,\frac{5h}{4}\right) \cap M$ from Lemma 3 in \cite{Arias13}, and 
\begin{align*}
\left(\mathcal{B}\left(x,\frac{3h}{5}\right) \cap M\right) 
\subset 
\Psi_x \left( \mathcal{B}_{T_x M }\left(0, h\right)\right)
\subset 
\mathcal{B}\left(x,\frac{5h}{4}\right) \cap M.
\end{align*}

%

\item Straightforward application of Lemma 3 in \cite{Arias13}.

\item Notice that Lemma \ref{model_properties_k=2} gives the existence of such an expansion for $k=2$. Hence, we can assume $k \geq 3$.
Taking $h = \frac{\tau_{min} \wedge L_\perp^{-1}}{4}$, we showed in the proof of (ii) that the map $G$ is a diffeomorphism onto its image, with $\norm{d_v G - Id_{T_x M}}_{op} \leq \frac{1}{4} < 1$. Additionally, the chain rule yields $\norm{d^i_v G}_{op} \leq \norm{d_v^i \Psi_x}_{op} \leq L_i$ for all $2 \leq i \leq k$. Therefore, from Lemma \ref{lischitz_of_inverse}, the differentials of $G^{-1}$ up to order $k$ are uniformly bounded. As a consequence, we get the announced expansion writing
\begin{align*}
y-x
&=
\Psi_x \circ G^{-1} \left( \pi^\ast (y-x) \right),
\end{align*}
and using the Taylor expansions of order $k$ of $\Psi_x$ and $G^{-1}$. 

Let us now check that $T^\ast_2 = II^M_x$.
Since, by construction, $T_2^\ast$ is the second order term of the Taylor expansion of $\Psi_x \circ G^{-1}$ at zero, a straightforward computation yields
\begin{align*}
T_2^\ast
&=
(I_D - \pi_{T_x M}) \circ d^2_0 \Psi_x
\\
&=
\pi_{T_xM^\perp} \circ d^2_0 \Psi_x.
\end{align*}
Let $v \in T_x M$ be fixed. Letting $\gamma(t) = \Psi_x(tv)$ for $|t|$ small enough, it is clear that $\gamma''(0) = d^2_0 \Psi(\tens{v}{2})$.
Moreover, by definition of the second fundamental form \cite[Proposition 2.1, p.127]{DoCarmo92}, since $\gamma(0) = x$ and $\gamma'(0) = v$, we have
\begin{align*}
II^M_x(\tens{v}{2})
&=
\pi_{T_x M^\perp} (\gamma''(0)).
\end{align*}
Hence 
\begin{align*}
T_2^\ast(\tens{v}{2})
&=
\pi_{T_xM^\perp} \circ d^2_0 \Psi_x(\tens{v}{2})
\\
&=
\pi_{T_x M^\perp} (\gamma''(0))
\\
&=
II^M_x(\tens{v}{2}),
\end{align*}
which concludes the proof.

\item The first statement is a rephrasing of Proposition 6.1 in \cite{Niyogi08}. It yields the bound on sectional curvature, using the Gauss equation \cite[Proposition 3.1 (a), p.135]{DoCarmo92}.

\end{enumerate}

\end{proof}

In the proof of Lemma \ref{Alem:fourre_tout_geom_model} (iv), we used a technical lemma of differential calculus that we now prove. It states quantitatively that if $G$ is $\mathcal{C}^k$-close to the identity map, then it is a diffeomorphism onto its image and the differentials of its inverse $G^{-1}$ are controlled.

\begin{Alem}\label{lischitz_of_inverse}
Let $k\geq 2$ and $U$ be an open subset of $\R^d$. Let $G : U \rightarrow \R^d$ be $\mathcal{C}^{k}$. Assume that $\norm{I_d - d G }_{op} \leq \varepsilon < 1$, and that for all $2 \leq i \leq k$, $\norm{d^i G}_{op} \leq L_i$ for some $L_i >0$.
Then $G$ is a $\mathcal{C}^{k}$-diffeomorphism onto its image, and for all $2 \leq i \leq k$,
\begin{align*}
\norm{I_d - d G^{-1}}_{op} \leq \frac{\varepsilon}{1-\varepsilon}
\text{\quad and \quad}
\norm{d^i G^{-1}}_{op}
\leq
L'_{i,\varepsilon,L_2,\ldots,L_i} < \infty,
\text{ for } 2\leq i \leq k.
\end{align*}
\end{Alem}

\begin{proof}[Proof of Lemma \ref{lischitz_of_inverse}]
For all $x\in U$, $\norm{d_x G - I_d}_{op} <1$, so $G$ is one to one, and for all $y=G(x) \in G(U)$,
\begin{align*}
\norm{I_d - d_y G^{-1}}_{op}
&=
\norm{I_d - (d_x G)^{-1}}_{op}
\\
&\leq
\norm{(d_x G)^{-1}}_{op} \norm{I_d- d_x G}_{op}
\\
&\leq
\frac{\norm{I_d- d_x G}_{op}}{1-\norm{I_d- d_x G}_{op}}
\\
&\leq
\frac{\varepsilon}{1-\varepsilon}
.
\end{align*}
For $2\leq  i \leq k$ and $1\leq j \leq i$, write $\Pi_i^{(j)}$ for the set of partitions of $\left\{1,\ldots,i\right\}$ with $j$ blocks.
Differentiating $i$ times the identity $G \circ G^{-1} = Id_{G(U)}$, Faa di Bruno's formula yields that, for all $y = G(x)\in G(U)$ and all unit vectors $h_1,\ldots,h_i \in \R^D$,
\begin{align*}
0 = 
d_y \left( G \circ G^{-1} \right).(h_\alpha)_{1\leq \alpha \leq i}
=
\sum_{j = 1}^i
\sum_{\pi \in \Pi_i^{(j)}}
d_x^{j} G
.
\left(
\left(
d_y^{\left| I \right|} G^{-1}.
\left(
h_\alpha
\right)_{\alpha \in I}
\right)_{I \in \pi}
\right).
\end{align*}
Isolating the term for $j=1$ entails
\begin{align*}
&\norm{
d_x G.
\left(
d_y^i G^{-1}.
\left(h_\alpha\right)_{1 \leq \alpha \leq i}
\right)
}_{op}
\\
&\hspace{1ex}=
\norm{
-
\sum_{j = 2}^i
\sum_{\pi \in \Pi_i^{(j)}}
d_x^{j} G
.
\left(
\left(
d_y^{\left| I \right|} G^{-1}.
\left(
h_\alpha
\right)_{\alpha \in I}
\right)_{I \in \pi}
\right)
}_{op}
\\
&\hspace{1ex}\leq
\sum_{j = 2}^i
\sum_{\pi \in \Pi_i^{(j)}}
\norm{d^{j} G}_{op}
\prod_{I \in \pi}
\norm{
d^{\left| I \right|} G^{-1}
}_{op}
.
\end{align*}
Using the first order Lipschitz bound on $G^{-1}$, we get
\begin{align*}
\norm{
d^i G^{-1}
}_{op}
&\leq
\frac{1+\varepsilon}{1-\varepsilon}
\sum_{j = 2}^i
L_j
\sum_{\pi \in \Pi_i^{(j)}}
\prod_{I \in \pi}
\norm{
d^{\left| I \right|} G^{-1}
}_{op}.
\end{align*}
The result follows by induction on $i$.
\end{proof}

\subsection{Proof of Proposition 1}\label{subsec:A3}

This section is devoted to prove Proposition 1 (reproduced below as Proposition \ref{statistical_model_stability}), that asserts the stability of the model  with respect to ambient diffeomorphisms.

\begin{Aprop}\label{statistical_model_stability}
Let $\Phi : \R^D \rightarrow \R^D$ be a global $\mathcal{C}^k$-diffeomorphism. If
$\norm{d \Phi - I_D}_{op}$
,
$\norm{d^2 \Phi}_{op}$
,
\ldots
,
$\norm{d^k \Phi}_{op}
$ 
are small enough,
then for all $P$ in $\mathcal{P}^{k}_{\tau_{min},\mathbf{L}, f_{min}, f_{max}}$, the pushforward distribution $P' = \Phi_\ast P$ belongs to $\mathcal{P}^{k}_{\tau_{min}/2,2\mathbf{L}, f_{min}/2, 2 f_{max}}$.

Moreover, if $\Phi = \lambda I_D$ ($\lambda>0$) is an homogeneous dilation, then $P' \in \mathcal{P}^{k}_{\lambda \tau_{min},\mathbf{L}_{(\lambda)},f_{min}/\lambda^d,f_{max}/\lambda^d}$, where $\mathbf{L}_{(\lambda)} = (L_\perp/\lambda,L_3/\lambda^2,\ldots,L_k/\lambda^{k-1})$.
\end{Aprop}

\begin{proof}[Proof of Proposition \ref{statistical_model_stability}]
The second part is straightforward since the dilation $\lambda M$ has reach $\tau_{\lambda M} = \lambda \tau_{M}$, and can be parametrized locally by $\tilde{\Psi}_{\lambda p} (v) = \lambda \Psi_p(v/\lambda) = \lambda p + v + \lambda \mathbf{N}_p(v/\lambda)$, yielding the differential bounds $\mathbf{L}_{(\lambda)}$. 
Bounds on the density follow from homogeneity of the $d$-dimensional Hausdorff measure.

The first part follows combining Proposition \ref{best_model_stability} and Lemma \ref{change_of_variable_hausdorff}.
\end{proof}

Proposition \ref{best_model_stability} asserts the stability of the geometric model, that is, the reach bound and the existence of a smooth parametrization when a submanifold is perturbed.

\begin{Aprop}\label{best_model_stability}
Let $\Phi : \R^D \rightarrow \R^D$ be a global $\mathcal{C}^k$-diffeomorphism. If
$\norm{d \Phi - I_D}_{op}$
,
$\norm{d^2 \Phi}_{op}$
,
\ldots
,
$\norm{d^k \Phi}_{op}
$ 
are small enough,
then for all $M$ in $\mathcal{C}^{k}_{\tau_{min},\mathbf{L}}$, the image $M' = \Phi\left(M\right)$ belongs to $\mathcal{C}^{k}_{\tau_{min}/2,2L_{\perp},2L_3,\ldots,2 L_k}$.
\end{Aprop}

\begin{proof}[Proof of Proposition \ref{best_model_stability}]
To bound $\tau_{M'}$ from below, we use the stability of the reach with respect to $\mathcal{C}^2$ diffeomorphisms. Namely, from Theorem 4.19 in \cite{federer1959},
\begin{align*}
\tau_{M'}
= \tau_{\Phi(M)}
&\geq
\frac{(1-\norm{I_D - d \Phi}_{op})^2}{\frac{1+\norm{I_D - d \Phi}_{op}}{\tau_{M}} + \norm{d^2 \Phi}_{op}}
\\
&\geq
\tau_{min}
\frac{(1-\norm{I_D - d \Phi}_{op})^2}{1+\norm{I_D - d \Phi}_{op} + \tau_{min} \norm{d^2 \Phi}_{op}}
\geq
\frac{\tau_{min}}{2}
\end{align*}
for $\norm{I_D - d \Phi}_{op}$ and $\norm{d^2 \Phi}_{op}$ small enough.
This shows the stability for $k=2$, as well as that of the reach assumption for $k\geq 3$. 

By now, take $k\geq 3$.
We focus on the existence of a good parametrization of $M'$ around a fixed point $p' = \Phi(p) \in M'$. For $v' \in T_{p'} M' = d_p \Phi \left( T_p M \right)$, let us define 
\begin{align*}
\Psi'_{p'}(v')
&=
\Phi \left( \Psi_p\left( d_{p'} \Phi^{-1}.v' \right) \right)
\\
&=
p' + v' + \mathbf{N}'_{p'}(v')
,
\end{align*}
where $ \mathbf{N}'_{p'}(v') = \left\lbrace \Phi \left( \Psi_p\left( d_{p'} \Phi^{-1}.v' \right) \right) - p' -v' \right\rbrace$.
\begin{center}

\begin{tikzcd}
M	\arrow{rr}{\Phi}& & M' \\
T_p M \arrow{u}{\Psi_p}	\arrow{rr}[swap]{d_p \Phi}& & T_{p'} M' \arrow{u}[swap]{\Psi'_{p'}}
\end{tikzcd}

\end{center}
The maps $\Psi'_{p'}(v')$ and $\mathbf{N}'_{p'}(v')$ are well defined whenever $\norm{d_{p'} \Phi^{-1}.v'} \leq \frac{1}{4 L_\perp}$, so in particular if $\norm{v'} \leq \frac{1}{4 \left(2 L_\perp\right)} \leq \frac{1- \norm{I_D - d \Phi}_{op}}{4L_\perp}$ and $\norm{I_D - d\Phi }_{op} \leq \frac{1}{2}$.
One easily checks that $\mathbf{N}'_{p'}(0) = 0$, $d_0 \mathbf{N}'_{p'} = 0$ and writing $c(v') = p +d_{p'} \Phi^{-1} . v' + \mathbf{N}_{p'}\left( d_{p'} \Phi^{-1} . v'\right)$, for all unit vector $w' \in T_{p'} M'$,

\begin{align*}
\norm{d^2_{v'} \mathbf{N}'_{p'}(\tens{w'}{2})}
&=
\Bigl\Vert
d^2_{c(v')} \Phi 
\left(
\tens{
\left\{
d_{d_{p'} \Phi^{-1} . v'} {\Psi}_{p} \circ d_{p'} \Phi^{-1}. w'
\right\}
}{2}
\right)
\\
& \hspace{2em} +
d_{c(v')} \Phi 
\circ
d^2_{d_{p'} \Phi^{-1} . v'} {\Psi}_{p}
\left(
\tens{
\left\{
d_{p'} \Phi^{-1}. w'
\right\}
}{2}
\right)
\Bigr\Vert
\\
&=
\Bigl\Vert
d^2_{c(v')} \Phi 
\left(
\tens{
\left\{
d_{d_{p'} \Phi^{-1} . v'} {\Psi}_{p} \circ d_{p'} \Phi^{-1}. w'
\right\}
}{2}
\right)
\\
& \hspace{2em} +
\left( d_{c(v')} \Phi -Id \right)
\circ
d^2_{d_{p'} \Phi^{-1} . v'} {\Psi}_{p}
\left(
\tens{
\left\{
d_{p'} \Phi^{-1}. w'
\right\}
}{2}
\right)
\\
& \hspace{2em} +
d^2_{d_{p'} \Phi^{-1} . v'} {\Psi}_{p}
\left(
\tens{
\left\{
d_{p'} \Phi^{-1}. w'
\right\}
}{2}
\right)
\Bigr\Vert
\\
&\leq
\norm{d^2 \Phi}_{op} \left(1 + L_\perp\norm{d_{p'} \Phi^{-1}.v'}\right)^2\norm{d_{p'} \Phi^{-1}.w'}^2
\\
& \hspace{2em} +
\norm{I_D-d \Phi}_{op} L_\perp \norm{d_{p'} \Phi^{-1}.w'}^2
\\
& \hspace{2em} +
L_\perp \norm{d_{p'} \Phi^{-1}.w'}^2
\\
&\leq
\norm{d^2 \Phi}_{op} (1 + 1/4)^2\norm{d_{p'} \Phi^{-1}}_{op}^2
\\
& \hspace{2em} +
\norm{I_D-d \Phi}_{op} L_\perp \norm{d \Phi^{-1}}_{op}^2
\\
& \hspace{2em} +
L_\perp \norm{d_{p'} \Phi^{-1}}_{op}^2.
\end{align*}
Writing further $\norm{d \Phi^{-1}}_{op} \leq (1 - \norm{I_D-d\Phi}_{op})^{-1} \leq 1+2 \norm{I_D-\Phi}_{op}$ for $ \norm{I_D-d\Phi}_{op}$ small enough depending only on $L_\perp$, it is clear that the right-hand side of the latter inequality goes below $2 L_\perp$ for $\norm{ I_D -d \Phi}_{op}$ and $\norm{d^2 \Phi}_{op}$ small enough.
Hence, for $\norm{I_D -d \Phi}_{op}$ and $\norm{d^2 \Phi}_{op}$ small enough depending only on $L_\perp$, $\Vert{d^2_{v'} \mathbf{N}'_{p'}}\Vert_{op} \leq 2L_\perp$ for all $\norm{v'} \leq \frac{1}{4(2L_\perp)}$.
From the chain rule, the same argument applies for the order $3 \leq i \leq k$ differential of $\mathbf{N}'_{p'}$.
\end{proof}

%

Lemma \ref{change_of_variable_hausdorff} deals with the condition on the density in the models $\mathcal{P}^k$. It gives a change of variable formula for pushforward of measure on submanifolds, ensuring a control on densities with respect to intrinsic volume measure.

\begin{Alem}[Change of variable for the Hausdorff measure]
\label{change_of_variable_hausdorff}
Let $P$ be a probability distribution on $M \subset \R^D$ with density $f$ with respect to the $d$-dimensional Hausdorff measure $\mathcal{H}^d$.
Let $\Phi : \R^D \rightarrow \R^D$ be a global diffeomorphism such that $\norm{I_D - d \Phi}_{\mathrm{op}} <1/3$. Let $P' = \Phi_{\ast} P$ be the pushforward of $P$ by $\Phi$. Then $P'$ has a density $g$ with respect to $\mathcal{H}^d$. This density can be chosen to be, for all $z \in \Phi(M)$,
\[
g(z) = \frac{f\left(\Phi^{-1} \left(z\right)\right)}{\sqrt{ \det \left(\pi_{T_{\Phi^{-1}(z)} M} \circ d_{\Phi^{-1}(z)}  \Phi^T \, \circ  \, \restriction{ d_{\Phi^{-1}(z)} \Phi }{ {T_{\Phi^{-1}(z)} M} } \right)}}.
\]
In particular, if $f_{min} \leq f \leq f_{max}$ on $M$, then for all $z \in \Phi(M)$,
\begin{align*}
\left( 1 - 3d/2 \norm{I_D - d \Phi}_{\mathrm{op}}\right) f_{min} 
\leq 
g(z) 
\leq 
f_{max} \left( 1 + 3( 2^{d/2} - 1 )\norm{I_D - d \Phi}_{\mathrm{op}}\right)
.
\end{align*}
\end{Alem}
\begin{proof}[Proof of Lemma \ref{change_of_variable_hausdorff}]
Let $p \in M$ be fixed and $A \subset \mathcal{B}(p,r)\cap M$ for $r$ small enough. For a differentiable map $h: \R^d \rightarrow \R^D$  and for all $x\in \R^d$, we let $J_h (x)$ denote the $d$-dimensional Jacobian $J_h (x) = \sqrt{\det \left(d_x h^T \, d_x h\right)}$.
 The area formula (\cite[Theorem 3.2.5]{Federer69}) states that if $h$ is one-to-one, 
\[ \int_{A} u\left(h(x)\right) J_h(x) \lambda^d(d x) = \int_{h(A)} u(y) \mathcal{H}^d(d y),
\]
whenever $u : \R^D \rightarrow \R$ is Borel, where $\lambda^d$ is the Lebesgue measure on $\R^d$. By definition of the pushforward, and since $d P = f d \mathcal{H}^d$, 

\[ \int_{\Phi(A)} d P'(z) = \int_{A} f(y) \mathcal{H}^d(d y).
\]
Writing $\Psi_p = \exp_p: T_p M \rightarrow \R^D$ for the exponential map of $M$ at $p$, we have
\begin{align*}
\int_{A} f(y) \mathcal{H}^d(d y)
&= \int_{{\Psi_p}^{-1}(A)} f({\Psi_p}(x)) J_{\Psi_p}(x) \lambda^d(d x)
.
\end{align*}
Rewriting the right hand term, we apply the area formula again with $h = \Phi \circ {\Psi_p}$,
\begin{align*}
\int_{{\Psi_p}^{-1}(A)} 
&
f({\Psi_p}(x)) J_{\Psi_p}(x) \lambda^d(d y)
\\
&= \int_{{\Psi_p}^{-1}(A)} f\left(\Phi^{-1} \left(h(x)\right)\right) \frac{J_{\Psi_p}(h^{-1}\left(h(x)\right))}{J_{\Phi\circ {\Psi_p}}(h^{-1}\left(h(x)\right))}J_{\Phi\circ {\Psi_p}}(x) \lambda^d(d x) \\
&= \int_{\Phi(A)} f\left(\Phi^{-1} \left(z\right)\right) \frac{J_{\Psi_p}(h^{-1}\left(z\right))}{J_{\Phi\circ {\Psi_p}}(h^{-1}\left(z\right))}\mathcal{H}^d(d z).
\end{align*}
Since this is true for all $A \subset \mathcal{B}(p,r)\cap M$, $P'$ has a density $g$ with respect to $\mathcal{H}^d$, with
\[
g(z) = f\left(\Phi^{-1} \left(z\right)\right) \frac{J_{\Psi_{\Phi^{-1} \left(z\right)}}(\Psi_{\Phi^{-1} \left(z\right)}^{-1} \circ \Phi^{-1} \left(z\right))}{J_{\Phi\circ \Psi_{\Phi^{-1} \left(z\right)}}(\Psi_{\Phi^{-1} \left(z\right)}^{-1} \circ \Phi^{-1}\left(z\right))}.
\]
Writing $p = \Phi^{-1}(z)$, it is clear that $\Psi_{\Phi^{-1} \left(z\right)}^{-1} \circ \Phi^{-1}\left(z\right)= \Psi_{p}^{-1} (p) = 0 \in T_p M$. Since $d_0 \exp_p: T_p M \rightarrow \R^D$ is the inclusion map,  we get the first statement. 

We now let $B$ and $\pi_T$ denote $d_p \Phi$ and $\pi_{T_p M}$  respectively. For any unit vector $v\in T_p M$,
\begin{align*}
\left|\norm{\pi_T B^T B v}-\norm{v} \right|
&\leq \norm{\pi_T \left(B^T B - I_D\right) v}\\
&\leq \norm{B^T B - I_D}_{\mathrm{op}} \\
&\leq \left( 2 + \norm{I_D - B}_{\mathrm{op}} \right) \norm{I_D - B}_{\mathrm{op}}\\
&\leq 3 \norm{I_D - B}_{\mathrm{op}}.
\end{align*}
Therefore, $ 1 -  3 \norm{I_D - B}_{\mathrm{op}} \leq \norm{ \pi_T B^T \restriction{B}{T_p M}}_{\mathrm{op}}\leq 1 +  3 \norm{I_D - B}_{\mathrm{op}}$. Hence,
\begin{align*}
\sqrt{\det \left( \pi_T B^T \restriction{B}{T_p M} \right)} 
\leq \left(  1 +  3 \norm{I_D - B}_{\mathrm{op}} \right)^{d/2}
\leq \frac{1}{1- \frac{3d}{2}\norm{I_D - B}_{\mathrm{op}}},
\end{align*}
and
\begin{align*}
\sqrt{\det \left( \pi_T B^T \restriction{B}{T_p M} \right)} 
\geq \left(  1 -  3 \norm{I_D - B}_{\mathrm{op}} \right)^{d/2}
\geq \frac{1}{1 + 3 (2^{d/2}-1)\norm{I_D - B}_{\mathrm{op}}},
\end{align*}
which yields the result.
\end{proof}

\newpage
\section*{Appendix B: Some Probabilistic Tools}
\label{sec:appendix:proba_tools}
\addtocounter{section}{1}
\setcounter{subsection}{0}
\subsection{Volume and Covering Rate}\label{subsec:B1}
    The first lemma of this section gives some details about the covering rate of a manifold with bounded reach.

\begin{Alem}\label{Alem:volume+hausdorff_density}
Let $P_0 \in \mathcal{P}^k$ have support $M \subset \R^D$. Then for all $r \leq \tau_{min}/4$ and $x$ in $M$,
\begin{align*}
c_d f_{min} r^d
\leq
p_x(r)
\leq
C_d f_{max} r^d,
\end{align*}
for some $c_d,C_d >0$, with $p_x(r)= P_0 \bigl(\mathcal{B}(x,r) \bigr)$.

Moreover, letting $h = \left( \frac{C'_d k}{f_{min}}  \frac{\log n}{n}\right)^{1/d}$ with $C'_d$ large enough, the following holds.
For $n$ large enough so that $h \leq \tau_{min}/4$, with probability at least $1 - \left( \frac{1}{n}\right)^{k/d}$,
\begin{align*}
d_H\left(M,\Y_n\right)
\leq
h/2.
\end{align*}

\end{Alem}

\begin{proof}[Proof of Lemma \ref{Alem:volume+hausdorff_density}]
Denoting by $\mathcal{B}_M(x,r)$ the geodesic ball of radius $r$ centered at $x$, Proposition 25 of \cite{AamariLevrard2015} yields
\begin{align*}
\mathcal{B}_M(x,r)
\subset
\mathcal{B}(x,r) \cap M
\subset
\mathcal{B}_M(x,6r/5).
\end{align*}
Hence, the bounds on the Jacobian of the exponential map given by Proposition 27 of \cite{AamariLevrard2015} yield
\begin{align*}
c_d r^d
\leq
Vol\bigl(\mathcal{B}(x,r) \cap M \bigr)
\leq
C_d r^d,
\end{align*}
for some $c_d,C_d >0$. Now, since $P$ has a density $f_{min} \leq f \leq f_{max}$ with respect to the volume measure of $M$, we get the first result.

Now we notice that since $ p_x(r) \geq c_d f_{min} r^d$, Theorem 3.3 in \cite{Chazal2013} entails, for $s \leq \tau_{min}/8$,
\begin{align*}
\mathbb{P}
\left(
d_H\bigl(M,\X_n\bigr)
\geq
s
\right)
\leq 
\frac{4^d}{c_d f_{min} s^d}
\exp \left( - \frac{c_d f_{min}}{2^d} n s^d \right)
.
\end{align*}
Hence, taking $s=h/2$, and $h = \left( \frac{C'_d k}{f_{min}}  \frac{\log n}{n}\right)^{1/d}$ with $C'_d$ so that
$
C'_d
\geq
\frac{8^d}{c_d k}
\vee
\frac{2^d(1+k/d)}{c_d k}
$
yields the result. Since $k\geq 1$, taking $C'_d = \frac{8^d}{c_d}$ is sufficient.
\end{proof}

\subsection{Concentration Bounds for Local Polynomials}\label{subsec:B2}
     This section is devoted to the proof of the following proposition.
\begin{Aprop}\label{polminoration}
      Set $h= \left ( K \frac{\log n }{n-1} \right )^{\frac{1}{d}}$. There exist constants $\kappa_{k,d}$, $c_{k,d}$ and $C_d$ such that, if $K \geq (\kappa_{k,d} f_{max}^2 /f_{min}^3)$ and $n$ is large enough so that $3h/2 \leq h_0 \leq \tau_{min}/4$, then with probability at least $1-\left ( \frac{1}{n} \right )^{\frac{k}{d}+1}$, we have
      \[
      \begin{array}{@{}ccc}
      P_{0,n-1} [S^2(\pi^*(x)) \mathbbm{1}_{\mathcal{B}(h/2)}(x)] &\geq& c_{k,d} h^d f_{min} \|S_h\|_2^2, \\
      N(3h/2) & \leq & C_d f_{max} (n-1) h^d,
      \end{array}
      \]
      for every $S \in \R^{k}[x_{1:d}]$, where $N(h) = \sum_{j = 2}^n \mathbbm{1}_{\mathcal{B}(0,h)}(Y_j)$. 
\end{Aprop}

        A first step is to ensure that empirical expectations of order $k$ polynomials are close to their deterministic counterparts.

        \begin{Aprop}\label{poldeviations}
      Let $b \leq \tau_{min}/8$. For any $y_0$ $\in$ $M$, we have
      \begin{multline*}\label{eqpoldeviations}
      \mathbb{P} \left [ \sup_{u_1, \hdots, u_k, \varepsilon \in \{0,1\}^k} \left |(P_0-P_{0,n-1}) \prod_{j=1}^{p} \left (\frac{\left\langle u_j,y \right\rangle} {b} \right )^{\varepsilon_j}\mathbbm{1}_{\mathcal{B}(y_0,b)}(y) \right | \right .   \\ 
      \left . \vphantom{\sup_{u_1, \hdots, u_p, \varepsilon \in \{0,1\}^p} \left |(P-P_n) \prod_{j=1}^{k} \left (\frac{\left\langle u_j,y \right\rangle} {b} \right )^{\varepsilon_j}\mathbbm{1}_{B_x(h)}(y) \right |} \geq p_{y_0}(b) \left ( \frac{4 k \sqrt{2 \pi}}{\sqrt{(n-1)p_{y_0}(b)}} + \sqrt{\frac{2t}{(n-1)p_{y_0}(b)}} + \frac{2}{3 (n-1) p_{y_0}(b)} \right ) \right ] \leq e^{-t},
      \end{multline*}
      where $P_{0,n-1}$ denotes the empirical distribution of $n-1$ i.i.d. random variables $Y_i$ drawn from $P_0$.
      \end{Aprop}
      
      \begin{proof}[Proof of Proposition \ref{poldeviations}]
Without loss of generality we choose $y_0=0$ and shorten notation to $\mathcal{B}(b)$ and $p(b)$. Let $\mathcal{Z}$ denote the empirical process on the left-hand side of Proposition \ref{poldeviations}. Denote also by $f_{u,\varepsilon}$ the map $\prod_{j=1}^{k} \left (\frac{\left\langle u_j,y \right\rangle} {b} \right )^{\varepsilon_j}\mathbbm{1}_{\mathcal{B}(b)}(y)$, and let $\mathcal{F}$ denote the set of such maps, for $u_j$ in $\mathcal{B}(1)$ and $\varepsilon$ in $\{ 0,1 \}^k$.

Since $\| f_{u,\varepsilon} \|_{\infty} \leq 1$ and $P f_{u,\varepsilon}^2 \leq p(b)$, the Talagrand-Bousquet inequality (\cite[Theorem 2.3]{Bousquet02}) yields
\begin{align*}
\mathcal{Z} \leq 4 \mathbb{E}\mathcal{Z} + \sqrt{\frac{2 p(b)t}{n-1}} + \frac{2 t}{3 (n-1)},
\end{align*}
with probability larger than $1 - e^{-t}$. It remains to bound $\mathbb{E}\mathcal{Z}$ from above.
\begin{Alem}\label{expectationofmaximalprocess}
       We may write
       \[
       \mathbb{E} \mathcal{Z} \leq \frac{  \sqrt{2 \pi p(b)}}{\sqrt{n-1}} k.
       \]
\end{Alem}
\begin{proof}[Proof of Lemma \ref{expectationofmaximalprocess}]
Let $\sigma_i$ and $g_i$ denote some independent Rademacher and Gaussian variables. For convenience, we denote by $\mathbb{E}_A$ the expectation with respect to the random variable $A$. Using symmetrization inequalities we may write
\begin{align*}
\mathbb{E} \mathcal{Z} &= \mathbb{E}_Y \sup_{u,\varepsilon} \left |(P_0-P_{0,n-1}) \prod_{j=1}^{k} \left ( \frac{\left\langle u_j,y \right\rangle}{b} \right )^{\varepsilon_j} \mathbbm{1}_{\mathcal{B}(b)}(y) \right | \\
             & \leq \frac{2}{n-1} \mathbb{E}_Y \mathbb{E}_\sigma \sup_{u,\varepsilon} \sum_{i=1}^{n-1}{\sigma_i \prod_{j=1}^{k}\left ( \frac{\left\langle u_j,Y_i \right\rangle}{b} \right )^{\varepsilon_j} \mathbbm{1}_{\mathcal{B}(b)}(Y_i)} \\
             & \leq \frac{\sqrt{2 \pi}}{n-1}\mathbb{E}_Y \mathbb{E}_g \sup_{u,\varepsilon} \sum_{i=1}^{n-1}{g_i \prod_{j=1}^{k}\left ( \frac{\left\langle u_j,Y_i \right\rangle}{b} \right )^{\varepsilon_j} \mathbbm{1}_{\mathcal{B}(b)}(Y_i)}.
\end{align*}

         Now let $\mathcal{Y}_{u,\varepsilon}$ denote the Gaussian process $\sum_{i=1}^{n-1}{g_i \prod_{j=1}^{k}\left ( \frac{\left\langle u_j,Y_i \right\rangle}{b} \right )^{\varepsilon_j}} \mathbbm{1}_{\mathcal{B}(b)}(Y_i)$. Since, for any $y$ in $\mathcal{B}(b)$, $u$,$v$ in $\mathcal{B}(1)^k$, and $\varepsilon$, $\varepsilon'$ in $\{0,1\}^k$, we have
         \begin{multline*}
\left | \prod_{j=1}^{k} \left ( \frac{\left\langle y , u_j \right\rangle }{b} \right )^{\varepsilon_j} - \prod_{j=1}^{k} \left ( \frac{\left\langle y , v_j \right\rangle }{b} \right )^{\varepsilon'_j} \right |  
\\ 
\begin{aligned}[t] 
& \begin{multlined}[t] \leq 
\left | \sum_{r=1}^{k} \left (\prod_{j=1}^{k+1-r} \left ( \frac{\left\langle y , u_j \right\rangle }{b} \right )^{\varepsilon_j} \prod_{j=k+2-r}^{k} \left ( \frac{\left\langle y , v_j \right\rangle }{b} \right )^{\varepsilon'_j} \right . \right .\\
\left . \left .  - \prod_{j=1}^{k-r} \left ( \frac{\left\langle y , u_j \right\rangle }{b} \right )^{\varepsilon_j} \prod_{j=k+1-r}^{k} \left ( \frac{\left\langle y , v_j \right\rangle }{b} \right )^{\varepsilon'_j} \right ) \right | 
\end{multlined}
\\
& \leq \begin{multlined}[t] 
\sum_{r=1}^{k} \left | \prod_{j=1}^{k-r} \left ( \frac{\left\langle y , u_j \right\rangle }{b} \right )^{\varepsilon_j} \prod_{j=k+2-r}^{k} \left ( \frac{\left\langle y , v_j \right\rangle }{b} \right )^{\varepsilon'_j} \left [ \left ( \frac{\left\langle u_{k+1-r},y \right\rangle}{b} \right )^{\varepsilon_{k+1-r}} \right . \right . \\
 \left . \left . - \left ( \frac{\left\langle v_{k+1-r},y \right\rangle}{b} \right )^{\varepsilon'_{k+1-r}} \right ] \right |
 \end{multlined} 
\\
&\leq 
\sum_{r=1}^{k} \left | \frac{\left\langle \varepsilon_r u_r - \varepsilon'_r v_r,y \right\rangle}{b} \right |,          
\end{aligned}
\end{multline*}
we deduce that 
      \begin{align*}
      \mathbb{E}_g (\mathcal{Y}_{u,\varepsilon} - \mathcal{Y}_{v,\varepsilon'})^2 & \leq k \sum_{i=1}^{n-1} \sum_{r=1}^{k} \left ( \frac{\left\langle \varepsilon_r u_r, Y_i \right\rangle}{b} - \frac{\left\langle \varepsilon'_r v_r, Y_i \right\rangle}{b}\right )^2 \mathbbm{1}_{\mathcal{B}(b)}(Y_i) \\
                  & \leq \mathbb{E}_g (\Theta_{u,\varepsilon} - \Theta_{v,\varepsilon'})^2, 
      \end{align*}
      where $\Theta_{u,\varepsilon} = \sqrt{k}\sum_{i=1}^{n-1}\sum_{r=1}^{k} g_{i,r} \frac{\left\langle \varepsilon_r u_r, Y_i \right\rangle}{b} \mathbbm{1}_{\mathcal{B}(b)}(Y_i)$. According to Slepian's Lemma \cite[Theorem 13.3]{Massart13}, it follows that
      \begin{align*}
      \mathbb{E}_g \sup_{u,\varepsilon} \mathcal{Y}_g & \leq  \mathbb{E}_g \sup_{u,\varepsilon} \Theta_{u,\varepsilon} \\
                                            & \leq  \sqrt{k} \mathbb{E}_g \sup_{u,\varepsilon}  \sum_{r=1}^{k} \frac{\left\langle \varepsilon_r u_r, \sum_{i=1}^{n-1} g_{i,r} \mathbbm{1}_{\mathcal{B}(b)}(Y_i) Y_i \right\rangle}{b} \\
                                            & \leq  \sqrt{k} \mathbb{E}_g \sup_{u,\varepsilon} \sqrt{k \sum_{r=1}^{k} \frac{\left\langle \varepsilon_r u_r, \sum_{i=1}^{n-1} g_{i,r} \mathbbm{1}_{\mathcal{B}(b)}(Y_i) Y_i \right\rangle^2}{b^2} }.
         \end{align*}
We deduce that
\begin{align*}
\mathbb{E}_g \sup_{u,\varepsilon} Y_g & \leq  \mathbb{E}_g \sup_{u,\varepsilon} \Theta_g                                      \\
                                            & \leq  k \sqrt{ \mathbb{E}_g \sup_{\|u\|=1, \varepsilon \in \{0,1\}}\frac{\left\langle \varepsilon u, \sum_{i=1}^{n-1} g_{i} \mathbbm{1}_{\mathcal{B}(b)}(Y_i) Y_i \right\rangle^2}{b^2}}  \\
                                            & \leq  k \sqrt{\mathbb{E}_g \left \| \sum_{i=1}^{n-1} \frac{g_i Y_i}{b} \mathbbm{1}_{\mathcal{B}(b)}(Y_i) \right \|^2} \\
                                            & \leq  k \sqrt{N(b)}.
      \end{align*}
       Then we can deduce that $\mathbb{E}_X\mathbb{E}_g \sup_{u,\varepsilon} Y_g \leq  k \sqrt{p(b)}$.
\end{proof}
          Combining Lemma \ref{expectationofmaximalprocess} with Talagrand-Bousquet's inequality gives the result of Proposition \ref{poldeviations}.
\end{proof}

         We are now in position to prove Proposition \ref{polminoration}.
\begin{proof}[Proof of Proposition \ref{polminoration}]
If $ h/2 \leq \tau_{min}/4$, then, according to Lemma \ref{Alem:volume+hausdorff_density}, $p( h/2) \geq c_d f_{min} h^d$, hence, if $h= \left ( K \frac{\log(n)}{n-1} \right )^{\frac{1}{d}}$, $(n-1)p(h/2) \geq K c_{d}f_{min} \log(n)$. Choosing $b=h/2$ and $t =  (k/d +1) \log(n) + \log(2)$ in Proposition \ref{poldeviations} and $ K = K'/f_{min}$, with $K'>1$  leads to
\begin{multline*}
\mathbb{P} \left [ \sup_{u_1, \hdots, u_k, \varepsilon \in \{0,1\}^k} \left |(P_0-P_{0,n-1}) \prod_{j=1}^{k} \left (2 \frac{\left\langle u_j,y \right\rangle} {h} \right )^{\varepsilon_j}\mathbbm{1}_{\mathcal{B}(y_0,h/2)}(y) \right |  \right . \\
\left . \vphantom{\sup_{u_1, \hdots, u_k, \varepsilon \in \{0,1\}^k} \left |(P-P_{n-1}) \prod_{j=1}^{k} \left (\frac{\left\langle u_j,y \right\rangle} {h} \right )^{\varepsilon_j}\mathbbm{1}_{B_x(h)}(y) \right |} \geq \frac{c_{d,k}f_{max}}{\sqrt{K'}} h^{d} \right ] \leq \frac{1}{2}\left ( \frac{1}{n} \right )^{\frac{k}{d}+1}.
\end{multline*}
     On the complement of the probability event mentioned just above, for a polynomial $S= \sum_{\alpha \in [0,k]^d | |\alpha| \leq k} a_{\alpha} y_{1:d}^\alpha$, we have
     \begin{align*}
     (P_{0,n-1}-P_0) S^2(y_{1:d}) \mathbbm{1}_{\mathcal{B}(h/2)}(y) & \geq - \sum_{ \alpha, \beta}\frac{c_{d,k}f_{max}}{\sqrt{K'}} |a_\alpha a_\beta| h^{d+|\alpha| + |\beta|} \\
     & \geq - \frac{c_{d,k}f_{max}}{\sqrt{K'}} h^d \|S_h\|_2^2.
     \end{align*}
       On the other hand, we may write, for all $r>0$ ,      
       \[
 \int_{\mathcal{B}(0,r)} S^2(y_{1:d})dy_1 \hdots dy_d \geq C_{d,k} r^d \|S_r\|^2_2,
 \]
  for some constant $C_{d,k}$. It follows that  
  \[
  P_0 S^2(y_{1:d}) \mathbbm{1}_{\mathcal{B}(h/2)}(y) \geq P_0 S^2(y_{1:d}) \mathbbm{1}_{B(7h/16)}(y_{1:d}) \geq c_{k,d} h^d f_{min} \| S_h \|_2^2
  ,
  \]
   according to Lemma \ref{Alem:fourre_tout_geom_model}. Then we may choose $K' = \kappa_{k,d} (f_{max}/f_{min})^2$, with $\kappa_{k,d}$ large enough so that  
       \[
       P_{0,n-1} S^2(x_{1:d}) \mathbbm{1}_{\mathcal{B}(h/2)}(y) \geq c_{k,d}f_{min} h^d \|S_h\|^2_2.
       \]
       
       The second inequality of Proposition \ref{polminoration} is derived the same way from Proposition \ref{poldeviations}, choosing $\varepsilon = \left ( 0, \hdots, 0 \right )$, $b=3h/2$ and $h \leq \tau_{min}/8$ so that $b \leq \tau_{min}/4$.
\end{proof}

\newpage
\section*{Appendix C: Minimax Lower Bounds}
\label{sec:appendix:lower_bounds}
\addtocounter{section}{1}
\setcounter{subsection}{0}

\subsection{Conditional Assouad's Lemma}\label{subsec:C1}
This section is dedicated to the proof of Lemma 7, reproduced below as Lemma \ref{conditional_assouad}.
\begin{Alem}[Conditional Assouad]
\label{conditional_assouad}
Let $m \geq 1$ be an integer and let 
$\left\{ \mathcal{Q}_\tau\right\}_{\tau\in \{0,1\}^m}$ be a family of $2^m$ submodels $\mathcal{Q}_\tau \subset \mathcal{Q}$.
Let $\left\{U_k \times U'_k\right\}_{1 \leq k \leq m}$ be a family of pairwise disjoint subsets of $\mathcal{X} \times \mathcal{X}'$, and $\mathcal{D}_{\tau,k}$ be subsets of $\mathcal{D}$. Assume that
for all $\tau \in \left\{ 0,1 \right\}^m$ and $1 \leq k \leq m$,
\begin{itemize}
\item for all $Q_\tau \in \mathcal{Q}_\tau$, $\theta_X(Q_\tau) \in \mathcal{D}_{\tau,k}$ on the event $\left\{ X \in U_k \right\}$;
\item for all $\theta \in \mathcal{D}_{\tau,k}$ and $\theta' \in \mathcal{D}_{\tau^k,k}$, $d(\theta,\theta')\geq \Delta$.
\end{itemize}
For all $\tau \in \left\{0,1 \right\}^m$, let $\overline{Q}_\tau \in \overline{Conv} (\mathcal{Q}_\tau)$, and write $\bar{\mu}_\tau$ and $\bar{\nu}_\tau$ for the marginal distributions of $\overline{Q}_\tau$ on $\mathcal{X}$ and $\mathcal{X}'$ respectively. Assume that if $(X,X')$ has distribution $\overline{Q}_\tau$, $X$ and $X'$ are independent conditionally on the event $\left\{(X,X')\in U_k \times U'_k \right\}$, and that
\begin{align*}
\min_{\mathclap{\substack{\tau \in \left\{ 0,1 \right\}^m \\ 1 \leq k \leq m}}}
~~
\left\{
\left( 
\int_{U_k} d \bar{\mu}_\tau \wedge d \bar{\mu}_{\tau^k}
\right)
\left( 
\int_{U'_k} d \bar{\nu}_\tau \wedge d \bar{\nu}_{\tau^k}
\right)
\right\}
\geq
1- \alpha.
\end{align*}
Then,
\begin{align*}
\inf_{\hat{\theta}} 
\sup_{Q \in \mathcal{Q}}
\E_Q
\left[
d
	\bigl( 
	\theta_X(Q)
	,
	\hat{\theta}(X,X')
	\bigr)
\right]
&\geq
m
\frac{\Delta}{2} 
(1- \alpha)
,
\end{align*}
where the infimum is taken over all the estimators $\hat{\theta} : \mathcal{X} \times \mathcal{X}' \rightarrow \mathcal{D}$.
\end{Alem}

\begin{proof}[Proof of Lemma \ref{conditional_assouad}]

The proof follows that of Lemma 2 in \cite{Yu97}. Let $\hat{\theta} = \hat{\theta}(X,X')$ be fixed.
For any family of $2^m$ distributions $\left\{Q_\tau\right\}_\tau \in \left\{\mathcal{Q}_\tau\right\}_{\tau}$, since the $U_k\times U'_k$'s are pairwise disjoint,
\begin{align*}
\sup_{Q \in \mathcal{Q}}
\E_Q
&\left[
d
	\bigl( 
	\theta_X(Q)
	,
	\hat{\theta}(X,X')
	\bigr)
\right]
\\
&\geq
\max_{\tau} \E_{Q_\tau} 
d(
\hat{\theta}
, 
\theta_X(Q_\tau)
)
\\
&\geq
\max_{\tau} \E_{Q_\tau} 
\sum_{k=1}^m 
d
\bigl( 
\hat{\theta}
,
\theta_X(Q_\tau)
\bigr)
\mathbbm{1}_{U_k \times U'_k}(X,X')
\\
&\geq
2^{-m}
\sum_\tau
\sum_{k=1}^m 
\E_{Q_\tau} 
d
\bigl( 
\hat{\theta}
,
\theta_X(Q_\tau)
\bigr)
\mathbbm{1}_{U_k \times U'_k}(X,X')
\\
&\geq
2^{-m}
\sum_\tau
\sum_{k=1}^m 
\E_{Q_\tau} 
d
\bigl( 
\hat{\theta}
,
\mathcal{D}_{\tau,k}
\bigr)
\mathbbm{1}_{U_k \times U'_k}(X,X')
\\
&=
\sum_{k=1}^m 
2^{-(m+1)}
\sum_\tau
\Bigl(
\E_{Q_\tau} 
d
\bigl( 
\hat{\theta}
,
\mathcal{D}_{\tau,k}
\bigr)
\mathbbm{1}_{U_k \times U'_k}(X,X')
\\
&\hspace{25ex}
+
\E_{Q_{\tau^k}} 
d
\bigl( 
\hat{\theta}
,
\mathcal{D}_{\tau^k,k}
\bigr)
\mathbbm{1}_{U_k \times U'_k}(X,X')
\Bigr).
\end{align*}
Since the previous inequality holds for all $Q_\tau\in \mathcal{Q}_\tau$, it extends to $\overline{Q}_\tau \in \overline{Conv}(\mathcal{Q}_\tau)$ by linearity.
Let us now lower bound each of the terms of the sum for fixed $\tau \in \left\{0,1\right\}^m$ and $1 \leq k \leq m$. By assumption, if $(X,X')$ has distribution $\overline{Q}_\tau$, then conditionally on $\left\{(X,X') \in U_k \times U'_k\right\}$, $X$ and $X'$ are independent. Therefore,
\begin{align*}
\E_{\overline{Q}_\tau} 
&d
\bigl( 
\hat{\theta}
,
\mathcal{D}_{\tau,k}
\bigr)
\mathbbm{1}_{U_k \times U'_k}(X,X')
+
\E_{\overline{Q}_{\tau^k}} 
d
\bigl( 
\hat{\theta}
,
\mathcal{D}_{\tau^k,k}
\bigr)
\mathbbm{1}_{U_k \times U'_k}(X,X')
\\
&\geq
\E_{\overline{Q}_\tau} 
d
\bigl( 
\hat{\theta}
,
\mathcal{D}_{\tau,k}
\bigr)
\mathbbm{1}_{U_k}(X)\mathbbm{1}_{U'_k}(X')
+
\E_{\overline{Q}_{\tau^k}} 
d
\bigl( 
\hat{\theta}
,
\mathcal{D}_{\tau^k,k}
\bigr)
\mathbbm{1}_{U_k}(X)\mathbbm{1}_{U'_k}(X')
\\
&=
\E_{\bar{\nu}_\tau}
\left[
\E_{\bar{\mu}_\tau}
\left(
d
\bigl( 
\hat{\theta}
,
\mathcal{D}_{\tau,k}
\bigr)
\mathbbm{1}_{U_k}(X)
\right)
\mathbbm{1}_{U'_k}(X')
\right]
\\
&\hspace{2em}
+
\E_{\bar{\nu}_{\tau^k}}
\left[
\E_{\bar{\mu}_{\tau^k}}
\left(
d
\bigl( 
\hat{\theta}
,
\mathcal{D}_{\tau^k,k}
\bigr)
\mathbbm{1}_{U_k}(X)
\right)
\mathbbm{1}_{U'_k}(X')
\right]
\\
&=
\int_{U_k}
\int_{U'_k}
d(\hat{\theta},\mathcal{D}_{\tau,k})
d \bar{\mu}_{\tau}(x)
d \bar{\nu}_{\tau}(x')
+
\int_{U_k}
\int_{U'_k}
d(\hat{\theta},\mathcal{D}_{\tau^k,k})
d \bar{\mu}_{\tau^k}(x)
d \bar{\nu}_{\tau^k}(x')
\\
&\geq
\int_{U_k}
\int_{U'_k}
\left(
d(\hat{\theta},\mathcal{D}_{\tau,k})
+
d(\hat{\theta},\mathcal{D}_{\tau^k,k})
\right)
d \bar{\mu}_{\tau} \wedge d \bar{\mu}_{\tau^k}(x)
d \bar{\nu}_{\tau} \wedge d \bar{\nu}_{\tau^k}(x')
\\
&\geq
\Delta
\left( 
\int_{U_k} d \bar{\mu}_\tau \wedge d \bar{\mu}_{\tau^k}
\right)
\left( 
\int_{U'_k} d \bar{\nu}_\tau \wedge d \bar{\nu}_{\tau^k}
\right)
\\
&\geq
\Delta(1-\alpha)
,
\end{align*}
where we used that 
$d(\hat{\theta},\mathcal{D}_{\tau,k})
+
d(\hat{\theta},\mathcal{D}_{\tau^k,k})
\geq
\Delta.
$
The result follows by summing the above bound  $\left|\left\{1,\ldots,m\right\}\times\left\{0,1\right\}^m \right|= m 2^m$ times.
\end{proof}

\subsection{Construction of Generic Hypotheses}
\label{subsec:generic_hypotheses}\label{subsec:C2}

%

Let $M_0^{(0)}$ be a $d$-dimensional $\mathcal{C}^\infty$-submanifold of $\R^D$ with reach greater than $1$ and such that it contains $\mathcal{B}_{\R^d \times \{0\}^{D-d}}(0,1/2).$ 
$M_0^{(0)}$ can be built for example by flattening smoothly a unit $d$-sphere in $\R^{d+1} \times \{0\}^{D-d-1}$.
Since $M_{0}^{(0)}$ is $\mathcal{C}^\infty$, the uniform probability distribution $P_0^{(0)}$ on $M_{0}^{(0)}$ belongs to 
$
\mathcal{P}^{k}_{1,\mathbf{L}^{(0)},1/V_0^{(0)},1/V_0^{(0)}}
$, for some $\mathbf{L}^{(0)}$ and $V_0^{(0)} = Vol(M_0^{(0)})$.

Let now $M_0 = (2\tau_{min}) M_{0}^{(0)}$ be the submanifold obtained from $M_{0}^{(0)}$ by homothecy.
By construction, and from Proposition \ref{statistical_model_stability}, we have 
\begin{align*}
\tau_{M_0}\geq 2\tau_{min},
\qquad
\mathcal{B}_{\R^d\times \{0\}^{D-d}}(0,\tau_{min}) \subset M_0,
\qquad
Vol(M_0) = C_d \tau_{min}^d,
\end{align*}
and the uniform probability distribution $P_0$ on $M_0$ satisfies 
\begin{align*}
P_0 \in \mathcal{P}^{k}_{2\tau_{min},\mathbf{L}/2,2f_{min},f_{max}/2},
\end{align*}
whenever
$L_{\perp}/2 \geq L_{\perp}^{(0)}/(2 \tau_{min})$, $\ldots$, $L_{k}/2 \geq L_k^{(0)}/(2\tau_{min})^{k-1}$, and provided that  $2f_{min}\leq
\bigl((2\tau_{min})^dV_0^{(0)}\bigr)^{-1} \leq f_{max}/2$. 
Note that $L_\perp^{(0)},$ $\ldots,L_k^{(0)}$, $Vol(M_0^{(0)})$ depend only on $d$ and $k$.
For this reason, all the lower bounds will be valid for 
$\tau_{min}{L_\perp},\ldots,\tau_{min}^{k-1}{L_k},(\tau_{min}^df_{min})^{-1}$ and ${\tau_{min}^d}{f_{max}}$ large enough to exceed the thresholds
$L_{\perp}^{(0)}/2,\ldots, L_k^{(0)}/2^{k-1}$, $2^dV_0^{(0)}$ and $(2^dV_0^{(0)})^{-1}$ respectively.


For $0 < \delta \leq  \tau_{min}/4$, let $x_1,\ldots,x_{m} \in M_0\cap \mathcal{B}(0,\tau_{min}/4)$ be a family of points such that
\begin{align*}
\mbox{for} \quad 1 \leq k \neq k' \leq m, \quad \norm{x_k - x_{k'}} \geq \delta.
\end{align*}
For instance, considering the family $\left\{ \bigl( l_1 \delta,\ldots,l_d \delta,0,\ldots,0 \bigr)\right\}_{l_i \in \mathbb{Z}, |l_i| \leq \lfloor \tau_{min}/(4\delta) \rfloor}$,
\begin{align*}
m \geq c_d \left(\frac{\tau_{min}}{\delta} \right)^d,
\end{align*} 
for some $c_d >0$.

We let $e \in \R^D$ denote the $(d+1)$th vector of the canonical basis. In particular, we have the orthogonal decomposition of the ambient space
\begin{align*}
\R^D = 
\bigl(\R^d \times \left\{ 0 \right\}^{D-d}\bigr)
+ 
span(e) 
+ 
\bigl(\left\{ 0 \right\}^{d+1} \times \R^{D-d-1}\bigr)
.
\end{align*}

Let $\phi:\R^D\rightarrow [0,1]$ be a smooth scalar map such that 
$\left.\phi\right|_{\mathcal{B}\left(0,\frac{1}{2}\right)} = 1
\text{ and }
\left.\phi\right|_{\mathcal{B}\left(0,1\right)^c} = 0
.$

Let $\Lambda_+>0$ and $1 \geq A_+>A_->0$ be real numbers to be chosen later. Let $\mathbf{\Lambda} = \left(\Lambda_1,\ldots,\Lambda_{m}\right)$ with entries $-\Lambda_+ \leq \Lambda_k \leq \Lambda_+$,
and $\mathbf{A} = \left(A_1,\ldots,A_{m}\right)$  with entries $A_- \leq A_k \leq A_+$.
For $z \in \R^D$, we write $z = (z_1,\ldots,z_D)$ for its coordinates in the canonical basis.
For all $\tau = \left(\tau_1,\ldots,\tau_m\right) \in \left\{0,1\right\}^m$, define the bump map as

\begin{equation}\label{multibump_definition}
\Phi^{\mathbf{\Lambda},\mathbf{A},i}_{\tau}(x) = x + \sum_{k=1}^m \phi\left(\frac{x-x_k}{\delta}\right) \left\{ \tau_k  A_k (x-x_k)_1^i + (1-\tau_k)\Lambda_k \right\} e.
\end{equation}
An analogous deformation map was considered in \cite{AamariLevrard2015}.
We let $P^{\mathbf{\Lambda},\mathbf{A},(i)}_{\tau}$ denote the pushforward distribution of $P_0$ by $\Phi^{\mathbf{\Lambda},\mathbf{A},(i)}_{\tau}$, and write $M^{\mathbf{\Lambda},\mathbf{A},(i)}_{\tau}$ for its support. Roughly speaking, $M^{\mathbf{\Lambda},\mathbf{A},i}_{\tau}$ consists of $m$ bumps at the $x_k$'s having different shapes (Figure \ref{fig:bumps_all}). If $\tau_k = 0$, the bump at $x_k$ is a symmetric plateau function and has height $\Lambda_k$. If $\tau_k = 1$, it fits the graph of the polynomial $A_k (x-x_k)_1^i$ locally.
\begin{figure}[h]
	\begin{center}
		\subfloat[Flat bump: $\tau_k = 0$.]{
			\includegraphics[width = 0.5\textwidth]{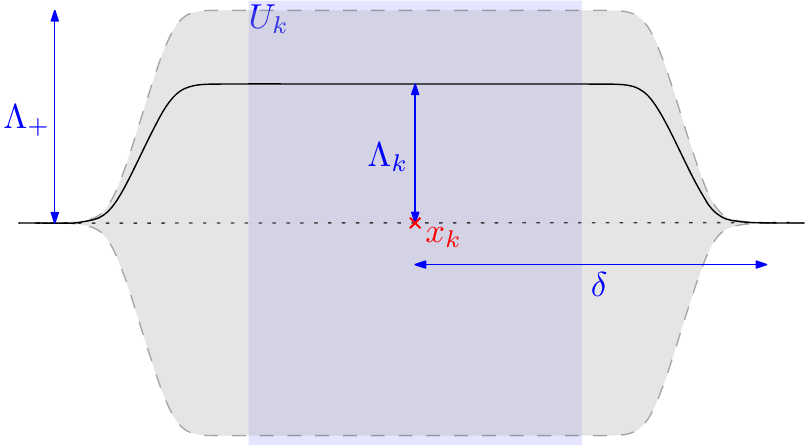}
			\label{subfig:bump_flat}
		}

		\subfloat[Linear bump: $\tau_k = 1$, $i = 1$.]{		
			\includegraphics[width = 0.5\textwidth]{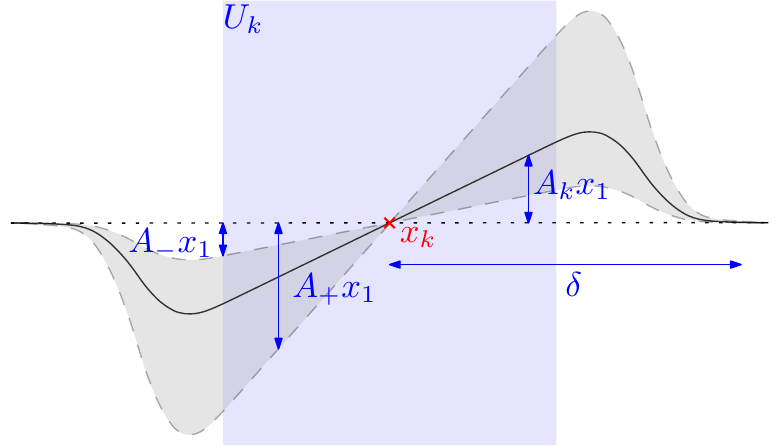}
			\label{subfig:bump_linear}
		}
		\subfloat[Quadratic bump: $\tau_k = 1$, $i = 2$.]{
			\includegraphics[width = 0.5\textwidth]{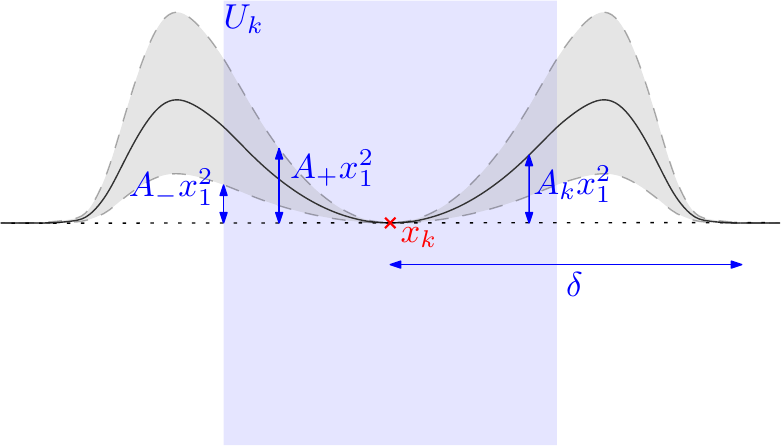}
			\label{subfig:bump_quadratic}
		}

	\end{center}	
	\caption{
		The three shapes of the bump map $\Phi^{\mathbf{\Lambda},\mathbf{A},i}_{\tau}$ around $x_k$.
	}
	\label{fig:bumps_all}
\end{figure}
The following Lemma \ref{differential_bounds_on_multiple_bumps} gives differential bounds and geometric properties of $\Phi^{\mathbf{\Lambda},\mathbf{A},i}_{\tau}$.
\begin{Alem}\label{differential_bounds_on_multiple_bumps}
There exists $c_{\phi,i} < 1$ such that if $A_+ \leq c_{\phi,i} \delta^{i-1}$ and $\Lambda_+ \leq c_{\phi,i} \delta$, then $\Phi^{\mathbf{\Lambda},\mathbf{A},i}_{\tau}$ is a global $\mathcal{C}^\infty$-diffeomorphism of $\R^D$ such that for all $1 \leq k \leq m$, $\Phi^{\mathbf{\Lambda},\mathbf{A},i}_{\tau}\left( \mathcal{B}(x_k,\delta ) \right) = \mathcal{B}(x_k,\delta ) $. Moreover,
\begin{align*}
\norm{I_D - d \Phi^{\mathbf{\Lambda},\mathbf{A},i}_{\tau}}_{op}
\leq
C_{i}
\left\{ \frac{A_+}{\delta^{1-i}} \right\}
\vee
\left\{ \frac{\Lambda_+}{\delta} \right\},
\end{align*}
and for $j \geq 2$,
\begin{align*}
\norm{d^j \Phi^{\mathbf{\Lambda},\mathbf{A},i}_{\tau}}_{op}
\leq
C_{i,j}
\left\{ \frac{A_+}{\delta^{j-i}} \right\}
\vee
\left\{ \frac{\Lambda_+}{\delta^{j}} \right\}
.
\end{align*}
\end{Alem}

\begin{proof}[Proof of Lemma \ref{differential_bounds_on_multiple_bumps}]
Follows straightforwardly from chain rule, similarly to Lemma 11 in \cite{AamariLevrard2015}.
\end{proof}

\begin{Alem}
\label{Alem:generic_hypotheses_properties}
If $\tau_{min}{L_\perp},\ldots,\tau_{min}^{k-1}{L_k},({\tau_{min}^d}f_{min})^{-1}$ and ${\tau_{min}^d}{f_{max}}$ are large enough (depending only on $d$ and $k$),
then provided that $\Lambda_+ \vee A_+\delta^i \leq c_{k,d,\tau_{min}} \delta^k$,
for all $\tau \in \{0,1\}^m$, 
$P^{\mathbf{\Lambda},\mathbf{A},i}_{\tau} \in \mathcal{P}^{k}_{\tau_{min},\mathbf{L},f_{min},f_{max}}
$
\end{Alem}

\begin{proof}[Proof of Lemma \ref{Alem:generic_hypotheses_properties}]
Follows using the stability of the model Lemma \ref{statistical_model_stability} applied to the distribution 
$
P_0 \in \mathcal{P}^{k}_{2\tau_{min},\mathbf{L}/2,2f_{min},f_{max}/2}
$
and the map $\Phi^{\mathbf{\Lambda},\mathbf{A},i}_{\tau}$,  of which differential bounds are asserted by Lemma \ref{differential_bounds_on_multiple_bumps}.
\end{proof}

\subsection{Hypotheses for Tangent Space and Curvature}\label{subsec:C3}

\subsubsection{Proof of Lemma 8}

This section is devoted to the proof of Lemma 8, for which we first derive two slightly more general results, with parameters to be tuned later.
The proof is split into two intermediate results Lemma \ref{Alem:hypotheses_properties} and Lemma \ref{Alem:hypotheses_TV_estimates}.

Let us write $\bar{Q}^{(i)}_{\tau,n}$ for the mixture distribution on $(\R^D)^n$ defined by
\begin{align}\label{eqn:mixture_definition}
\bar{Q}^{(i)}_{\tau,n}
&=
\int_{[-\Lambda_+,\Lambda_+]^m}
\int_{[A_-,A_+]^m}
\left( P^{\mathbf{\Lambda},\mathbf{A},(i)}_{\tau} \right)^{\otimes n}
\frac{d \mathbf{A}}{\left(A_+ - A_-\right)^m}
\frac{d \mathbf{\Lambda}}{\left(2 \Lambda_+\right)^m}.
\end{align}
Although the probability distribution $\bar{Q}^{(i)}_{\tau,n}$ depends on $A_-,A_+$ and $\Lambda_+$, we omit this dependency for the sake of compactness.
Another way to define $\bar{Q}^{(i)}_{\tau,n}$ is the following: draw uniformly $\mathbf{\Lambda}$ in $[-\Lambda_+,\Lambda_+]^m$ and $\mathbf{A}$ in $[A_-,A_+]^m$, and given $\left(\mathbf{\Lambda},\mathbf{A}\right)$, take $Z_i = \Phi^{\mathbf{\Lambda},\mathbf{A},i}_{\tau}\left(Y_i\right)$, where $Y_1, \ldots, Y_n$ is an i.i.d. $n$-sample with common distribution $P_0$ on $M_0$.
Then $\left(Z_1, \ldots, Z_n\right)$ has distribution $\bar{Q}^{(i)}_{\tau,n}$.

\begin{Alem}\label{Alem:hypotheses_properties}
Assume that the conditions of Lemma \ref{differential_bounds_on_multiple_bumps} hold, and let
\begin{align*}
U_k = 
\mathcal{B}_{\R^d \times \{0\}^{D-d}} \left(x_k,\delta/2\right)
+
\mathcal{B}_{span(e)}(0, {\tau_{min}}/{2}),
\end{align*}
and
\begin{align*}
U'_k = 
\left( 
\R^D 
\setminus 
\left\{
\mathcal{B}_{\R^d \times \{0\}^{D-d}} \left(x_k,\delta\right)
+
\mathcal{B}_{span(e)}(0, {\tau_{min}}/{2})
\right\}
\right)^{n-1}
.
\end{align*} 
Then the sets $U_k \times U'_k$ are pairwise disjoint, 
$\bar{Q}^{(i)}_{\tau,n} \in \overline{Conv} \bigl(\bigl( \mathcal{P}^{(i)}_{\tau} \bigr)^{\otimes n} \bigr)$, 
and if $\left(Z_1,\ldots,Z_n\right) = \left(Z_1,Z_{2:n}\right)$ has distribution $\bar{Q}^{(i)}_{\tau,n}$, $Z_1$ and $Z_{2:d}$ are independent conditionally on the event $\left\{\left(Z_1,Z_{2:n}\right) \in U_k \times U'_k\right\}$.

Moreover, if $\left(X_1,\ldots,X_n\right)$ has distribution $\bigl({P}^{\mathbf{\Lambda},\mathbf{A},(i)}_{\tau}\bigr)^{\otimes n}$ (with fixed $\mathbf{A}$ and $\mathbf{\Lambda}$), then on the event $\left\{X_1 \in U_k \right\}$, we have:
\begin{itemize}
\item  if $\tau_k = 0$,
\begin{align*}
T_{X_1} {{M}^{\mathbf{\Lambda},\mathbf{A},(i)}_{\tau}}  = \R^d \times \left\{0\right\}^{D-d}
\text{\quad , \quad }
\norm{
II_{X_1}^{{{M}^{\mathbf{\Lambda},\mathbf{A},(i)}_{\tau}} }
\circ
\pi_{T_{X_1} {{M}^{\mathbf{\Lambda},\mathbf{A},(i)}_{\tau}}}
}_{op} = 0
\end{align*}
and $d_H\bigl(M_0, {{M}^{\mathbf{\Lambda},\mathbf{A},(i)}_{\tau}} \bigr) \geq |\Lambda_k|$. 
\item if $\tau_k = 1$,
\begin{itemize}
\item for $i = 1$:
$\angle \left(T_{X_1} {{M}^{\mathbf{\Lambda},\mathbf{A},(1)}_{\tau}} , \R^d \times \left\{0\right\}^{D-d} \right) \geq A_-/2$.
\item for $i = 2$:
$\norm{II_{X_1}^{{{M}^{\mathbf{\Lambda},\mathbf{A},(2)}_{\tau}}}
\circ
\pi_{T_{X_1} {{M}^{\mathbf{\Lambda},\mathbf{A},(2)}_{\tau}}
 }}_{op} \geq A_-/2$. 
\end{itemize}
\end{itemize}
\end{Alem}

\begin{proof}[Proof of Lemma \ref{Alem:hypotheses_properties}]
It is clear from the definition (\ref{eqn:mixture_definition}) that $\bar{Q}^{(i)}_{\tau,n} \in \overline{Conv} \bigl( \bigl( \mathcal{P}^{(i)}_{\tau} \bigr)^{\otimes n}\bigr)$. By construction of the $\Phi^{\mathbf{\Lambda},\mathbf{A},i}_{\tau}$'s, these maps leave the sets
\begin{align*}
\mathcal{B}_{\R^d \times \{0\}^{D-d}} \left(x_k,\delta\right)
+
\mathcal{B}_{span(e)}(0, {\tau_{min}}/{2})
\end{align*}
unchanged for all $\mathbf{\Lambda},\mathbf{L}$. Therefore,  on the event $\left\{ \left(Z_1,Z_{2:n}\right) \in U_k \times U'_k\right\}$, one can write $Z_1$ only as a function of $X_1,\Lambda_k,A_k$, and $Z_{2:n}$ as a function of the rest of the $X_j$'s,$\Lambda_{k}$'s and $A_{k}$'s. Therefore, $Z_1$ and $Z_{2:n}$ are independent.

We now focus on the geometric statements. For this, we fix a deterministic point $z = \Phi^{\mathbf{\Lambda},\mathbf{A},(i)}_{\tau}(x_0) \in U_k \cap {M}^{\mathbf{\Lambda},\mathbf{A},(i)}_{\tau}$. By construction, one necessarily has $x_0 \in M_0 \cap \mathcal{B}(x_k,\delta/2)$.
\begin{itemize}
\item  If $\tau_k = 0$, locally around $x_0$, $\Phi^{\mathbf{\Lambda},\mathbf{A},(1)}_{\tau}$ is the translation of vector $\Lambda_k e$. Therefore, since $M_0$ satisfies $T_{x_0} M_0  = \R^d \times \left\{0\right\}^{D-d}$ and $II_{x_0}^{M_0} = 0$, we have
\begin{align*}
T_{z} {{M}^{\mathbf{\Lambda},\mathbf{A},(i)}_{\tau}}  = \R^d \times \left\{0\right\}^{D-d}
\text{\quad and \quad }
\norm{II_{z}^{{{M}^{\mathbf{\Lambda},\mathbf{A},(i)}_{\tau}} }
\circ
\pi_{ T_{z} {M}^{\mathbf{\Lambda},\mathbf{A},(i)}_{\tau}  }
}_{op} = 0
.
\end{align*}

\item if $\tau_k = 1$,
\begin{itemize}
\item for $i = 1$: locally around $x_0$, $\Phi^{\mathbf{\Lambda},\mathbf{A},(1)}_{\tau}$ can be written as $x \mapsto x + A_k (x-x_k)_1 e$. Hence, $T_z {M}^{\mathbf{\Lambda},\mathbf{A},(i)}_{\tau}$ contains the direction $(1,A_k)$ in the plane $span(e_1,e)$ spanned by the first vector of the canonical basis and $e$. As a consequence, since $e$ is orthogonal to $\R^d \times \left\{0\right\}^{D-d} $,
\begin{align*}
\angle
	\left(
		T_{z} {{M}^{\mathbf{\Lambda},\mathbf{A},(1)}_{\tau}} 
		, 
		\R^d \times \left\{0\right\}^{D-d} 
	\right)
	&\geq  
	\left( 1 + 1/A_k^2\right)^{-1/2}
	\geq A_k/2 \geq A_- /2
.
\end{align*}

\item for $i = 2$: locally around $x_0$, $\Phi^{\mathbf{\Lambda},\mathbf{A},(2)}_{\tau}$ can be written as $x \mapsto x + A_k (x-x_k)_1^2 e$. Hence, ${M}^{\mathbf{\Lambda},\mathbf{A},(2)}_{\tau}$ contains an arc of parabola of equation $y =A_k (x-x_k)_1^2$ in the plane $span(e_1,e)$. As a consequence,
\begin{align*}
\norm{
II_{z}^{{{M}^{\mathbf{\Lambda},\mathbf{A},(2)}_{\tau}} }
\circ
\pi_{ T_{z} {M}^{\mathbf{\Lambda},\mathbf{A},(2)}_{\tau}  }
}_{op}
&\geq 
A_k/2 
\geq A_-/2
.
\end{align*}
\end{itemize}
\end{itemize}
\end{proof}

\begin{Alem}\label{Alem:hypotheses_TV_estimates}
Assume that the conditions of Lemma \ref{differential_bounds_on_multiple_bumps} and Lemma \ref{Alem:hypotheses_properties} hold. If in addition, $c A_+ (\delta/4)^i \leq \Lambda_+ \leq C A_+ (\delta/4)^i$ for some absolute constants $C \geq c > 3/4$, and $A_- = A_+/2$, then,

\begin{align*}
\int_{U_k}
d \bar{Q}^{(i)}_{\tau,1} \wedge d \bar{Q}^{(i)}_{\tau^k,1} 
\geq 
\frac{c_{d,i}}{C}  \left(\frac{\delta}{\tau_{min}}\right)^d,
\end{align*}
and
\begin{align*}
\int_{U'_k}
d \bar{Q}^{(i)}_{\tau,n-1} \wedge d \bar{Q}^{(i)}_{\tau^k,n-1} 
= 
\left( 1 - c'_d  \left(\frac{\delta}{\tau_{min}}\right)^d \right)^{n-1}.
\end{align*}
\end{Alem}
\begin{proof}[Proof of Lemma \ref{Alem:hypotheses_TV_estimates}]
First note that all the involved distributions  have support in $\R^d \times span(e) \times \left\{ 0 \right\}^{D-(d+1)}$. Therefore, we use the canonical coordinate system of $\R^d \times span(e)$, centered at $x_k$, and we denote the components by $(x_1,x_2,\ldots,x_d,y) = (x_1,x_{2:d},y)$. Without loss of generality, assume that $\tau_k = 0$ (if not, flip $\tau$ and $\tau^k$).
Recall that $\phi$ has been chosen to be constant and equal to $1$ on the ball $\mathcal{B}(0,1/2)$.

By definition (\ref{eqn:mixture_definition}), on the event $\left\{Z \in U_k \right\}$, a random variable $Z$ having distribution $\bar{Q}^{(i)}_{\tau,1}$ can be represented by $Z = X + \phi\left(\frac{X-x_k}{\delta}\right)\Lambda_ke = X + \Lambda_ke$ where $X$ and $\Lambda_k$ are independent and have respective distributions $P_0$ (the uniform distribution on $M_0$) and the uniform distribution on $[-\Lambda_+,\Lambda_+]$. 
Therefore, on $U_k$, $\bar{Q}^{(i)}_{\tau,1}$ has a density with respect to the Lebesgue measure $\lambda_{d+1}$ on $\R^d \times span(e)$ that can be written as
\begin{align*}
\bar{q}^{(i)}_{\tau,1}(x_1,x_{2:d},y)
&=
\frac{\mathbbm{1}_{[-\Lambda_+,\Lambda_+]}(y)}{2 Vol(M_0) \Lambda_+}
.
\end{align*}
Analogously, nearby $x_k$ a random variable $Z$ having distribution $\bar{Q}^{(i)}_{\tau^k,1}$ can be represented by $Z = X + A_k(X-x_k)_1^i e$ where $A_k$ has uniform distribution on $[A_-,A_+]$. 
Therefore, a straightforward change of variable yields the density
\begin{align*}
\bar{q}^{(i)}_{\tau^k,1}(x_1,x_{2:d},y)
&=
\frac{\mathbbm{1}_{[A_-x_1^i,A_+x_1^i]}(y)}{Vol(M_0)\left(A_+ - A_-\right)x_1^i}
.
\end{align*}
We recall that $Vol(M_0) = (2\tau_{min})^d Vol\bigl( M_0^{(0)} \bigr) = c'_d \tau_{min}^d$.
Let us now tackle the right-hand side inequality, writing

\begin{align*}
\int_{U_k}
&
d \bar{Q}^{(i)}_{\tau,1} \wedge d \bar{Q}^{(i)}_{\tau^k,1} 
\\
&=
\int_{\mathcal{B}(x_k,\delta/2)}
\left(
\frac{\mathbbm{1}_{[-\Lambda_+,\Lambda_+]}(y)}{2 Vol(M_0) \Lambda_+}
\right)
\wedge 
\left(
\frac{\mathbbm{1}_{[A_-x_1^i,A_+x_1^i]}(y)}{Vol(M_0)\left(A_+ - A_-\right)x_1^i}
\right)
d y d x_1 d x_{2:d}
\\
&\geq
\int_{\mathcal{B}_{\R^{d-1}}(0,\frac{\delta}{4})}
\int_{-\delta/4}^{\delta/4}
\int_\R
\left(
\frac{\mathbbm{1}_{[-\Lambda_+,\Lambda_+]}(y)}{2\Lambda_+}
\right)
\wedge
\left(
\frac{\mathbbm{1}_{[A_-x_1^i,A_+x_1^i]}(y)}{A_+x_1^i/2}
\right)
\frac{d y
d x_1
d x_{2:d}}{Vol(M_0)}
.
\end{align*}
It follows that
\begin{align*}
\int_{U_k}
&
d \bar{Q}^{(i)}_{\tau,1} \wedge d \bar{Q}^{(i)}_{\tau^k,1} \\
&\geq
\frac{c_d}{\tau_{min}^d} \delta^{d-1}
\int_0^{\delta/4}
\int_{A_+x_1^i/2}^{\Lambda_+ \wedge (A_+ x_1^i)}
\frac{1}{2\Lambda_+}
\wedge
\frac{2}{A_+x_1^i}
dy d x_1
\\
&\geq
\frac{c_d}{\tau_{min}^d} \delta^{d-1}
\int_0^{\delta/4}
\int_{A_+x_1^i/2}^{(c \wedge 1) (A_+ x_1^i)}
\frac{(2c \wedge 1/2)}{2\Lambda_+}
dy d x_1
\\
&=
\frac{c_d}{\tau_{min}^d} \delta^{d-1} (2c \wedge 1/2)\left( c\wedge 1 - 1/2 \right) \frac{A_+}{\Lambda_+} \frac{\left(\delta/4\right)^{i+1}}{i+1}
\\
&\geq
\frac{c_{d,i}}{C} \left(\frac{\delta}{\tau_{min}}\right)^d.
\end{align*}
%
For the integral on $U'_k$, notice that by definition, $\bar{Q}^{(i)}_{\tau,n-1}$ and $\bar{Q}^{(i)}_{\tau^k,n-1}$ coincide on 
$U'_k$ since they are respectively the image distributions of $P_0$ by functions that are equal on that set. 
Moreover, these two functions leave $
\R^D 
\setminus 
\left\{
\mathcal{B}_{\R^d \times \{0\}^{D-d}} \left(x_k,\delta\right)
+
\mathcal{B}_{span(e)}(0, {\tau_{min}}/{2})
\right\}
$
unchanged. 
Therefore,
\begin{align*}
\int_{U'_k}
d \bar{Q}^{(i)}_{\tau,n-1} \wedge & d \bar{Q}^{(i)}_{\tau^k,n-1} 
\\
&= 
P_0^{\otimes n-1} \left( U'_k \right)
\\
&
= 
\left( 
1 
- 
P_0\left(
\mathcal{B}_{\R^d \times \{0\}^{D-d}} \left(x_k,\delta\right)
+
\mathcal{B}_{span(e)}(0, {\tau_{min}}/{2})
\right)\right)^{n-1}
\\
&= 
\left(1 - \omega_d \delta^d/Vol(M_0) \right)^{n-1}
,
\end{align*}
hence the result.
\end{proof}

\begin{proof}[Proof of Lemma 8]
The properties of $\bigl\{ \bar{Q}^{(i)}_{\tau,n} \bigr\}_\tau$ and $\left\{U_k \times U'_k\right\}_k$ given by Lemma \ref{Alem:hypotheses_properties} and Lemma \ref{Alem:hypotheses_TV_estimates} yield the result, setting $\Lambda_+ = A_+\delta^i/4$, $A_+ = 2 A_- = \varepsilon \delta^{k-i}$ for $\varepsilon = \varepsilon_{k,d,\tau_{min}}$, and $\delta$ such that $c'_d  \left(\frac{\delta}{\tau_{min}}\right)^d = \frac{1}{n-1}$.
\end{proof}

\subsubsection{Proof of Lemma 9}

This section details the construction leading to Lemma 9 that we restate in Lemma \ref{Alem:hypotheses_glissant_noise}.

\begin{Alem}
\label{Alem:hypotheses_glissant_noise}
Assume that $\tau_{min}{L_\perp}$,$\ldots$,$\tau_{min}^{k-1}{L_k}$,$({\tau_{min}^d}f_{min})^{-1}$, ${\tau_{min}^d}{f_{max}}$ are large enough (depending only on $d$ and $k$), and $\sigma \geq C_{k,d,\tau_{min}} \left({1}/({n-1)} \right)^{k/d}$ for $C_{k,d,\tau_{min}} > 0$ large enough.
Given $i \in \{1,2\}$, there exists a collection of $2^m$ distributions $\bigl\{\mathbf{P}_\tau^{(i),\sigma}\bigr\}_{\tau \in \{0,1\}^m} \subset \mathcal{P}^k(\sigma)$ with associated submanifolds $\bigl\{M_\tau^{(i),\sigma}\bigr\}_{\tau \in \{0,1\}^m}$, together with pairwise disjoint subsets $\{U^\sigma_k\}_{1 \leq k \leq m}$ of $\R^D$ such that the following holds for all $\tau \in \{0,1\}^m$ and $1 \leq k \leq m$.

If $x \in U_k^\sigma$ and $y = \pi_{M_\tau^{(i),\sigma}} (x)$, we have
\begin{itemize}
\item  if $\tau_k = 0$,
\begin{align*}
T_{y} M_\tau^{(i),\sigma}  = \R^d \times \left\{0\right\}^{D-d}
\text{\quad , \quad }
\norm{
II_{y}^{ M_\tau^{(i),\sigma} }
\circ
\pi_{T_{y} M_\tau^{(i),\sigma}}
}_{op} = 0,
\end{align*}
 
\item if $\tau_k = 1$,
\begin{itemize}
\item for $i = 1$:
$
\displaystyle
\angle 
\left(
T_{y} M_\tau^{(1),\sigma} 
, 
\R^d \times \left\{0\right\}^{D-d} \right) 
\geq c_{k,d,\tau_{min}} \left( \frac{\sigma}{n-1}\right)^{\frac{k-1}{k+d}}$,
\item for $i = 2$:
$
\displaystyle
\norm{
II_{y}^{M_\tau^{(2),\sigma} }
\circ
\pi_{
T_{y} M_\tau^{(2),\sigma} 
}
}_{op} 
\geq c'_{k,d,\tau_{min}} \left( \frac{\sigma}{n-1}\right)^{\frac{k-2}{k+d}}
$. 
\end{itemize}
\end{itemize}
Furthermore,
\begin{align*}
\int_{(\R^D)^{n-1}}
\bigl(\mathbf{P}_{\tau}^{(i),\sigma}\bigr)^{\otimes n-1} 
\wedge 
\bigl(\mathbf{P}_{\tau^k}^{(i),\sigma}\bigr)^{\otimes n-1} 
\geq
c_0,
\text{ and}
\quad
m \cdot \int_{U_k^\sigma}
\mathbf{P}_{\tau}^{(i),\sigma} \wedge \mathbf{P}_{\tau^k}^{(i),\sigma}
&
\geq
c_d.
\end{align*}
\end{Alem}

\begin{proof}[Proof of Lemma \ref{Alem:hypotheses_glissant_noise}]

Following the notation of Section \ref{subsec:generic_hypotheses}, for $i \in \left\{1,2\right\}$, $\tau \in \{0,1\}^m$, $\delta \leq \tau_{min}/4$ and $A>0$, consider
\begin{equation}\label{eq:multibump_definition_noise_glissant}
\Phi^{A,i}_{\tau}(x) = x + \sum_{k=1}^m \phi\left(\frac{x-x_k}{\delta}\right) \left\{ \tau_k  A (x-x_k)_1^i \right\} e.
\end{equation}
Note that \eqref{eq:multibump_definition_noise_glissant} is a particular case of \eqref{multibump_definition}.
Clearly from the definition, $\Phi^{A,i}_{\tau}$ and $\Phi^{A,i}_{\tau^k}$ coincide outside $\mathcal{B}(x_k,\delta)$, 
$(\Phi(x) - x) \in span(e)$ for all $x\in \R^D$, and $\norm{I_D - \Phi}_\infty \leq A\delta^i$.
Let us define $M_{\tau}^{A,i} = \Phi_{\tau}^{A,i}(M_0)$.
From Lemma \ref{Alem:generic_hypotheses_properties}, we have $M_{\tau}^{A,i} \in  \mathcal{C}^{k}_{\tau_{min},\mathbf{L}}$ provided that $\tau_{min}{L_\perp},\ldots,\tau_{min}^{k-1}{L_k}$ are large enough, and that $\delta \leq \tau_{min}/2$, with $A / \delta^{k-i} \leq \varepsilon$ for $\varepsilon =\varepsilon_{k,d,\tau_{min},i}$ small enough.

Furthermore, let us write
\begin{align*}
{U^\sigma_k} = 
\mathcal{B}_{\R^d \times \{0\}^{D-d}} \left(x_k,\delta/2\right)
+
\mathcal{B}_{\{0\}^{d}\times \R^{D-d}} \left(x_k,\sigma/2\right).
\end{align*}
Then the family $\{{U^\sigma_k}\}_{1\leq k \leq m}$ is pairwise disjoint.
Also, since $\tau_k = 0$ implies that $M_{\tau}^{A,i}$ coincides with $M_0$ on $\mathcal{B}(x_k,\delta)$, we get that if $x \in {U^\sigma_k}$ and $y = \pi_{M_{\tau}^{A,i}}(x)$,
\begin{align*}
T_{y} M_{\tau}^{A,i}  = \R^d \times \left\{0\right\}^{D-d}
\text{\quad , \quad }
\norm{
II_{y}^{M_{\tau}^{A,i}}
\circ
\pi_{T_{y} {{M_{\tau}^{A,i}}}}
}_{op} = 0.
\end{align*}
Furthermore, by construction of the bump function $\Phi_{\tau}^{A,i}$, if $x \in {U_k^\sigma}$ and $\tau_k = 1$, then
\[
\angle \left(T_{y}M_{\tau}^{A,i} , \R^d \times \left\{0\right\}^{D-d} \right) \geq \frac{A}{2},
\]
and
\[
\norm{II_{y}^{M_{\tau}^{A,i}}
\circ
\pi_{T_{y} M_{\tau}^{A,i}
 }}_{op} \geq \frac{A}{2}.
\]
Now, let us write 
\begin{align*}
\mathcal{O}_{\tau}^{A,i} &= \left\{ y + \xi \left| y \in M_{\tau}^{A,i}, \xi \in \left(T_y M_{\tau}^{A,i}\right)^\perp, \norm{\xi} \leq \sigma/2 \right.\right\}
\end{align*}
for the offset of $M_{\tau}^{\Lambda,A,i}$ of radius $\sigma/2$. 
The sets $\bigl\{\mathcal{O}_{\tau}^{A,i}\bigr\}_\tau$ are closed subsets of $\R^D$ with non-empty interiors.
Let $\mathbf{P}_{\tau}^{A,i}$ denote the uniform distribution on $\mathcal{O}_{\tau}^{A,i}$.
Finally, let us denote by $P_{\tau}^{A,i} = \bigl(\pi_{M_{\tau}^{A,i}}\bigr)_\ast \mathbf{P}_{\tau}^{A,i}$ the pushforward distributions of $\mathbf{P}_{\tau}^{A,i}$ by the projection maps $\pi_{M_{\tau}^{A,i}}$.
From Lemma 19 in \cite{Maggioni16}, $P_{\tau}^{A,i}$ has a density $f_{\tau}^{A,i}$ with respect to the volume measure on $M_{\tau}^{A,i}$, and this density satisfies
\begin{align*}
Vol\left(M_{\tau}^{A,i}\right) f_{\tau}^{A,i}
\leq
\left( \frac{\tau_{min}+\sigma/2}{\tau_{min}-\sigma/2} \right)^d
\leq
\left( \frac{5}{3}\right)^d,
\end{align*}
and
\begin{align*}
Vol\left(M_{\tau}^{A,i}\right) f_{\tau}^{A,i}
\geq
\left( \frac{\tau_{min}-\sigma/2}{\tau_{min}+\sigma/2} \right)^d 
\geq
\left( \frac{3}{5}\right)^d. 
\end{align*}
Since, by construction, $Vol(M_0) = c_d \tau_{min}^d$, and $c'_d \leq Vol\bigr( M_{\tau}^{\Lambda,A,i} \bigr)/Vol(M_0) \leq C'_d$ whenever $A/\delta^{i-1} \leq \varepsilon'_{d,\tau_{min},i}$, we get that $P_{\tau}^{A,i}$ belongs to the model $\mathcal{P}^k$ provided that $({\tau_{min}^d}f_{min})^{-1}$ and ${\tau_{min}^d}{f_{max}}$ are large enough. 
This proves that under these conditions, the family $\bigl\{\mathbf{P}_{\tau}^{A,i}\bigr\}_{\tau \in \{0,1\}^m}$ is included in the model $\mathcal{P}^k({\sigma})$.

Let us now focus on the bounds on the $L^1$ test affinities. 
Let $\tau \in \{0,1\}^m$ and $1 \leq k \leq m$ be fixed, and assume, without loss of generality, that $\tau_k = 0$ (if not, flip the role of $\tau$ and $\tau^k$).
First, note that 
\begin{align*}
\int_{(\R^D)^{n-1}}
\bigl(\mathbf{P}_{\tau}^{A,i}\bigr)^{\otimes n-1} 
\wedge 
\bigl(\mathbf{P}_{\tau^k}^{A,i}\bigr)^{\otimes n-1} 
\geq
\left(
\int_{\R^D}
\mathbf{P}_{\tau}^{A,i}
\wedge 
\mathbf{P}_{\tau^k}^{A,i}
\right)^{n-1}.
\end{align*}
Furthermore, since $\mathbf{P}_{\tau}^{A,i}$ and $\mathbf{P}_{\tau^k}^{A,i}$ are the uniform distributions on $\mathcal{O}_{\tau}^{A,i}$ and $\mathcal{O}_{\tau}^{A,i}$,
\begin{align*}
\int_{\R^D}
\mathbf{P}_{\tau}^{A,i} \wedge \mathbf{P}_{\tau^k}^{A,i}
&=
1-
\frac{1}{2}
\int_{\R^D}
\left|
\mathbf{P}_{\tau}^{A,i} 
- 
\mathbf{P}_{\tau^k}^{A,i}
\right|
\\
&=
1-
\frac{1}{2}
\int_{\R^D}
\left| 
\frac{
	\mathbbm{1}_{\mathcal{O}_{\tau}^{A,i}}(a)
	}
	{
	Vol\left(\mathcal{O}_{\tau}^{A,i}\right)
	} 
-
\frac{
	\mathbbm{1}_{\mathcal{O}_{\tau}^{A,i}}(a)
	}
	{
	Vol\left(\mathcal{O}_{\tau^k}^{A,i}\right)
	} 
\right| 
d\mathcal{H}^D(a).
\end{align*}
Furthermore,
\begin{align*}
\frac{1}{2}
\int_{\R^D}
\left| 
\frac{
	\mathbbm{1}_{\mathcal{O}_{\tau}^{A,i}}(a)
	}
	{
	Vol\left(\mathcal{O}_{\tau}^{A,i}\right)
	}
\right. 
&-
\left.
\frac{
	\mathbbm{1}_{\mathcal{O}_{\tau^k}^{A,i}}(a)
	}
	{
	Vol\left(\mathcal{O}_{\tau^k}^{A,i}\right)
	} 
\right| 
d\mathcal{H}^D(a)
\\
&=
\frac{1}{2} Vol\left( \mathcal{O}_{\tau}^{A,i} \cap \mathcal{O}_{\tau^k}^{A,i}\right)
\left|
\frac{
	1
	}
	{
	Vol\left(\mathcal{O}_{\tau}^{A,i}\right)
	} 
-
\frac{
	1
	}
	{
	Vol\left(\mathcal{O}_{\tau^k}^{A,i}\right)
	} 
\right|
\\
&\quad 
+
\frac{1}{2} 
\left( 
\frac{Vol \left( \mathcal{O}_{\tau}^{A,i} \setminus \mathcal{O}_{\tau^k}^{A,i} \right)}{Vol\left( \mathcal{O}_{\tau}^{A,i}\right)} 
+
\frac{Vol \left( \mathcal{O}_{\tau^k}^{A,i} \setminus \mathcal{O}_{\tau}^{A,i} \right)}{Vol\left( \mathcal{O}_{\tau^k}^{A,i} \right)}
\right)
\\
&\leq
\frac{3}{2}
\frac{
	Vol \left( \mathcal{O}_{\tau}^{A,i} \setminus \mathcal{O}_{\tau^k}^{A,i} \right)
	\vee
	Vol \left( \mathcal{O}_{\tau^k}^{A,i} \setminus \mathcal{O}_{\tau}^{A,i} \right)
	}
	{
	Vol\left(
\mathcal{O}_{\tau}^{A,i} 
\right)
\wedge
Vol\left(
\mathcal{O}_{\tau^k}^{A,i}
\right)	
	}.
\end{align*}
To get a lower bound on the denominator, note that for $\delta \leq \tau_{min}/2$, $M_{\tau}^{A,i}$ and $M_{\tau^k}^{A,i}$ both contain 
\begin{align*}
\mathcal{B}_{\R^d\times \{0\}^{D-d}}(0,\tau_{min})\setminus \mathcal{B}_{\R^d\times \{0\}^{D-d}}(0,\tau_{min}/4)
,
\end{align*}
so that $\mathcal{O}_{\tau}^{A,i}$ and $\mathcal{O}_{\tau^k}^{A,i}$ both contain
\begin{align*}
\left(
\mathcal{B}_{\R^d\times \{0\}^{D-d}}(0,\tau_{min})
\setminus
\mathcal{B}_{\R^d\times \{0\}^{D-d}}(0,\tau_{min}/4)
\right)
+
\mathcal{B}_{\{0\}^{d} \times \R^{D-d} }(0,\sigma/2).
\end{align*}
As a consequence, 
$
Vol\left(
\mathcal{O}_{\tau}^{A,i} 
\right)
\wedge
Vol\left(
\mathcal{O}_{\tau^k}^{A,i}
\right)
\geq
c_d \omega_d \tau_{min}^d \omega_{D-d} (\sigma/2)^{D-d},
$
where $\omega_\ell$ denote the volume of a $\ell$-dimensional unit Euclidean ball.

We now derive an upper bound on $Vol \bigl( \mathcal{O}_{\tau}^{A,i} \setminus \mathcal{O}_{\tau^k}^{A,i} \bigr)$.
To this aim, let us consider $a_0 = y + \xi \in  \mathcal{O}_{\tau}^{A,i} \setminus \mathcal{O}_{\tau^k}^{A,i}$, with $y \in  M_{\tau}^{A,i}$ and $\xi \in \bigl(T_y M^{A,i}_{\tau}
\bigr)^\perp$. 
Since $\Phi^{A,i}_{\tau}$ and $\Phi^{A,i}_{\tau^k}$ coincide outside $\mathcal{B}(x_k,\delta)$, so do $M^{A,i}_{\tau}$ and $M^{A,i}_{\tau^k}$. Hence, one necessarily has $y \in \mathcal{B}(x_k,\delta)$.
Thus, $\bigl(T_y M^{A,i}_{\tau}
\bigr)^\perp = T_y M_0^\perp = span(e) + \left\{0\right\}^{d+1}\times \R^{D-d-1}$, so we can write $\xi = se + z$ with $s \in \R$ and $z \in \left\{0\right\}^{d+1}\times \R^{D-d-1}$.
By definition of $\mathcal{O}_{\tau}^{A,i}$, $\norm{\xi} = \sqrt{ s^2 + \norm{z}^2} \leq \sigma/2$, which yields $\norm{z} \leq \sigma/2$ and $|s| \leq \sqrt{ (\sigma/2)^2 - \norm{z}^2}$.
Furthermore, $y_0$ does not belong to $ \mathcal{O}_{\tau^k}^{A,i}$, which translates to
\begin{align*}
\sigma/2
<
d\bigl(a_0, M_{\tau^k}^{A,i} \bigr)
&\leq
\norm{y_0 + se + z - \Phi_{\tau^k}^{A,i}(y_0)}
\\
&=
\sqrt{\left| s + \left\langle e, y_0 -  \Phi_{\tau^k}^{A,i}(y_0) \right\rangle \right|^2 + \norm{z}^2},
\end{align*}
from what we get $|s| \geq \sqrt{ (\sigma/2)^2 - \norm{z}^2} - \norm{I_D -  \Phi_{\tau^k}^{A,i}}_\infty$. 
We just proved that $\mathcal{O}_{\tau}^{A,i} \setminus \mathcal{O}_{\tau^k}^{A,i}$ is a subset of
\begin{multline*}
\mathcal{B}_{d}(x_k,\delta) 
+
\biggl\{ 
se + z
\bigg|
~(s,z) \in \R \times \R^{D-d-1}, 
\norm{z} \leq \sigma/2 
\text{ and} 
\\
\sqrt{ (\sigma/2)^2 - \norm{z}^2} - \norm{I_D - \Phi_{\tau^k}^{A,i}}_\infty
\leq
|s|
\leq 
\sqrt{ (\sigma/2)^2 - \norm{z}^2}
\biggr\}.
\end{multline*}
Hence, 
\begin{align}
\label{eq:bound_vol_1}
Vol\left(\mathcal{O}_{\tau}^{A,i} \setminus \mathcal{O}_{\tau^k}^{A,i}\right)
\leq
\omega_d \delta^d \times 2 \norm{I_D - \Phi_{\tau^k}^{A,i}}_\infty \times \omega_{D-d-1} (\sigma/2)^{D-d-1}.
\end{align}
Similar arguments lead to
\begin{align}
\label{eq:bound_vol_2}
Vol\left(\mathcal{O}_{\tau^k}^{A,i} \setminus \mathcal{O}_{\tau}^{A,i} \right)
\leq
\omega_d \delta^d \times 2 \norm{I_D - \Phi_{\tau}^{A,i}}_\infty \times \omega_{D-d-1} (\sigma/2)^{D-d-1}.
\end{align}
Since $\norm{I_D - \Phi_{\tau}^{A,i}}_\infty \vee \norm{I_D - \Phi_{\tau^k}^{A,i}}_\infty \leq A \delta^{i}$, summing up bounds \eqref{eq:bound_vol_1} and \eqref{eq:bound_vol_2} yields
\begin{align*}
\int_{\R^D}
\mathbf{P}_{\tau}^{A,i} \wedge \mathbf{P}_{\tau^k}^{A,i}
&
\geq
1-
3 \frac{
\omega_d \omega_{D-d-1} A \delta^{i} \cdot\delta^d (\sigma/2)^{D-d-1}
}
{
\omega_d \tau_{min}^d \omega_{D-d} (\sigma/2)^{D-d}
}
\\
&
\geq
1-
3
\frac{A\delta^{i}}{\sigma} \left(\frac{\delta}{\tau_{min}}\right)^d.
\end{align*}

To derive the last bound, we notice that since $U_k^\sigma \subset \mathcal{O}_{\tau}^{A,i} = Supp \bigl( \mathbf{P}_{\tau}^{A,i}\bigr ) $, we have
\begin{align*}
\int_{U_k^\sigma}
\mathbf{P}_{\tau}^{A,i} \wedge \mathbf{P}_{\tau^k}^{A,i}
&
\geq
\frac{
	Vol \left( U_k^\sigma \cap \mathcal{O}_{\tau^k}^{A,i} \right)
	}
	{
	Vol\left(
\mathcal{O}_{\tau}^{A,i} 
\right)
\wedge
Vol\left(
\mathcal{O}_{\tau^k}^{A,i}
\right)	
	}
\\
&\geq
\frac{
	Vol \left( U_k^\sigma \right)
	-
	Vol \left( U_k^\sigma \setminus \mathcal{O}_{\tau^k}^{A,i} \right)
	}
	{
	Vol\left(
\mathcal{O}_{\tau}^{A,i} 
\right)
\wedge
Vol\left(
\mathcal{O}_{\tau^k}^{A,i}
\right)	
	}
\\
&\geq
\frac{
	Vol \left( U_k^\sigma \right)
	-
	Vol \left( \mathcal{O}_{\tau}^{A,i} \setminus \mathcal{O}_{\tau^k}^{A,i} \right)
	}
	{
	Vol\left(
\mathcal{O}_{\tau}^{A,i} 
\right)
\wedge
Vol\left(
\mathcal{O}_{\tau^k}^{A,i}
\right)	
	}
\\
&
\geq
\frac{
\omega_d (\delta/2)^d \omega_{D-d} (\sigma/2)^{D-d}
-
\omega_d \delta^d A \delta^{i} \omega_{D-d-1}(\sigma/2)^{D-d-1}
}
{
\omega_d \tau_{min}^d \omega_{D-d} (\sigma/2)^{D-d}
}.
\end{align*}
Hence, whenever $A \delta^i \leq c_d \sigma$ for $c_d$ small enough, we get
\begin{align*}
\int_{U_k^\sigma}
\mathbf{P}_{\tau}^{A,i} \wedge \mathbf{P}_{\tau^k}^{A,i}
&
\geq
c'_d \left(\frac{\delta}{\tau_{min}}\right)^d.
\end{align*}
Since $m$ can be chosen such that $m \geq c_d (\tau_{min}/\delta)^{d}$, we get the last bound.

Eventually, writting $\mathbf{P}_\tau^{(i),\sigma} = \mathbf{P}_{\tau}^{A,i}$  for the particular parameters $A = \varepsilon \delta^{k-i}$, for $\varepsilon = \varepsilon_{k,d,\tau_{min}}$ small enough, and $\delta$ such that $\frac{3 A\delta^{i}}{\sigma} \left(\frac{\delta}{\tau_{min}}\right)^d =  \frac{1}{n-1}$ yields the result.
Such a choice of parameter $\delta$ does meet the condition $A \delta^i = \varepsilon \delta^k \leq c_d \sigma$, provided that
$
\sigma \geq \frac{c_d}{\varepsilon} \left(\frac{1}{n-1} \right)^{k/d}$.
\end{proof}

\subsection{Hypotheses for Manifold Estimation}\label{subsec:C4}

\subsubsection{Proof of Lemma 5}
Let us prove Lemma 5, stated here as Lemma \ref{Alem:hypotheses_hausdorff_nonoise}.
\begin{Alem}
\label{Alem:hypotheses_hausdorff_nonoise}
If $\tau_{min}{L_\perp},\ldots,\tau_{min}^{k-1}{L_k},({\tau_{min}^d}f_{min})^{-1}$ and ${\tau_{min}^d}{f_{max}}$ are large enough (depending only on $d$ and $k$), there exist $P_0,P_1 \in \mathcal{P}^k$ with associated submanifolds $M_0,M_1$ such that
\begin{align*}
d_H(M_0,M_1) 
\geq 
c_{k,d,\tau_{min}}
\left(\frac{1}{n}\right)^{\frac{k}{d}}
,
\text{ and}
\quad
\norm{P_0 \wedge P_1}_1^n \geq c_0.
\end{align*}
\end{Alem}

\begin{proof}[Proof of Lemma \ref{Alem:hypotheses_hausdorff_nonoise}]
Following the notation of Section \ref{subsec:generic_hypotheses}, for $\delta \leq \tau_{min}/4$ and $\Lambda>0$, consider
\begin{align*}
\Phi^{\Lambda}_{\tau}(x) = x + \phi\left(\frac{x}{\delta}\right) \Lambda \cdot e,
\end{align*}
which is a particular case of \eqref{multibump_definition}. Define $M^\Lambda = \Phi^\Lambda(M_0)$, and $P^\Lambda = \Phi^\Lambda_\ast P_0$. Under the conditions of Lemma \ref{Alem:generic_hypotheses_properties}, $P_0$ and $P^\Lambda$ belong to $\mathcal{P}^k$, and by construction, $d_H(M_0,M^\Lambda) = \Lambda$. In addition, since $P_0$ and $P^\Lambda$ coincide outside $\mathcal{B}(0,\delta)$,
\begin{align*}
\int_{\R^D} d P_0 \wedge d P^\Lambda
&=
P_0\bigl(\mathcal{B}(0,\delta)\bigr)
=
\omega_d \left( \frac{\delta}{\tau_{min}}\right)^d.
\end{align*}
Setting $P_1 = P^\Lambda$ with  $\omega_d \left( \frac{\delta}{\tau_{min}}\right)^d = \frac{1}{n}$ and $\Lambda = c_{k,d,\tau_{min}} \delta^k$ for $c_{k,d,\tau_{min}}>0$ small enough yields the result.
\end{proof}

\subsubsection{Proof of Lemma 6}

Here comes the proof of Lemma 6, stated here as Lemma \ref{Alem:hypotheses_hausdorff_nonoise}.

\begin{Alem}
\label{Alem:hypotheses_hausdorff_noise}
If $\tau_{min}{L_\perp},\ldots,\tau_{min}^{k-1}{L_k},({\tau_{min}^d}f_{min})^{-1}$ and ${\tau_{min}^d}{f_{max}}$ are large enough (depending only on $d$ and $k$), there exist $P_0^\sigma,P_1^\sigma \in \mathcal{P}^k(\sigma)$ with associated submanifolds $M_0^\sigma,M_1^\sigma$ such that
\begin{align*}
d_H(M_0^\sigma,M_1^\sigma) 
\geq 
c_{k,d,\tau_{min}}
\left(\frac{\sigma}{n}\right)^{\frac{k}{d+k}}
,
\text{ and}
\quad
\norm{P_0^\sigma \wedge P_1^\sigma}_1^n \geq c_0.
\end{align*}
\end{Alem}

\begin{proof}[Proof of Lemma \ref{Alem:hypotheses_hausdorff_noise}]
The proof follows the lines of that of Lemma \ref{Alem:hypotheses_glissant_noise}. Indeed, with the notation of Section \ref{subsec:generic_hypotheses}, for $\delta \leq \tau_{min}/4$ and $0< \Lambda \leq c_{k,d,\tau_{min}} \delta^k$ for $c_{k,d,\tau_{min}}>0$ small enough, consider
\begin{align*}
\Phi^{\Lambda}_{\tau}(x) = x + \phi\left(\frac{x}{\delta}\right) \Lambda \cdot e.
\end{align*}
Define $M^\Lambda = \Phi^\Lambda(M_0)$.
Write $\mathcal{O}_0$, $\mathcal{O}^{\Lambda}$ for the offsets of radii $\sigma/2$ of $M_0$, $M^\Lambda$, and 
and $\mathbf{P}_0, \mathbf{P}^\Lambda$ for the uniform distributions on these sets.

By construction, we have $d_H(M_0,M^\Lambda) = \Lambda$, and as in the proof of Lemma \ref{Alem:hypotheses_glissant_noise}, we get
\begin{align*}
\int_{\R^D}
\mathbf{P}_0 \wedge \mathbf{P}^{\Lambda}
&
\geq
1-
3
\frac{\Lambda}{\sigma} \left(\frac{\delta}{\tau_{min}}\right)^d.
\end{align*}
Denoting ${P}_0^\sigma = \mathbf{P}_0$ and ${P}_1^\sigma = \mathbf{P}^\Lambda$ with $\Lambda = \varepsilon_{k,d,\tau_{min}} \delta^k$ and $\delta$ such that $3
\frac{\Lambda}{\sigma} \left(\frac{\delta}{\tau_{min}}\right)^d$ yields the result.

\end{proof}

\subsection{Minimax Inconsistency Results}\label{subsec:C5}

This section is devoted to the proof of Theorem 1, reproduced here as Theorem \ref{Athm:minimax_non_consistency_tangent_and_curvature}.
\begin{Athm}\label{Athm:minimax_non_consistency_tangent_and_curvature}
Assume that $\tau_{min} = 0$.
If $D \geq d+3$, then, for all $k \geq 2$ and $L_\perp>0$, provided that $L_3/L_\perp^2,\ldots,{L_k}/L_\perp^{k-1},{L_\perp^d}/{f_{min}}$ and ${f_{max}}/{L_\perp^d}$ are large enough (depending only on $d$ and $k$), for all $n\geq 1$, 
\begin{align*}
\inf_{\hat{T}} 
\sup_{P \in \mathcal{P}_{(x)}^k}
\E_{P^{\otimes n}}
\angle
\bigl(
T_{x} M
,
\hat{T}
\bigr)
\geq \frac{1}{2} > 0,
\end{align*}
where the infimum is taken over all the estimators $\hat{T} = \hat{T}\bigl(X_1,\ldots,X_n\bigr)$. 

Moreover, for any $D \geq d+1$, provided that $L_3/L_\perp^2,\ldots,{L_k}/L_\perp^{k-1},{L_\perp^d}/{f_{min}}$ and ${f_{max}}/{L_\perp^d}$ are large enough (depending only on $d$ and $k$), for all $n\geq 1$, 
\begin{align*}
\inf_{\widehat{II}} 
\sup_{P \in \mathcal{P}_{(x)}^k}
\E_{P^{\otimes n}}
\norm{
II_{x}^M
\circ
\pi_{T_{x} M}
-
\widehat{II}
}_{op}
\geq \frac{L_\perp}{4} > 0,
\end{align*}
where the infimum is taken over all the estimators $\widehat{II} = \widehat{II}\bigl(X_1,\ldots,X_n\bigr)$.
\end{Athm}

We will make use of Le Cam's Lemma, which we recall here.

\begin{Athm}[Le Cam's Lemma \cite{Yu97}]\label{lecam}
For all pairs $P,P'$ in $\mathcal{P}$,
\begin{align*}
\inf_{ \hat{\theta}}
\sup_{P \in \mathcal{P}}
\E_{P^{\otimes n}}
d( \theta(P) , \hat{\theta} )
\geq
\frac{1}{2} 
d\left( \theta(P) , \theta(P') \right) 
\norm{P \wedge P'}_1^n
,
\end{align*}
where the infimum is taken over all the estimators $\hat{\theta} = \hat{\theta}(X_1,\ldots,X_n)$.

\end{Athm}

\begin{proof}[Proof of Theorem \ref{Athm:minimax_non_consistency_tangent_and_curvature}]

For $\delta \geq \Lambda > 0$, let $\mathcal{C},\mathcal{C}' \subset \R^3$ be closed curves of the Euclidean space as in Figure \ref{fig:non_consistency_two_hypotheses_tangent}, and such that outside the figure, $\mathcal{C}$ and $\mathcal{C}'$ coincide and are $\mathcal{C}^\infty$.
The bumped parts are obtained with a smooth diffeomorphism similar to \eqref{multibump_definition} and centered at $x$. 
Here, $\delta$ and $\Lambda$ can be chosen arbitrarily small.

\begin{figure}
\centering
\includegraphics[width= 0.7\textwidth]{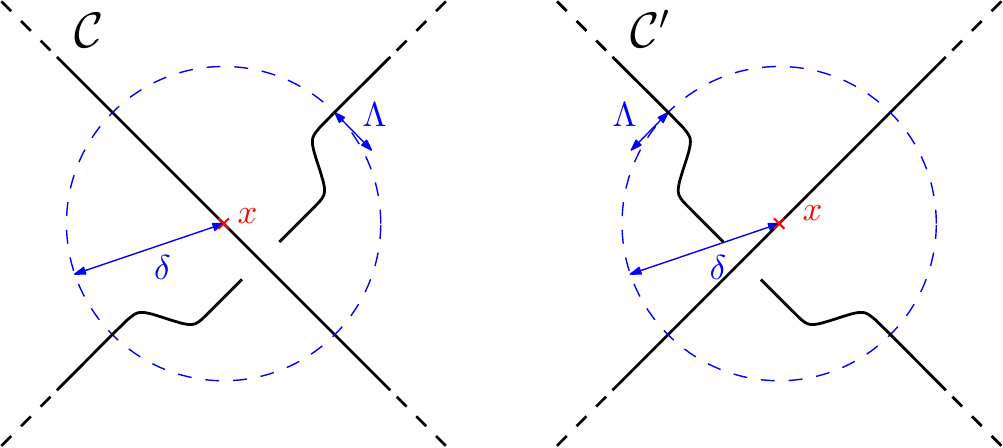}
\caption{Hypotheses for minimax lower bound on tangent space estimation with $\tau_{min} = 0$. }
\label{fig:non_consistency_two_hypotheses_tangent}
\end{figure}
Let $\mathcal{S}^{d-1} \subset \R^{d}$ be a $d-1$-sphere of radius $1/L_{\perp}$. Consider the Cartesian products $M_1 = \mathcal{C}\times \mathcal{S}^{d-1}$ and $M_1' = \mathcal{C}'\times \mathcal{S}^{d-1}$. $M_1$ and $M_1'$ are subsets of $\R^{d+3} \subset \R^D$. 
Finally, let $P_1$ and $P_1'$ denote the uniform distributions on $M$ and $M'$. 
Note that $M$, $M'$ can be built by homothecy of ratio $\lambda = 1/L_\perp$ from some unitary scaled $M_1^{(0)},{M'}_1^{(0)}$, similarly to Section 5.3.2 in \cite{AamariL17}, yielding, from Proposition \ref{statistical_model_stability}, that $P_1,P'_1$ belong to $\mathcal{P}^k_{(x)}$ provided that $L_3/L_\perp^2,\ldots,{L_k}/L_\perp^{k-1},{L_\perp^d}/{f_{min}}$ and ${f_{max}}/{L_\perp^d}$ are large enough (depending only on $d$ and $k$), and that $\Lambda,\delta$ and  $\Lambda^k/\delta$ are small enough. From Le Cam's Lemma \ref{lecam}, we have for all $n \geq 1$,
\begin{align*}
\inf_{\hat{T}} 
\sup_{P \in \mathcal{P}_{(x)}^k}
\E_{P^{\otimes n}}
\angle
\bigl(
T_{x} M
,
\hat{T}
\bigr)
&\geq
\frac{1}{2} 
\angle
\bigl(
T_{x} M_1
,
T_{x} M_1'
\bigr)
\norm{P_1 \wedge P_1'}_1^n.
\end{align*}
By construction, $\angle \bigl(T_{x} M_1,T_x M_1'\bigr) = 1$, and since $\mathcal{C}$ and $\mathcal{C}'$ coincide outside $\mathcal{B}_{\R^3}(0,\delta)$, 
\begin{align*}
\norm{P_1 \wedge P_1'}_1
&= 
1 - Vol\left(   \left( \mathcal{B}_{\R^3}(0,\delta) \cap \mathcal{C}\right) \times \mathcal{S}^{d-1} \right) 
/ Vol\left( \mathcal{C} \times \mathcal{S}^{d-1} \right)
\\
&=
1-
Length \left( \mathcal{B}_{\R^3}(0,\delta) \cap \mathcal{C}\right) /  Length(\mathcal{C}) 
\\
&\geq
1-
c_{L_\perp} \delta.
\end{align*} 
Hence, at fixed $n\geq 1$, letting $\Lambda,\delta$ go to $0$ with $\Lambda^k/\delta$ small enough, we get the announced bound.

We now tackle the lower bound on curvature estimation with the same strategy. Let $M_2,M_2' \subset \R^D$ be $d$-dimensional submanifolds as in Figure \ref{fig:non_consistency_two_hypotheses_curvature}: they both contain $x$, the part on the top of $M_2$ is a half $d$-sphere of radius $2/L_\perp$, the bottom part of $M'_2$ is a piece of a $d$-plane, and the bumped parts are obtained with a smooth diffeomorphism similar to \eqref{multibump_definition}, centered at $x$.
Outside $\mathcal{B}(x,\delta)$, $M_2$ and $M_2'$ coincide and connect smoothly the upper and lower parts.
\begin{figure}
\centering
\includegraphics[width= 0.7\textwidth]{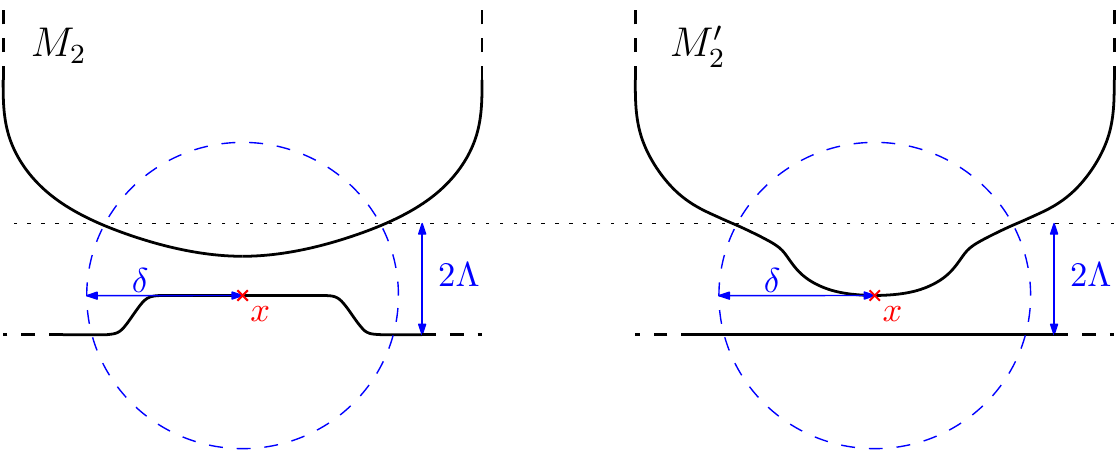}
\caption{Hypotheses for minimax lower bound on curvature estimation with $\tau_{min} = 0$. }
\label{fig:non_consistency_two_hypotheses_curvature}
\end{figure}
Let $P_2,P'_2$ be the probability distributions obtained by the pushforward given by the bump maps.
Under the same conditions on the parameters as previously, $P_2$ and $P_2'$ belong to $\mathcal{P}^{k}_{(x)}$ according to Proposition \ref{statistical_model_stability}. 
Hence from Le Cam's Lemma \ref{lecam} we deduce
\begin{align*}
\inf_{\widehat{II}} 
\sup_{P \in \mathcal{P}_{(x)}^k}
\E_{P^{\otimes n}}
&
\norm{
II_{x}^M
\circ
\pi_{T_x M}
-
\widehat{II}
}_{op}
\\
&\geq 
\frac{1}{2} 
\norm{
II_{x}^{M_2}
\circ
\pi_{T_{x} M_2}
-
II_{x}^{M_2'}
\circ
\pi_{T_{x} M_2'}
}_{op}
\norm{P_2 \wedge P_2'}_1^n.
\end{align*}
But by construction, $\norm{II_{x}^{M_2}
\circ
\pi_{T_{x} M_2}}_{op}= 0$,  and since $M_2'$ is a part of a sphere of radius $2/L_\perp$ nearby $x$, $\norm{II_{x}^{M_2'}
\circ
\pi_{T_{x} M_2'}}_{op} = L_\perp/2$. Hence,
\[
\norm{
II_{x}^{M_2}
\circ
\pi_{T_{x} M_2}
-
II_{x}^{M_2'}
\circ
\pi_{T_{x} M_2'}
}_{op}
\geq L_\perp/2
.
\]
Moreover, since $P_2$ and $P'_2$ coincide on $\R^D \setminus \mathcal{B}(x,\delta)$, 
\[
\norm{P_2 \wedge P'_2}_1 = 1 - P_2(\mathcal{B}(x,\delta)) \geq 1 - c_{d,L_\perp} \delta^d.
\] 
At $n\geq 1$ fixed, letting $\Lambda, \delta$ go to $0$ with $\Lambda^k/\delta$ small enough, we get the desired result.

\end{proof}
\end{document}